\documentclass[11pt]{amsart}

\usepackage{graphicx}
\usepackage{framed}

\usepackage{dsfont}
\usepackage{mathrsfs}

\usepackage[utf8]{inputenc}

\usepackage[all]{xy}
\usepackage{pb-diagram, pb-xy}

\usepackage{tikz}
\usepackage{pgffor}
\usetikzlibrary{arrows,backgrounds}
\usetikzlibrary{calc}

\usepackage{enumitem}
\setlist[itemize]{noitemsep}

\usepackage{amssymb}
\usepackage{amsthm}

\usepackage{hyperref}

\usepackage{calc}

\newtheorem{thm}{Theorem}[section]
\newtheorem{lem}[thm]{Lemma}
\newtheorem{cor}[thm]{Corollary}
\newtheorem{prop}[thm]{Proposition}

\newtheorem{conj}[thm]{Conjecture}
\newtheorem*{theorem*}{Theorem}
\newtheorem*{proposition*}{Proposition}
\newtheorem*{corollary*}{Corollary}
\newtheorem*{conjecture*}{Conjecture}

\theoremstyle{definition}
\newtheorem{definition}[thm]{Definition}
\newtheorem{example}[thm]{Example}
\newtheorem{remark}[thm]{Remark}

\def\Z{\mathbb{Z}}  

\def\1{\mathds{1}}

\newcommand{\frakg}{{\mathfrak g}}

\newcommand{\calC}{{\mathcal C}}

\newcommand{\calE}{{\mathcal E}}

\newcommand{\calT}{{\mathcal T}}

\newcommand{\one}{{\mathbf{1}}}

\DeclareMathOperator{\sdim}{sdim}

\usepackage{xypic,dsfont}

\hypersetup{
    colorlinks,
    citecolor=red,
    filecolor=black,
    linkcolor=blue,
    urlcolor=black
}

\begin{document}

\sloppy

\thispagestyle{empty}
\title[Homotopy quotients]{Homotopy quotients and comodules of supercommutative Hopf algebras}
\author{Thorsten Heidersdorf} 
\address{T.H.: Max-Planck Institut f\"ur Mathematik, Bonn}
\email{heidersdorf.thorsten@gmail.com}
\curraddr{T.H.: Institut f\"ur Mathematik, Universit\"at Bonn}
\author{Rainer Weissauer} 
\address{R.W.: Mathematisches Institut, Universit\"at Heidelberg}
\email{weissaue@mathi.uni-heidelberg.de}


\begin{abstract} We study model structures on the category of comodules of a supercommutative Hopf algebra $A$ over fields of characteristic 0. Given a graded Hopf algebra quotient $A \to B$ satisfying some finiteness conditions, the Frobenius tensor category $\mathcal{D}$ of graded $B$-comodules with its stable model structure induces a monoidal model structure on $\mathcal{C}$. We consider the corresponding homotopy quotient $\gamma: \mathcal{C} \to Ho \mathcal{C}$ and the induced quotient $\mathcal{T} \to Ho \mathcal{T}$ for the tensor category $\mathcal{T}$ of finite dimensional $A$-comodules. Under some mild conditions we prove vanishing and finiteness theorems for morphisms in $Ho \mathcal{T}$. We apply these results in the $Rep (GL(m|n))$-case and study its homotopy category $Ho \mathcal{T}$ associated to the parabolic subgroup of upper triangular block matrices. We construct cofibrant replacements and show that the quotient of $Ho\mathcal{T}$ by the negligible morphisms is again the representation category of a supergroup scheme. 

\end{abstract}

\thanks{2010 {\it Mathematics Subject Classification}: 16T15, 17B10, 18D10, 18E40, 18G55, 20G05, 55U35}

\maketitle

\vspace{-1.6cm}

\tableofcontents

\section{  Introduction}

\subsection{Model structures for representations of supergroups} \label{sec:intro-ov} 

The monoidal structures of representation categories of algebraic supergroups such as the general linear supergroup $GL(m|n)$ are poorly understood. In this article we study a monoidal model structure naturally appearing in the representation theory of  $GL(m|n)$. This model structure gives an abstract way to think about resolutions by Kac modules. The associated homotopy category (in the sense of Quillen) in turn is interesting since its monoidal structure can be seen as an approximation of that of $Rep(GL(m|n))$. This construction can be generalized to other affine supergroup schemes. More precisely we construct a model category for every pair $H \subseteq G$ where $G_0 \subseteq H \subseteq G$ and where $G$ is a quasireductive supergroup \cite{Serganova-quasireductive}. For the sake of generality we put the construction of our model category in an abstract setting. Then our construction can be seen as a way to construct model structures on an abelian Frobenius category $\mathcal{C}$, but before we wade into technical matters, let us consider the $GL(m|n)$-case first. 


\medskip\noindent
We work in this case over an algebraically closed field of characteristic 0. The supergroup $GL(m|n)$ contains the parabolic subgroup $P(m|n)^+$ of upper triangular block matrices. An irreducible representation of weight $\lambda$ of the even subgroup $G_0 = GL(m) \times GL(n)$ can be trivially extended to $P(m|n)^+$ and then induced to $GL(m|n)$. This parabolic induction  yields the so-called Kac modules \[ V(\lambda) = Ind_{P(m|n)^+}^{\ GL(m|n)} L_P(\lambda),\] the universal highest weight modules in $\mathcal{T}_{m|n} = Rep(GL(m|n))$. They are the standard modules in the highest weight category $\mathcal{T}_{m|n}$. The twisted dual $V(\lambda)^*$ (see section \ref{sec:gl-m-n-cofibrant}) (also called anti Kac module)  is then the corresponding costandard module. The full subcategory $\mathcal{T}_+$  of representations with a filtration by Kac modules (simply called Kac objects) and the full subcategory $\mathcal{T}_-$ of representations with a filtration by anti Kac modules (called anti Kac objects) are orthogonal in the sense that \[ Ext^1(\mathcal{T}_+, \mathcal{T}_-) =0\] Furthermore $\mathcal{T}_+ \cap \mathcal{T}_- = Proj$ (every tilting module is projective). 


\medskip\noindent
It can easily be shown that a module $M$ is in $\mathcal{T}_-$ if and only if its restriction to $P(m|n)^+$ is projective, i.e. zero in the stable category of $Rep(P(m|n)^+)$ (or equivalently injective, since projectives and injectives coincide in $Rep(P(m|n)^+)$ and also $Rep(GL(m|n))$). The stable category itself can be realized as the homotopy category of a model category. Roughly speaking a \emph{model category} is a category with three classes $\mathcal{L}$ (the \emph{cofibrations}), $\mathcal{R}$ (the \emph{fibrations}) and $\mathcal{W}$ (the \emph{weak equivalences}) of morphisms enjoying various lifting properties. Its homotopy category is then the localization by the class of weak equivalences. The categories $Rep(P(m|n)^+)$ and $Rep(GL(m|n))$ are related by the faithful restriction functor $Res$, its left adjoint $Coind$ and its right adjoint $Ind$. It can be easily seen that $Coind \cong Ind$ and hence both are exact. This situation remains true if we allow all algebraic  representations of $GL(m|n)$ (those corresponding to arbitrary comodules of $k[GL(m|n)]$), not just the finite dimensional ones. Indeed every algebraic representation is an inductive limit of finite dimensional ones, and hence the algebraic representations can simply be identified with the ind-category $Rep(GL(m|n))^{\infty}$. By the functorial construction of the ind-category, the functors $Res, Coind$ and $Ind$ extend with the same properties. 

\medskip\noindent
We are now in the following setting: Two abelian categories \[ \mathcal{C} = Rep(GL(m|n))^{\infty}, \\ \mathcal{D} = Rep(P(m|n)^+)^{\infty} \] which are related by two exact functors $U = Res$ and $F = Coind$ \[ Hom_{\mathcal{C}}(FX,Y) \cong Hom_{\mathcal{D}}(X,UY).\]

A standard construction in model category allows to transfer a model structure on a model category $\mathcal{D}$ to a model structure on a category $\mathcal{C}$ provided there is a Quillen adjunction between them: An adjunction $(F,U)$ such that F maps cofibrations to cofibrations and U maps fibrations to fibrations. In the case the model structure on $\mathcal{D}$ is cofibrantly generated, these conditions simplify, but in our specific situation we will have to sacrifice the finite dimensionality and pass to the ind categories in order to satisfy the requirements. 

\medskip\noindent
Hence the stable model structure on $Rep(P(m|n)^+)^{\infty}$ defines a model structure on $Rep(GL(m|n))^{\infty}$. Every model category comes with two distinguished classes of objects, the \emph{fibrant} objects, namely the objects $X$ such that $0 \to X$ (where $0$ is the initial object) is in $\mathcal{L}$, and the \emph{cofibrant} objects $X$ where $X \to *$ (where $*$ is the terminal object) is in $\mathcal{R}$. In our model category every object is fibrant, so we also consider the trivially fibrant objects $X$ where $X \to *$ is in $\mathcal{R} \cap \mathcal{W}$. The cofibrant objects define the full subcategory $\mathcal{C}_+$ and the trivially fibrant objects the full subcategory $\mathcal{C}_-$. Then $\mathcal{C}_-$ is the ind-category of the anti Kac objects $\mathcal{T}_-$ and $\mathcal{C}_+$ is the ind-category of the Kac objects $\mathcal{T}_+$. These two categories satisfy \[ Ext^1(\mathcal{C}_+,\mathcal{C}_-) = 0\] and form a cotorsion pair on $\mathcal{T}_{m|n}^{\infty}$.

\medskip\noindent
Every model category $\mathcal{C}$ can be localized by the weak equivalences $\mathcal{W}$. The localization is called the \emph{homotopy category} $Ho\mathcal{C}$ of $\mathcal{C}$. By definition of our model structure on $\mathcal{T}_{m|n}^{\infty}$, a morphism $f$ is in $\mathcal{W}$ if and only if $Res(f)$ is a weak equivalence for the stable model structure on $Rep(P(m|n)^+)^{\infty}$. It follows that the kernel of the functor $\mathcal{T}_{m|n}^{\infty} \to Ho\mathcal{T}_{m|n}^{\infty}$ to the homotopy category consists of the trivially fibrant objects $\mathcal{C}_-$, the inductive limits of objects in $\mathcal{T}_{m|n}$ with an anti Kac filtration. If we take instead the parabolic subgroup $P(m|n)^-$ of lower triangular block matrices for the definition of our model structure,  the roles of $\mathcal{C}_+$, $\mathcal{T}_+$, $\mathcal{C}_-$ and $\mathcal{T}_-$ switch.

\medskip\noindent
One of the most important features of a model category is that every object $X$ has a \emph{cofibrant replacement} $QX \to X$ with $QX \in \mathcal{C}_+$ and kernel in $\mathcal{C}_-$. While these cofibrant replacements may seem to be a bit abstract on first sight, note that in the $GL(m|n)$-case a cofibrant replacement, say of $L(\lambda)$, is given by an exact sequence \[ 0 \to A \to QL(\lambda) \to L(\lambda) \to 0 \] with $QL(\lambda) \in \mathcal{T}_+^{\infty}$ and $A \in \mathcal{T}_-^{\infty}$. The minimal cofibrant replacements (called the \emph{minimal model} of $L(\lambda)$), if they exist, have a rather special form. We show in lemma \ref{lem:existence-minimal-model} that atypical irreducible objects have a minimal model. To capture this we define a canonical degree filtration \ref{sec:degree-filtration} (analog to the weight filtration in algebraic geometry) for each object with a Kac filtration. By lemma \ref{degree-filtration} each object $M \in \mathcal{T}_+$ has a filtration by submodules $F_i(M) \in \calT_+$ such that \[ \ldots \subseteq F_{i-1}(M) \subseteq F_i(M) \subseteq F_{i+1}(M) \subseteq \ldots \] and \[ F_i(M)/ F_{i-1}(M) = \bigoplus_{\lambda} V(\lambda) \] holds for certain Kac modules $V(\lambda) \in \calT_+$ of degree $deg(\lambda) = i$. This degree filtration extends to $\mathcal{C}_+$ and in particular to the minimal model of $L(\lambda)$. The minimal model lies in the exact subcategory $\mathcal{C}_+^{pol}$: objects $M$ with a degree filtration $F$ such that $F_k(M) = 0$ for some $k \in \mathbb{N}$ and $\dim gr_i^F(M) < C \cdot P(i)$ for all $i$ where $C = C(M)$ is a constant and $P = P(M)$ a polynomial. To a Kac module $V(\lambda)$ we assign the power series \[ q^{-deg(\lambda)} [V(\lambda)] \in K_0(\mathcal{T})[[q^{-1}]] \] and extend this to sequential inductive limits of Kac-modules of polynomial growth via the degree filtration. Likewise we can associate a power series to any object in $\mathcal{C}_-$ of polynomial growth. To link this to irreducible representations, we assign to $L(\lambda)$ the power series $q^{-deg(\lambda)} [L(\lambda)]$. This extends to the exact subcategory $\mathcal{C}^{pol} \subset \mathcal{C}$ of inductive limits of polynomial growth of finite dimensional modules. Under the ring homomorphism $K_0(\mathcal{C}_+^{pol})$ to $K_0(\mathcal{C}^{pol})$ given by \[ q^{-deg(\lambda)} [V(\lambda)] \mapsto \sum_L  q^{-deg(L)} [L]\] where $L$ runs over the irreducible constituents of $V(\lambda)$, the minimal model \[0 \to A \to \Omega L(\lambda) \to  L(\lambda) \to 0 \] 
gives via identifications in the power series ring the formula
\[ [L(\lambda)] = [\Omega L(\lambda)] - [A].\] Since the class of a Kac object $V$ is the same as the one of the anti Kac object $V^*$ for the twisted dual $()^*$ on $\mathcal{T}_{m|n}$, this can be seeing as analogous to the resolutions of $L(\lambda)$ by Kac objects used for example by Serganova \cite{Serganova-character}. 

\medskip\noindent
A particular important feature of the $GL(m|n)$-case is the existence of a monoidal structure on $Ho\mathcal{C}$ such that $\mathcal{C} \to Ho\mathcal{C}$ is a tensor functor. We discuss this in more detail in section \ref{sec:intro-gl}.


\subsection{Induced model structures coming from Frobenius pairs} We axiomatize the setting of the $GL(m|n)$-case now in the hope that the abstract mechanism to operate with Kac modules of $\mathcal{T}_{m|n}$ might be useful in other contexts. The representation categories are replaced by two abelian Frobenius categories $\mathcal{C}, \mathcal{D}$ \begin{align*}  Rep(GL(m|n))^{\infty} & \rightsquigarrow \mathcal{C} \\ Rep(P(m|n)^+)^{\infty} & \rightsquigarrow \mathcal{D} \end{align*} and the functors $Res$ and $Coind$ by two functors $U$ and $F$ \begin{align*} Res & \rightsquigarrow U \\ Coind & \rightsquigarrow F.\end{align*} We are assume now that we are given the following data - called a \emph{Frobenius pair} ($\mathcal{C},\mathcal{D}$) - as in section \ref{sec:induced-1}:
\begin{enumerate}
\item two abelian Frobenius categories $\mathcal{C}, \ \mathcal{D}$ such that $\mathcal{D}$ satisfies the additional conditions \ref{stable-f-1} - \ref{stable-f-4} (described in section \ref{sec:stable-cat}), and 
\item an adjoint pair of functors $U,F$ between them satisfying \[ Hom_{\mathcal{C}}(FX,Y) \cong Hom_{\mathcal{D}}(X,UY)\] such that $U$ and $F$ are exact and $U$ is faithful.
\end{enumerate}

\medskip\noindent
Under these conditions the cofibrantly generated stable model structure on $\mathcal{D}$ induces a model structure on $\mathcal{C}$. More precisely the model structure $({\mathcal L},{\mathcal R},{\mathcal W})$ on $\mathcal C$ satisfies $f\in {\mathcal W}$ if and only if $U(f)\in \mathcal W_D$ where the latter is defined via stable equivalence.

\medskip\noindent
If we denote by $\mathcal{C}_+$ the cofibrant objects and by $\mathcal{C}_-$ the trivially fibrant objects in this model structure, this defines a cotorsion pair in the sense of \cite{Beligiannis-Reiten} \cite{Hovey-cotorsion-1} on $\mathcal{C}$, and so in particular \[ Ext^1(\mathcal{C}_+, \mathcal{C}_-) =0.\] We stress that we do not obtain just a single model structure on $\mathcal{C}$ in this way. Indeed any such pair $(\mathcal{C},\mathcal{D})$ will give rise to a different cotorsion pair, a different model structure and a different homotopy category. 

\begin{theorem*} (Theorem \ref{thm:homotopy-stable}) The homotopy category $Ho \mathcal{C}$ of the model category $\mathcal{C}$ is equivalent as a triangulated category to the stable category $\mathcal{C}_+/P_{\mathcal{C}}$.
\end{theorem*}

\subsection{Comodule categories} We apply this construction to the case where $\mathcal{C}$ is the category of (graded) comodules of a supercommutative Hopf algebra $A$ over a field $k$ (with $char(k) \neq 2$). Then $\mathcal{C}$ is the ind-category of the category $\mathcal{T}$ of finite-dimensional comodules, and it is a Frobenius category if $\mathcal{T}$ is one. If $A \to B$ is a quotient of supercommutative Hopf algebras, then their comodule categories $\mathcal{C}$ and $\mathcal{D}$ are related by induction and restriction functors. These need not satisfy the strong conditions imposed above. We call a pair $(A,B)$ or the corresponding pair of affine supergroup schemes $(H,G)$ where $H \subset G$, a Frobenius pair if the comodule categories $\mathcal{C} = Comod(A)$ and $\mathcal{D} = Comod(B)$ are Frobenius categories and $G_0 \subset H \subset G$ where $G_0$ is the underlying even algebraic group of $G$. In particular for any Frobenius pair the functor $Res: Rep(G)^{\infty} \to Rep(H)^{\infty}$ has a left and right adjoint which are isomorphic. In this situation we obtain a model structure on $\mathcal{C}$ as in section \ref{sec:induced}. 


\subsection{Monoidal model structures} So far we have not used any monoidal properties of our categories. Since the comodule category $\mathcal{C}$ is a tensor category, we want of course that $Ho\mathcal{C}$ is again a tensor category such that the localization functor $\mathcal{C} \to Ho\mathcal{C}$ is a tensor functor. This requires that the model structure on $\mathcal{C}$ is compatible with the usual tensor product on $\mathcal{C}$ and $\mathcal{C}$ carries a monoidal model structure.

\begin{theorem*} (Theorem \ref{thm:monoidal}) $\mathcal{C}$ is a monoidal model category and the functor $\gamma: \mathcal{C} \to Ho\mathcal{C}$ is a tensor functor.
\end{theorem*}

We may pass from a supercommutative Hopf algebra to the associated affine supergroup scheme $G$. Then the category $\mathcal{T}$ of finite dimensional comodules is equivalent as a tensor category to the finite-dimensional algebraic representations of $G$. If $G$ is an algebraic supergroup and has reductive even part (e.g. $G$ is a basic supergroup such as $GL(m|n)$ or $OSp(m|2n))$), then the algebraic representations $\mathcal{T}$ are a Frobenius category; and so our construction yields a cotorsion pair and a model structure on $\mathcal{C} \simeq Rep(G)^{\infty}$ for any embedded subgroup $G_0 \subset H \subset G$ with reductive even part $H_0$. The kernel of $\gamma: \mathcal{C} \to Ho \mathcal{C}$ are then simply the representations which restrict to a projective representation on the subgroup. While the construction of the model structure works in this generality, one needs to choose $H$ carefully to get an interesting theory similar to the $GL(m|n)$-case.

\subsection{Categories with weight truncations} Now let $\mathcal T$ be the category of finite dimensional representations of a
(connected) algebraic supergroup $G$ with reductive $G_0$ over an algebraically closed field
$k$ of characteristic $char(k)=0$ and let $\mathcal C$ be its ind-category. There is a notion of weights 
with an ordering $\leq $ defined between weights. Then we can define ${\mathcal T}^{\leq w}$ and ${\mathcal C}^{\leq w}$ to be the full subcategories
of objects whose simple subquotients $L(\lambda)$ all satisfy $\lambda \leq w$. The tuple $(\mathcal C$, $\mathcal D$, $U$, $F$,$\leq$) forms a Quillen adjunction with weight truncation as in definition \ref{def:weights}.

\begin{theorem*} (Theorem \ref{thm:cofibrant})  Every object in ${\mathcal C}^{\leq w}$ has a cofibrant replacement $Z$ where $Z$ is a direct sum of an injective object and an object in ${\mathcal C}^{\leq w}$.
\end{theorem*}

This theorem is important because it gives of some control on the morphism spaces in $Ho\mathcal{C}$ since \[ Hom_{Ho\mathcal{C}}(X,Y) \simeq Hom_{\mathcal{C}}(QX,Y)/\sim \]

by theorem \ref{thm:homotopy-stable}. As an application we deduce the following important vanishing theorem:

\begin{theorem*} (Theorem \ref{thm:main}) If  $(\mathcal C$, $\mathcal D$, $U$, $F$,$\leq$) forms a Quillen adjunction with weight truncation  $$ Hom_{Ho\mathcal C}(QL(\mu),L(\lambda))=0 \ \ \mbox{ hence } \ \ [L(\mu),L(\lambda)]=0$$  for any irreducible objects $L(\lambda),  \ L(\mu)$ for which $\mu<\lambda$.
\end{theorem*}

Under some additional natural conditions \ref{A-1} - \ref{A-4} formulated in section \ref{sec:finiteness-theorems} we deduce from this 

\begin{corollary*} (Theorem \ref{thm:end(1)}) If assumptions \ref{A-1} - \ref{A-4} hold, $[L(\lambda), L(\lambda)] = k \cdot id_{L(\lambda)}$. 
\end{corollary*}

While some of these theorems hold in greater generality, the important case for us occurs if $\mathcal{C} = Rep(G)^{\infty}$ where $G$ is an algebraic supergroup with reductive even part $G_0$, e.g. \[ G = GL(m|n), \ OSp(m|2n), \ P(n), \ Q(n) \] or one of the exceptional simple supergroups. Any subgroup $H$ with $G_0 \subset H \subset G$ defines a model structure on $Rep(G)^{\infty}$ (e.g. the upper triangular block matrices in $G$ or a maximal parabolic containing $G_0$), but it is not even clear in the $OSp(m|2n)$-case what the appropriate analogue of $P(m|n)^+ \subset GL(m|n)$ should be. Put ${\mathcal T}=Rep_k(G)$ and ${\mathcal T}_H=Rep_k(H)$ (or a related tensor category 
$Rep_k(\mu,G)$ etc.). 
Attached to the pair $(H,G)$ we consider the ind categories $\mathcal C$ of $\mathcal T$ and $\mathcal D$ of ${\mathcal T}_H$. We consider the following chain of functors
$$ \gamma: {\mathcal C} \to Ho{\mathcal C} = {\mathcal C_+}/\sim_{stable} $$
Let ${\mathcal H}=Ho\mathcal T$ be the full triangulated tensor subcategory of $Ho\mathcal C$
generated by the image of $\mathcal T$ under $\gamma$.
Then there is the functor
$$ \gamma: {\mathcal T \to \mathcal H }=Ho\mathcal T \ $$
which in general is neither surjective nor injective on the set of morphisms. 

\begin{conjecture*}
(Conjecture \ref{thm:hom-finite}) If assumptions \ref{A-1} - \ref{A-4} hold and $\mathcal{C} = Rep(G)^{\infty}$ for an algebraic supergroup with reductive even part,, $\dim [X,Y] < \infty$ for any $X,Y \in \mathcal{T}$.
\end{conjecture*}

This conjecture is a theorem in the $GL(m|n)$-case.

\subsection{The $GL(m|n)$-case} \label{sec:intro-gl} In part II of our paper we study the aforementioned case of the Frobenius pair $(P(m|n)^+,GL(m|n))$. Our main technical result theorem \ref{thm:good-replacement} is that for any object $X$ in $\mathcal T$ there exists 
a cofibrant replacement $q_X:QX \to X$ with particularly nice properties.

\begin{theorem*} (Theorem \ref{thm:good-replacement}) For any object $X$ in $\mathcal T$ there exists 
a cofibrant replacement $q_X:QX \to X$ in $\mathcal C$ with the following property:
For any $Y$ in $\mathcal T$ there exists a subobject $K' \in QX$ of finite codimension
contained
in $Kern(q_X)$ such that $K'\in \mathcal C_-$ and such that  $Hom_{\mathcal C}(K',Y) =0$.
\end{theorem*}

Of course this is proven by giving an explicit construction of a cofibrant replacement $q: \Omega \to \one$ for the trivial representation in lemma \ref{thm:cofib-of-1} and then forming $QX = \Omega \otimes X$. An immediate consequence is the finite dimensionality of $[X,Y]$ for any $X,Y \in \mathcal{T}$.  

\medskip\noindent
We already explained in the beginning of the introduction that these cofibrant replacements should be seen as an abstract version of resolutions of objects by Kac objects. Similarly the dimension of the Hom space $[L(\lambda),L(\mu)]$ has a more concrete interpretation in the $GL(m|n)$-case.  Note that \[ [L(\lambda),L(\mu)] = Hom_{\mathcal C_+}(QL(\lambda),L(\mu)) = Hom_{\mathcal{C}}(QL(\lambda),L(\mu)) \] if $QL(\lambda)$ is clean by corollary \ref{thm:hom-formula}. Using that $QL(\lambda) \in Ind(\mathcal T_+)$ and that \[ Hom_{\mathcal{C}}(V(\lambda),V(\mu)^*) = \delta_{\lambda\mu}k \] by \cite[Proposition 3.6.2]{Germonie}, one  proves \[ \dim Hom_{\mathcal C}(QL(\lambda),V(\mu)^*) = [QL(\lambda):V(\mu)],\] the latter being the multiplicity of $V(\mu)$ in $QL(\lambda)$. Since every morphism of $QL(\lambda) \to L(\mu)$ extends via $L(\mu) \hookrightarrow V(\mu)^*$ to a morphism to $V(\mu)^*$, we obtain \[ \dim [L(\lambda),L(\mu)] \ \leq \ [QL(\lambda):V(\mu)].\] We do not know any direct representation theoretic interpretation of $\dim [L(\lambda),L(\mu)]$.


\subsection{A second interpretation of $Ho\mathcal{T}$} The construction of well-behaved cofibrant replacements of objects in $\mathcal{T} = \mathcal{T}_{m|n}$ also enables us to give a different interpretation of $Ho \mathcal{T}$:

\begin{theorem*} (Theorem \ref{thm:gl-m-n-homotopy}) $Ho \mathcal{T}$ is equivalent as a tensor category to the Verdier quotient of the stable category $\overline{\mathcal{T}}$ by the thick tensor ideal $\mathcal{T}_-$ of anti Kac modules.
\end{theorem*}

This theorem should not be read as the statement that we should view $Ho\mathcal{T}$ simply as that Verdier quotient. Instead both interpretations have their advantages. While theorem \ref{thm:gl-m-n-homotopy} gives a more concrete description of $Ho\mathcal{T}$, the important cofibrant replacements (i.e. the infinite resolutions by Kac objects) live naturally in the model theoretic interpretation.

\medskip\noindent
The category $Ho{\mathcal T}$ is a $k$-linear rigid symmetric monoidal category with
$End_{Ho{\mathcal T}}(\one)=k$, and the tensor ideal of negligible morphisms
$\mathcal N$ (the largest proper tensor ideal of $Ho\mathcal{T}$) is defined (see section \ref{sec:semisimple}).

\begin{theorem*} (Theorem \ref{thm:semisimple}) For $GL(m|n)$ the quotient $Ho \mathcal T / \mathcal N$ is the semisimple representation category of an affine supergroup scheme.
\end{theorem*}

Parts of our motivation to study $Ho\mathcal{T}$ comes from our search of understanding the complicated monoidal structure of $\mathcal{T}_{m|n}$. Recent results in this area include the classification of thick ideals of $\mathcal{T}_{m|n}$ in \cite{BKN-spectrum}, a semisimplicity theorem about  the Duflo-Serganova functor \cite{Heidersdorf-Weissauer-tensor}, a structural computation of tensor products up to superdimension $0$  \cite{Heidersdorf-semisimple} \cite{Heidersdorf-Weissauer-tannaka} and explicit tensor product decompositions for special classes of representations \cite{Heidersdorf-mixed-tensors}. One of the interesting aspects of the homotopy category is that it is the natural habitat of the Duflo-Serganova cohomology functor $DS: \mathcal{T}_{m|n} \to \mathcal{T}_{m-1|n-1}$ and its extension $DS: \mathcal{C}_{m|n} \to \mathcal{C}_{m-1|n-1}$ to the ind completion. Indeed the $DS$ functor factorizes over the homotopy category and induces a tensor functor (see section \ref{sec:DS})  \begin{align*} DS:& Ho {\mathcal C}_{m|n} \to Ho {\mathcal C}_{m-1|n-1}\\ DS:& Ho {\mathcal T}_{m|n} \to Ho {\mathcal T}_{m-1|n-1}\end{align*} which might allow us to study the homotopy categories or their semisimple quotients inductively. Despite all these results the overall understanding of the tensor category $\mathcal{T}$ is rather poor. Passing to $Ho\mathcal{T}$ simplifies the monoidal structure at the prize of complicating the category, as the following special case shows.


\subsection{The $GL(m|1)$-case} In the final sections we look at the $GL(m|1)$-case. We compute morphism spaces between irreducible objects and use this to determine the indecomposable objects in $Ho \mathcal{T}_{m|1}$. In the homotopy category $Ho \mathcal{T}_{m|1}$ every indecomposable representation of non-vanishing superdimension becomes isomorphic to an irreducible representation. Hence the irreducible representations in $Ho{\mathcal T}/\mathcal{N}$ are parametrized by the atypical weights. Therefore $Ho \mathcal{T}/ \mathcal{N}$ agrees with the quotient $\mathcal{I}_{m|1}/\mathcal{N}$ where $\mathcal{I}_{m|1}$ is the full tensor subcategory of $Rep(GL(m|1))$ of direct summands in iterated tensor products of irreducible representations. The latter has been determined in \cite{Heidersdorf-semisimple} and we obtain  

\begin{proposition*} (Proposition \ref{thm:gl-m-1-group}) $Ho{\mathcal T}/\mathcal{N}$ is monoidal equivalent to the super representations of $GL(m|1) \times GL(1)$. 
\end{proposition*}

The entire quotient $\mathcal{T}_{m|1}/\mathcal{N}$ is tensor equivalent to the super representations of $(GL(m-1) \times GL(1) \times GL(1))$. Hence passing to $Ho \mathcal{T}$ simplifies the complete quotient $\mathcal{T}_{m|1}/\mathcal{N}$, but retains all the information about the tensor products of irreducible representations. This might generalize to the $GL(m|n)$-case. In fact let $\mathcal{I}_{m|n}$ denote the full tensor subcategory of $\mathcal{T}_{m|n}$ of direct summands in iterated tensor products of irreducible representations and consider the full triangulated subcategory $Ho \mathcal{I}_{m|n}$ in $Ho \mathcal{T}_{m|n}$ generated by the image of $\mathcal{I}_{m|n}$. Then it seems plausible that \[ \mathcal{I}_{m|n}/ \mathcal{N} \simeq Ho \mathcal{I}_{m|n}/ \mathcal{N}. \] In particular each indecomposable object of non-vanishing superdimension in $\mathcal{I}_{m|n}$ would stay indecomposable in $Ho \mathcal{C}$. 

\subsection{Outlook} Apart from the last question the constructions and results in this article offer a number of interesting questions. It would be interesting to study such model structues for other basic supergroups such as $OSp(m|2n)$. Even in this case it is unclear what the appropriate replacements of $P$ and the Kac-modules would be. We discuss some further conjectures in Section \ref{sec:further-conj}. In particular we wonder whether for basic classical $G$ and $H$ satisfying $G_0 \subset H \subset G$ the quotient $Ho \mathcal T / \mathcal N$ is still the semisimple representation category of an affine supergroup scheme as in the $GL(m|n)$-case.

It would be interesting to understand the homotopy category geometrically. One possibility would be to understand the Balmer spectrum $Spc(Ho \mathcal T)$ \cite{Balmer}.

Maybe the most interesting question is whether a modified construction also defines a model structure (and hence a cotorsion pair) on the ind-category of $Rep(G)$ if $H \subset G$ are finite groups (considered as algebraic groups) over a field of characteristic $p \neq 2$ such that $|H|$ is not prime to $p$. We have not explored this further. The present construction uses some properties that are not satisfied in positive characteristic.

\subsection{Acknowledgements} The authors thank the referee of an earlier version for helpful remarks. Version 2 differs in minor ways from Version 1. In particular we added Lemma \ref{inj-lemma} and assumed $G_0$ pro-reductive in Section \ref{sec:frobenius}. Most importantly we assume characteristic 0 throughout. The work of Thorsten Heidersdorf was partially funded by the Deutsche Forschungsgemeinschaft (DFG, German Research Foundation) under Germany's Excellence Strategy - EXC-2047/1 - 390685813. 



\part{Induced model structures on categories of comodules}


\section{  Background on model categories}

\subsection{Model structures} For categories we assume that the morphisms between two objects form a set. A category is small if its objects define a set.  A category $\mathcal C$ has all small 
limits and colimits, if limits and colimits exist for all functors from small categories
to $\mathcal C$. A category $\mathcal C$ with small limits and colimits is a 
model category in the sense of \cite{Hovey} if $\mathcal C$ has a model structure.
A model structure consists of classes $\mathcal W, \mathcal{L}, \mathcal{R}$ of morphisms
(weak equivalences, cofibrations, fibrations)
such that if two of three morphisms $f,g,f\circ g$ are in $\mathcal W$ also the third is in $\mathcal W$.

We require the following three axioms \cite[Definition 1.1.3]{Hovey}: the retract axiom, the lifting axiom and the factorization axiom. 
The {\it lifting axiom} postulates the existence of a lifting $h$ for commutative diagrams 
$$ \xymatrix@+7mm{ A \ar[d]^i\ar[r]^f & C \ar[d]^p \cr
B \ar@{.>}[ru]^h\ar[r]^g & D \cr} $$
where $i \in {\mathcal W}\cap {\mathcal L}$  and $p\in \mathcal R$,
or where $i \in {\mathcal L}$  and $p \in {\mathcal W}\cap {\mathcal R}$ (trivial fibrations).
The {\it factorization axiom} states that every morphism $f$ can be (functorially) written in the form 
$f=\psi\circ \varphi$ for certain $\varphi \in {\mathcal W}\cap {\mathcal L}$ and $\psi\in \mathcal R$, and also
for certain $\varphi \in {\mathcal L}$  and $\psi \in {\mathcal W}\cap {\mathcal R}$.
The {\it retract axiom} states 
that $\mathcal L$, $\mathcal R$ or $\mathcal W$ are stable under retracts.


\subsection{Morphisms, fibrations and cofibrations}

We recall the following definitions \cite[Definition 2.1.7]{Hovey}:

\begin{definition} Let $I$ be a set of morphisms in $\mathcal{C}$.
\begin{enumerate}
\item A morphism $p:C\to D$ is called $I$-injective, if it has the right lifting property (see diagram above) with respect to all morphisms $i: A\to B$ in $ I$. The class of $I$-injective maps is denoted $Iinj$.
\item A morphism $i:A\to B$ is called $I$-projective, if $i$ has the left lifting property with respect to all morphisms $p: C\to D$ in $I$. The class of $I$-projective maps is denoted $Iproj$.
\item A morphism is an $I$-cofibration if it has the left lifting property with respect to all $I$-injective maps. The class of $I$-cofibrations is the class $(Iinj)proj$ and is denoted $Icof$.
\item A morphism is an $I$-fibration if it has the left lifting property with respect to all $I$-projective maps. The class of $I$-fibrations is the class $(Iproj)inj$ and is denoted $Ifib$.
\end{enumerate}
\end{definition}

The trivial fibrations $\mathcal R\cap \mathcal W$ are the $\mathcal L$-injective morphisms and
$\mathcal L\cap \mathcal W$ are the $\mathcal R$-projective morphisms
(\cite[lemma 1.1.10]{Hovey}). Since $f = p\circ i$ with $p\in \mathcal R\cap\mathcal W$
and $i\in \mathcal L$ implies $i\in \mathcal W$ for $f\in \mathcal W$, hence
$\mathcal L$ and $\mathcal R$ determine
$$\mathcal W = (\mathcal R \cap \mathcal W)\circ (\mathcal L \cap \mathcal W)\ .$$
$\mathcal L$ and $\mathcal R$ are closed under compositions,
(trivial) cofibrations are stable under pushout and (trivial) fibrations
are stable under pullback \cite[Corollary 1.1.11]{Hovey}. $\mathcal C$ has initial objects $0$ and terminal objects $*$.

\begin{definition} Let $\mathcal C_+$ resp. $\mathcal C_-$ denote the full subcategory of objects $X$ in $\mathcal C$ for
which $0 \to X$ is in $\mathcal L$ resp. $X\to *$ is in $\mathcal R\cap \mathcal W$. Objects in $\mathcal C_+$ are called cofibrant and objects in $\mathcal C_-$ trivially fibrant.
\end{definition}


\subsection{Cofibrant generation}

Suppose $\mathcal C$ admits arbitrary small limits and colimits. An object $X\in \mathcal C$ is called small, if it is $\kappa$-small for some cardinal $\kappa$: for all $\kappa$-filtered ordinals $\lambda$
and all $\lambda$-sequences $Y_i$ 
of morphisms in $I$ the canonical morphism
$co\lim_{i<\lambda} Hom(X,Y_i) \to Hom(X,co\lim_{i<\lambda} Y_i)$ is an isomorphism (\cite[p.29]{Hovey}). 
There is a similar notion for smallness with respect to a subcollection of the morphisms of $\mathcal C$.
A model category is cofibrantly generated, if  there exist {\it sets} of morphisms $J$ and $I$, such that the domains of the morphisms in $I$ (resp. $J$)
are small with respect to $I$-cellular (resp. $J$-cellular) maps and if
\begin{itemize}
\item The fibrations $\mathcal R$ are the $J$-injective morphisms $Jinj$
\item The trivial fibrations $\mathcal W\cap \mathcal R$ are the $I$-injective morphisms $Iinj$.
\end{itemize}
Then ${\mathcal L} = (Iinj)proj=Icof$ and ${\mathcal L\cap W} = (Jinj)proj=Jcof$. Hence $\mathcal L$ and $\mathcal L\cap \mathcal W$  are uniquely determined by $I$ and $J$, such that $I\subset \mathcal L$ and $J\subset \mathcal L\cap \mathcal W$.


\subsection{Quillen adjoint functors} 

Quillen adjoint functors are adjoint functors $F: \mathcal D \to \mathcal C$ and $U: \mathcal C \to \mathcal D$ 
$$ Hom_{\mathcal C}(FX,Y) = Hom_{\mathcal D}(X,UY) $$
between model categories $\mathcal D$ and $\mathcal C$
such that one of the following three equivalent conditions holds (see \cite[p.43]{Dwyer-Spalinski})
\begin{itemize}
\item $F$ maps (trivial) cofibrations to (trivial) cofibrations
\item $U$ maps (trivial) fibrations to (trivial) fibrations.
\item $F$ maps cofibrations to cofibrations and $U$ maps fibrations to fibrations.
\end{itemize}
If $\mathcal D$ is cofibrantly generated
by $J$ and $I$ this holds if $FI \subset \mathcal L_{\mathcal C}$ and 
$FJ \subset \mathcal L_{\mathcal C} \cap \mathcal W_{\mathcal C}$ (see \cite[p.14 and p.36, Lemma 2.1.20]{Hovey}). 


\subsection{Monoidal model structure} \label{sec:monoidal-model-def}

A symmetric monoidal category $(\mathcal C,\otimes)$ (see \cite[p.101 ff]{Hovey})  will be called
closed monoidal, if  
 internal ${\mathcal H}om$'s exist with functorial isomorphisms
$$ Hom_{\mathcal C}(X\otimes Y,Z) \cong  Hom_{\mathcal C}(X,{\mathcal  H}om(Y, Z)) \ $$
and the properties of \cite[Definition 4.1.13]{Hovey}. A model structure on a closed symmetric monoidal category $\mathcal C$ is called a symmetric monoidal model category if it satisfies the following two conditions \cite[Definition 4.2.6]{Hovey}:
\begin{enumerate}
\item The tensor functor ${\mathcal C \times \mathcal C} \to \mathcal C$ is a Quillen bifunctor \cite[Definition 4.2.1]{Hovey}.
\item If the unit $\one_{\mathcal C}$ is not cofibrant, factor $0\to \one$ into a cofibration and a trivial fibration $q:Q\one \to \one$. Then we require that
$q\otimes id: Q\one \otimes X \to \one\otimes X$ and $id\otimes q: X\otimes Q\one \to X\otimes \one$ are in $\mathcal W$ for all cofibrant $X$). 
\end{enumerate}



\section{  The stable module category of a Frobenius category}\label{sec:stable-cat}

An abelian (or more generally an exact) category $\mathcal D$ is a \emph{Frobenius category} if it has enough
projectives and enough injectives, and if the subcategories $\mathcal P_{\mathcal D}$ of projective objects and the subcategory $\mathcal I_{\mathcal D}$ of injective objects coincide $ \mathcal P_{\mathcal D}=\mathcal I_{\mathcal D}$.

\bigskip\noindent
Attached to an exact category $\mathcal D$ with enough injective objects is its stable category $\overline {\mathcal D}$ \cite{Happel} which is a suspended category. It has the same objects as $\mathcal D$. A morphism of $\overline {\mathcal D}$ is an equivalence class $\overline f$ of a morphism $f:X\to Y$ in $\mathcal D$ modulo the subgroup of morphisms factoring through an injective module. 
Objects $X,Y\in \mathcal D$ are called stably equivalent
if they become isomorphic in $\overline{\mathcal D}$.
The suspension
$SX=X[1]$ of $X$ is defined via an exact sequence $(X,IX,SX)$, where $i:X\to IX$ is an injective resolution and $SX=IX/i(X)$. Any morphism $f:X\to Y$ lifts to a morphism
$If: IX\to IY$, hence defines a suspension morphism $SX\to SY$ whose equivalence
class $Sf$ is well defined in $\overline{\mathcal D}$, i.e. independent of the resolutions and the choice of the 
lift $If$. Associated to an exact sequence $(X,Y,Z,j,p)$ in $\mathcal D$ is a standard
triangle $(X,Y,Z,i,p,\partial)$ in $\overline{\mathcal D}$, where $\partial: Z \to SX$
is the well defined class of the right vertical arrow
$$ \xymatrix@+5mm{X \ar@{=}[d]\ar@{^{(}->}[r]^j &  Y \ar@{.>}[d]\ar@{->>}[r]^p &  Z \ar@{.>}[d]^\partial  \cr
X \ar@{^{(}->}[r]^i &  IX \ar@{->>}[r] &  SX   \cr }$$
  in $\overline{\mathcal D}$.
A triangle $(A,B,C,a,b,c)$ in $\overline{\mathcal D}$ is called distinguished, 
if it is isomorphic to a standard triangle in $\overline{\mathcal D}$. Thus $\overline{\mathcal D}$
becomes a suspended category. 
If the exact category $\mathcal D$ is a Frobenius category, $\overline {\mathcal D}$ is a triangulated
category (see \cite{Chen}, \cite{Happel} or \cite[p.9]{Keller-use}). This is shown by using the loop functor defined by a projective resolution in a similar way. The following lemma is well-known.

\begin{lem} If  $\overline {\mathcal D}$ is the stable category of a Frobenius category ${\mathcal D}$, then for $n\geq 1$ \[Ext^n_{\mathcal D}(A,B) \ \cong \ Hom_{\overline {\mathcal D}}(A,B[n])\ .\]
\end{lem}

If $\mathcal D$ has arbitrary coproducts, the stable category
$\overline{\mathcal D}$ has arbitrary coproducts as well since the coproduct of objects $X_i, \ i\in I$, in $\overline{\mathcal D}$
is represented by the coproduct $\bigoplus_{i\in I} X_i$ in $\mathcal D$ (both categories have the same objects). Furthermore coproducts of projective objects are projective. From \cite[Proposition 1.6.8 and Lemma 3.2.10]{Neeman} we easily conclude

\begin{lem} Let $\mathcal D$ be a Frobenius category with arbitrary coproducts. Let
 $\mathcal H $ be a quotient category of the triangulated stable category $\overline{\mathcal D}$
of $\mathcal D$ devided by a thick triangulated subcategory. Then $\mathcal H$ is pseudo-abelian and has arbitrary
coproducts, and the functor $\mathcal D \to \mathcal H$ commutes with arbitrary coproducts.
\end{lem} 

A particular example of a Frobenis category is the category of left $R$-modules of a Frobenius ring as in \cite[2.2]{Hovey}. In this setting the stable category of $R$-mod carries a cofibrantly generated model structure \cite[Theorem 2.2.12]{Hovey}. A general Frobenius category $\mathcal D$ carries a model structure for which the associated homotopy category is the stable category \cite{Li-stable}. Under mild additional conditions $\overline{\mathcal D}$ is the homotopy category of a cofibrantly generated model structure on $\mathcal D$.  

\medskip\noindent
Assume the properties \ref{stable-f-1}-\ref{stable-f-4}.
\begin{enumerate}[label=\textbf{FC.\arabic*}]
\item \label{stable-f-1} $\mathcal D$ has small limits and colimits
\item \label{stable-f-2} Any object of $\mathcal D$ is  small with respect to $\mathcal D$ in the sense of \cite[p.29]{Hovey}
\item \label{stable-f-3} There exists a {\it set} $P_{\mathcal D}$ 
of projective objects in $\mathcal D$ being generators of $\mathcal D$ in the sense
that any nontrivial object $X$ of $\mathcal D$ admits a nontrivial morphism
$P\to X$ for some $P\in I$. We call $P_{\mathcal D}$ an admissible set of projectives
and denote by $J$ the set of monomorphisms $0\to P$ for $P\in P_{\mathcal D}$.
\item \label{stable-f-4} There exists a {\it set} $I$ of monomorphisms $i:A\to B$ in $\mathcal D$ containing $J$ such that $X\in \mathcal D$ is in $\mathcal I_{\mathcal D}$ if and only if $i^*: Hom_{\mathcal D}(B,X)\to Hom_{\mathcal D}(A,X)$ is surjective for all $i\in I$.  We call $I$ an admissible set of monomorphisms.  \end{enumerate}

We ignore whether these conditions can be relaxed. The proof of the following lemma is analog to the proofs of \cite{Hovey} and \cite{Li-stable} and will be skipped.

\begin{lem} \label{lemma-stable-cat} A Frobenius category $\mathcal D$ with the properties FC1-4 carries a  cofibrantly generated model structure called the stable model structure. The fibrations are the epimorphisms, the trivial fibrations are the split epimorphisms with kernel in $\mathcal I_{\mathcal D}$. The cofibrations are the monomorphisms, and the trivial cofibrations are the split monomorphisms with cokernel in $\mathcal P_{\mathcal D}$. Every object is fibrant and cofibrant. The associated homotopy category $Ho\mathcal D$ is the stable category $\overline{\mathcal D}$ of $\mathcal D$. The morphisms in $\mathcal W$  consist of the morphisms for which $Ext^i(X,f)$ and $Ext^i(f,X)$
are isomorphisms for all objects $X\in \mathcal D$ and all $i\geq 1$. The model structure is cofibrantly generated by $I$ and $J$ as above.
\end{lem}
  


\section{  Induced model structure}\label{sec:induced}

\subsection{Induced model structures} \label{sec:induced-1}

The following construction of D. M. Kan can be used to lift model structures from a cofibrantly generated model category to another \cite[Theorem 11.3.1, Theorem 11.3.2]{Hirschhorn}. Let $\mathcal D$ be a cofibrantly generated model category generated by 
$J_{\mathcal D}$ and $I_{\mathcal D}$ (we may assume $ I_{\mathcal D}$ contains $J_{\mathcal D}$). Let
$\mathcal C$ be a category with small limits and colimits and adjoint functors
$$  F: {\mathcal D \to \mathcal C} \quad , \quad  U: \mathcal C \to \mathcal D $$
such that $Hom(FX,Y)=Hom(X,UY)$. Suppose $U$ is faithful.
Put $J=F(J_{\mathcal D})$ and $I=F(I_{\mathcal D})$ and suppose
$$  U(J cof) \ \subset \ \mathcal W_{\mathcal D}  $$
(and certain smallness conditions in $\mathcal C$
automatically fulfilled in our later cases).
Then there exists a model structure $({\mathcal L},{\mathcal R},{\mathcal W})$ on $\mathcal C$ generated by $J$ and $I$
such that: $f\in {\mathcal W}$ if and only if $U(f)\in \mathcal W_{\mathcal D}$. The functors $(F,U)$ define a Quillen
adjunction. 

\bigskip\noindent
{\it The abelian case}. We will only apply this in the case where $\mathcal C$ and $\mathcal D$ are abelian
categories. So let us always assume this in the following. Recall that for
an adjoint pair of functors  $F:\mathcal D\to C$ and $U:\mathcal C\to \mathcal D$ between abelian categories
$F$ preserves colimits and $U$ preserves limits (e.g.\cite[p.137]{Gelfand-Manin}). Hence, if 
$F$ is left exact ($U$ is right exact), then $F$ (resp. $U$) is exact
and $U$ preserves injectives (resp. $F$ preserves projectives). 
The adjunction morphisms $ad:  FUX\to X$ are epimorphisms if $U$ is faithful $U(Z)=0 \Leftrightarrow Z=0$ (\cite[p.61 formula II.16]{Gelfand-Manin}).  Similarly the adjunction morphisms 
$Y\to UFY $ are monomorphisms if $F$ is faithful $F(Z)=0 \Leftrightarrow  Z=0$.

\begin{prop} \label{thm:induced-model} Assume that $\mathcal C$ and  $\mathcal D$ are abelian Frobenius categories, and that $\mathcal D$ satisfies the finiteness assumptions \ref{stable-f-1} - \ref{stable-f-4}. Assume also that $F$ and $U$ are exact.
Consider the stable model structure on $\mathcal D$
 cofribrantly generated by the set $J_{\mathcal D}=\{{0\to X, X\in P_{\mathcal D}} \}$ and by a set $I_{\mathcal D}\supset J_{\mathcal D}$
 of admissible monomorphisms, where $\mathcal W_{\mathcal D}$ is defined by stable equivalence.
\begin{itemize} 
\item Then \[ J cof \subset  {\mathcal W} \cap \{ split\ mono\}.\]
\item There exists a model structure $({\mathcal L},{\mathcal R},{\mathcal W})$ on $\mathcal C$ generated by $J$ and $I$ (as defined above) such that: $f\in {\mathcal W}$ if and only if $U(f)\in \mathcal W_D$. The functors $(F,U)$ define a Quillen
adjunction.
\end{itemize}
\end{prop}

\begin{definition} \label{def:frobenius-pair} We call a pair of abelian Frobenius categories $\mathcal C$ and  $\mathcal D$ with an adjoint pair of functors $U,F$ satisfying the conditions of proposition \ref{thm:induced-model} a Frobenius pair.
\end{definition}

We only have to prove part (1) of proposition \ref{thm:induced-model}. This claim immediately follows from the next two lemmas, since $U({\mathcal P_{\mathcal C}}) \subset
\mathcal P_{\mathcal D}$.

\begin{lem} \label{lemma-1}The morphisms $p:\mathcal{C} \to \mathcal{D}$ in ${\mathcal R} = Jinj$ are the epimorphisms in $\mathcal C$.
Every object of $\mathcal C$ is fibrant.
\end{lem}

\begin{proof} Let $Q$ be the kokernel of $p$. Notice $Q=0$ if and only if $UQ=0$, since $U$ is faithful. If $UQ$ were not zero,
choose a nontrivial morphisms $u:P\to UQ$ for $P\in \mathcal P_{\mathcal D}$ by \ref{stable-f-3}, lift it to
a morphism $P\to UD$ (projectivity of $P$) and
consider the adjoint morphism $g: FP \to D$.  Since $0 \to FP$ is in $J$, for $p\in \mathcal R$ the morphism $g$ lifts by definition
to a morphism $h:FP\to C$. This implies that the adjoint morphism $FP \to Q$ attached to
$u:P\to UQ$ is zero, since it factorizes over the zero morphism $C\to D\to Q$.
Hence $u$ itself is zero. Contradiction.
\end{proof}
 
The proof of the next lemma is similar to \cite[Lemma 2.2.11]{Hovey} and will be skipped.

\begin{lem} \label{lemma-2} The morphisms $i:A\to B$ in ${\mathcal R}proj = Jcof$ are the split monomorphisms in $\mathcal C$ with projective kokernel.
\end{lem}


\subsection{Explicit description of the induced model structure}

In fact we now explicitly describe these model structures via
\cite[2.1.19]{Hovey}. Again let us first ignore some important finiteness conditions
that have to be imposed. The reason is that these finiteness conditions are satisfied for categories of comodules of supercommutative
Hopf algebras $A$ (where we apply this). We usually call the functor $U$ the restriction functor.

\medskip\noindent
Step I. Define $\mathcal W$ as before by
$ f\in \mathcal W $ if and only if $U(f)\in \mathcal W_{\mathcal D}$. Then the two out of three property holds and $\mathcal W$ is closed under retracts

\medskip\noindent
Step II. Define $J=F(J_{\mathcal D})$ and $I=F(I_{\mathcal D}) \supset J$ as in the last section.
Put ${\mathcal R} = Jinj$ and ${\mathcal L} = Icof$.
Notice $I \subset Icof={\mathcal L}$. 
For $(\mathcal L,\mathcal R,\mathcal W)$ to be a model structure it suffices by the smallness property stated above that
\begin{enumerate}
\item $Jcof \subset \mathcal W$ (obvious by lemma \ref{lemma-2}), 
\item $Iinj \subset \mathcal W$, and then also $\subset \mathcal W\cap \mathcal R$ since ${\mathcal R}=Jinj \supset Iinj$ 
\item $(Jinj) \cap {\mathcal W =: \mathcal R\cap \mathcal W} 
\subset Iinj$ 
\end{enumerate}
The last property might be replaced by $ (Icof) \cap {\mathcal W} =: {\mathcal L \cap \mathcal W} \subset Jcof$, also denoted property 3'. Since all four conditions necessarily hold in a model
category, this shows that conditions 1,2,3 imply 3'. Since $J\subset I$ in our situation, we also have  $Jcof \subset Icof = \mathcal L$ and therefore by property 1 
$$ (Icof) \cap {\mathcal W} = {\mathcal L \cap \mathcal W}  = Jcof$$
once we have shown the properties 2 and 3.

\medskip\noindent
{\it Property 2}. $Iinj \subset \mathcal W$ follows from the next

\begin{lem} \label{lemma-3} $Iinj $ consists of the epimorphisms $p\in \mathcal R$ with kernel $K$ so that $UK$ is injective. 
\end{lem}

\begin{proof}
Recall that $I$ was the class of all morphisms $F(i)$ for admissible monomorphisms $i:A\to B$. 
Consider an $I$-injective morphism $p:X\to Y$
$$ \xymatrix{ F(A) \ar[r]\ar[d]_{F(i)} & X \ar[d]^p\cr
F(B) \ar@{.>}[ur]^h \ar[r] & Y \cr
 }$$
 Recall $J\subset I$, hence $Iinj \subset Jinj$, and therefore
 $p$ must be an epimorphism by lemma \ref{lemma-1}.
 Put $K=kern(p)$ and $Q=Kokern(i)$.
We now pass to the restriction
 diagram using adjunction
$$ \xymatrix{  & UK \ar[d]\cr
A \ar[r]\ar[d]_{i} & UX \ar[d]^{U(p)}\cr
B \ar@{.>}[ur] \ar[r]\ar[d] & UY \cr
Q  & }$$
Hence $U(p)$ has to be $I_{\mathcal D}$-injective with respect to the class $I_{\mathcal D}$ of all admissible monomorphisms
in $\mathcal D$. We have shown that this is equivalent to $U(p)$ is a split
epimorphism with injective kernel $UK$. 
Hence $K$ is in $\mathcal C_-$ and $p:X\to Y$ is 
in $\mathcal W$. 
\end{proof}

{\it Property 3} is now an immediate consequence of lemma \ref{lemma-1} and \ref{lemma-3}.

\begin{cor} An object $K$ is in $\mathcal C_-$  
if and only if $UK$ is injective ( if and only if $K\to 0$ is in ${\mathcal R \cap \mathcal W} =Iinj $).
\end{cor}

Obviously  $\mathcal I_{\mathcal C} \subset {\mathcal C}_- $ and $\mathcal C_-$ is closed under retracts. 

\begin{cor} The morphisms in ${\mathcal R \cap \mathcal W} =Iinj $ are the epimorphisms
 $p$ with kernel in $K\in \mathcal C_-$.
\end{cor}

Property 3'  now holds automatically, and using lemma \ref{lemma-2} gives 

\begin{cor} The morphisms in $\mathcal L \cap \mathcal W $ are the  split
monomorphisms with projective kokernel.
\end{cor}


We now describe the class $\mathcal L$. We start with a technical lemma. In an abelian category $\mathcal C$ we consider diagrams
$$ \xymatrix{ A \ar@{^{(}->}[d]^i\ar[r]^f & C \ar@{->>}[d]^p \cr
B \ar@{.>}[ru]^h\ar[r]^g & D \cr} $$
with  epimorphism $p$ (with kernel $K$) and monomorphism $i$ (with kokernel $Q$)
and look for liftings $h$ making the diagram commutative.

\begin{lem} \label{lem:lifting} \cite[Chapter VIII Lemma 3.1]{Beligiannis-Reiten} Under the assumption $Ext^1(Q,K)=0$ the diagram above has the lifting property in the situation above ($i$ monomorphism and $p$ epimorphism).
\end{lem}

Define $\mathcal L_-$ to be the class of morphisms $${\mathcal L}_- = \{{\mathcal C}_- \to 0\}proj \ . $$
Then ${\mathcal L}=Icof \subset \mathcal L_-$ is a consequence of lemma \ref{lemma-3}. 
Since $F$ and $U$ are exact, $F$ preserves monomorphisms and projectives, and $U$ preserves injectives and monomorphisms (since $U$ was supposed to be faithful). 
If $\mathcal C$ has enough injectives, this implies that any morphism
in $\mathcal L_-$ is a monomorphism. Indeed it suffices to observe that any $A\in \mathcal C$ can be embedded into an injective object $L$ of $\mathcal C$. Since $\mathcal I_{\mathcal{C}} \subset \mathcal C_-$
the lifting property then forces any $i \in \mathcal L_-$ to be a monomorphism. 

\begin{lem} $\mathcal L$
is the class
of monomorphisms in $\mathcal C$ whose kokernel is in $\mathcal C_+$.
\end{lem}

\begin{proof} We know $\mathcal L\subset \mathcal L_-$ and any morphism $i:A\to B$ in $\mathcal L_-$ is a monomorphism, since $\mathcal C$ has enough injectives. Let $Q$ denote the kokernel of $i$. Since the class of cofibrations is always closed under pushouts (\cite[cor.1.1.11]{Hovey}), $i\in \mathcal L$ implies that $0\to B/i(A)=Q$ is in $\mathcal L$, i.e. cofibrant
$$ \xymatrix{ A \ar[r]\ar[d]^{i\in \mathcal L}  & 0 \ar[d] \cr B \ar[r]  & B/i(A)\cr } $$
Hence $Q \in \mathcal C_+$ is a necessary condition (any extension $E$ of $Q$ by $K\in \mathcal C_-$
gives rise to a morphisms $p:E\to Q$ in $\mathcal R\cap \mathcal W$). For the converse we have to show the lifting property for epimorphisms $p$
with kernel in $\mathcal C_-$. The claim follows from lemma \ref{lem:lifting} since
$Q\in \mathcal C_+$ and $K\in \mathcal C_-$
\end{proof}


\begin{cor} $\mathcal C_- \cap \mathcal C_+ = \mathcal{I}_{\mathcal C} $.
\end{cor}

\begin{proof} From the description of $\mathcal L$ from above and the description of $\mathcal W$ we conclude: $i\in \mathcal L\cap \mathcal W$ if and only $i$ is monomorphisms with kokernel $Q\in \mathcal C_+\cap \mathcal C_-$. On the other hand we know that $\mathcal L\cap \mathcal W$ are the split monomorphisms with injective kokernel.
Considering this for the monomorphisms $i: 0\to Q$, the corollary follows.
\end{proof}

\begin{cor}\label{model-structure-descr} \begin{enumerate} 
\item The fibrations $\mathcal{R}$ are the epimorphisms.
\item The trivial fibrations $\mathcal{R} \cap \mathcal{W}$ are the epimorphisms with kernel in $\mathcal{C}_-$.
\item The cofibrations $\mathcal{L}$ are the monomorphisms with kokernel in $\mathcal{C}_+$.
\item The trivial cofibrations $\mathcal{L} \cap \mathcal{W}$ are the split monomorphisms with kokernel in $\mathcal{C}_+^{triv} = \mathcal{P}_{\mathcal{C}}$.
\end{enumerate}
\end{cor}


\subsection{Cotorsion pairs} Model categories whose morphisms satisfy certain good properties as in corollary \ref{model-structure-descr} give rise to cotorsion pairs as described in \cite{Beligiannis-Reiten} \cite{Hovey-cotorsion-1}. Recall from \cite[Chapter V Definition 3.1]{Beligiannis-Reiten} the following definition: A cotorsion pair in an abelian category $\mathcal{C}$ is a pair of  full subcategories $(\mathcal{X},\mathcal{Y})$ closed under isomorphisms and direct summands satisfying

\begin{enumerate}
\item $Ext^1(\mathcal{X},\mathcal{Y}) = 0$.
\item For any object $C \in \mathcal{C}$ there exists a short exact sequence \[ \xymatrix{ 0 \ar[r] & Y_C \ar[r] & X_C \ar[r] & C \ar[r] & 0 } \] with $Y_C \in \mathcal{Y}$ and $X_C \in \mathcal{X}$.
\item For any object $C \in \mathcal{C}$ there exists a short exact sequence \[ \xymatrix{ 0 \ar[r] & C \ar[r] & Y^C \ar[r] & X^C \ar[r] & 0  } \] with $Y^C \in \mathcal{Y}$ and $X^C \in \mathcal{X}$.
\end{enumerate}

Furthermore \cite[Chapter I Definition 2.1]{Beligiannis-Reiten} a torsion pair in $\mathcal{C'}$ ($\mathcal{C'}$ triangulated with suspension $\Sigma$) is a pair of strict full subcategories $(\mathcal{X},\mathcal{Y})$ satisfying

\begin{enumerate}
\item $Hom(\mathcal{X},\mathcal{Y}) = 0$.
\item $\Sigma \mathcal{X} \subset \mathcal{Y}$ and $\Sigma^{-1}(\mathcal{Y}) \subset \mathcal{Y}$.
\item For any $C \in \mathcal{C'}$ there exists a triangle \[ \xymatrix{ X_C \ar[r] & C \ar[r] & Y^C \ar[r] & \Sigma(X_C) } \] in $\mathcal{C}$ such that $X_C \in \mathcal{X}$ and $Y^C \in \mathcal{Y}$.
\end{enumerate}

\begin{thm} \label{thm:main-homotopy} \begin{enumerate}[label=\textbf{S.\arabic*}]
\item  \label{cof-trivial} $\mathcal{C}_+^{triv} = Proj$.

\item \label{cotorsion} The pairs $(\mathcal{C}_+,\mathcal{C}_-)$ and $(\mathcal{C}_+^{triv},\mathcal{C})$ are cotorsion pairs in $\mathcal{C}$.

\item \label{torsion} $(\mathcal{C}_+,\mathcal{C}_-)$ and $(\mathcal{C}_+^{triv} = Proj,\mathcal{C})$ define torsion pairs in the triangulated category $\overline{\mathcal{C}}$. In particular any map between a cofibrant object $X$ and a trivially fibrant object $Y$ factors through a projective object.

\item \label{cof-fib-prop}  $\mathcal{C}_-$ and $\mathcal{C}_+$ are closed under extensions and cokernels of monics and kernels of epis. 

\item \label{cof-frobenius} $\mathcal{C}_+$ is a Frobenius category.
\end{enumerate}
\end{thm}

\begin{proof} 

\ref{cof-trivial}: By \cite[Chapter VIII Lemma 2.2]{Beligiannis-Reiten} $\mathcal{C}_+ \cap \mathcal{C}_- = \mathcal{C}^{triv}_-$, therefore $\mathcal{C}_+^{triv} = Proj$.

\ref{cotorsion} is \cite[Chapter VIII Proposition 3.4]{Beligiannis-Reiten}.

\ref{torsion} follows from \cite[Chapter VIII Corollary 3.7]{Beligiannis-Reiten} and \cite[Chapter I Proposition 2.6]{Beligiannis-Reiten} (see also \cite[Chapter VIII Proposition 2.1]{Beligiannis-Reiten}). 

\ref{cof-fib-prop} follows from \cite[Chapter V Lemma 2.4]{Beligiannis-Reiten} and \cite[Chapter I Corollary 2.9]{Beligiannis-Reiten}. 

\ref{cof-frobenius} is \cite[Chapter VI Theorem 2.1]{Beligiannis-Reiten}.

\end{proof}

\begin{remark} The torsion pair is hereditary (see \cite[Chapter I Definition 2.5]{Beligiannis-Reiten}) and the categories $\mathcal{C}_-$ and $\mathcal{C}_+$  are resolving and coresolving.
\end{remark}



\subsection{Permanence properties of cofibrant objects}

We discuss some easy consequences for $\mathcal C_+$ and $\mathcal C_-$.

\begin{lem} For a monomorphism in $\mathcal C_+$ $$i: N\hookrightarrow N' \ $$
any morphism $$ \varphi: N \to M \quad , \quad M\in \mathcal C_- $$
can be extended to a morphism $\varphi': N' \to M$ such that $\varphi'\circ i =\varphi$.
\end{lem}

\begin{proof} We may replace $\varphi : N \to M$ by
$\varphi: N \to I$ for some injective object by theorem \ref{thm:main-homotopy}. Then the assertion becomes obvious. 
\end{proof}

\begin{cor} $\mathcal C_+$ is closed under sequential direct limits
with monomorphic transition maps. Similarly $\mathcal C_-$
is closed under sequential projective limits with surjective transition maps.
\end{cor}

\begin{proof} Suppose $X=co\lim_i X_i$ with $X_i\in \mathcal C_+$.
Consider an extension $E$ with class in $Ext^1(X,M)$ and $M\in \mathcal C_-$.
Its pullback to each $X_i$ splits. Let $\lambda_i: X_i \to E$ be a splitting morphism. It is unique 
up to a morphism $X_i\to M$. Since $\varphi: \lambda_{i+1}\vert_{X_i} - \lambda_i :X_i \to M$
can be extended to $X_{i+1}$ by the last lemma, one can choose $\lambda_{i+1}$ in such a 
way that $\varphi =0$. Proceeding inductively gives a compatible system of splittings which define a splitting $\lambda: X=co\lim_i X_i \to E$. Hence $Ext^1(X,M)=0$. The second assertition can be proven similarly.
\end{proof}

\begin{cor} For $C\in \mathcal C_+$ we have $Ext^i(C,M)=0$
for all $M\in\mathcal C_-$ and all $i\geq 1$.
\end{cor}

\begin{proof} Choose a monomorphism $M\hookrightarrow I$ with $I\in \mathcal I_{\mathcal C} \subset C_-$. Then  $I/M \in \mathcal C_-$, hence
$Ext^2(C,M)\cong Ext^1(C,I/M) = 0$. Then proceed by induction on $i$ to complete the proof. 
\end{proof}


\section{  The homotopy category $Ho\mathcal{C}$} \label{sec:homotopy-category}

We now discuss properties of the homotopy category of $\mathcal C$ and give a more explicit description as the stable category of the category of cofibrant objects. We first recall some background about cofibrant replacements and the homotopy category.

\subsection{Cofibrant replacements} Every object $X \in \mathcal{C}$ has a cofibrant replacement $q:QX \to X$ where $QX$ is cofibrant and $q$ is a trivial fibration. We denote by $Q$ the cofibrant replacement functor $\mathcal{C} \to \mathcal{C}_+$, $X \mapsto QX$. Note that the morphism $q$ defines an extension
$$ 0 \to K \to QX \to  X \to 0 $$
with kernel $K\in \mathcal C_-$, i.e. $U(K)\in \mathcal I_{\mathcal D}$.

\subsection{The homotopy category} The homotopy category of $\mathcal{C}$ is the localization $Ho \mathcal{C} = \mathcal{C}[\mathcal{W}^{-1}]$ by the weak equivalences. The functor $\gamma: \mathcal{C} \to Ho \mathcal{C}$ is then the identity on objects and takes morphisms in $\mathcal{W}$ to isomorphisms. It is universal with this property in the sense of \cite[Lemma 1.2.2]{Hovey}. For $X,Y \in \mathcal{C}$ there are natural isomorphisms \[ Hom_{Ho\mathcal{C}}(\gamma X, \gamma Y) \simeq Hom_{\mathcal{C}}(QX,RY)/\sim \] where $\sim$ denotes the homotopy equivalence \cite[Theorem 1.2.10]{Hovey}. We often abbreviate $Hom_{Ho\mathcal{C}}(X,Y)$ or $Hom_{\mathcal{C}}(\gamma X, \gamma Y)$ as $[X,Y]$.

\medskip\noindent

The category $\mathcal E = \mathcal C_+$ is an exact subcategory of the abelian category
$\mathcal C$ in the sense of \cite{Keller-use} (section 4 and 5)  (and similarly is $\mathcal C_-$). In particular $\mathcal E$ is closed under extensions. $\mathcal E$ contains $\mathcal I_{\mathcal C}= \mathcal P_{\mathcal C}$. Since
$\mathcal P_{\mathcal{C}} \subset \mathcal P_{\mathcal{E}}$ and $\mathcal I_{\mathcal{C}} \subset \mathcal I_{\mathcal{E}}$, the category $\mathcal E$ has enough injectives and projectives. 

\begin{thm} \label{thm:homotopy-stable} \begin{enumerate}[label=\textbf{H.\arabic*}]
\item \label{frob} $\mathcal{C}_+$ is a Frobenius category. Hence $\overline{\mathcal C}_+=\mathcal E/\mathcal I_{\mathcal E}$ 
is a triangulated category. 
\item \label{stable-cat}$ Ho\mathcal C$ is a triangulated category and quasi-equivalent to the stable category of $\mathcal C_+$
$$ Ho{\mathcal C} = \overline{\mathcal C}_+ \ , $$
and the functor ${\mathcal C}\to Ho\mathcal C$ is $k$-linear. In particular $$ [X,Y] = Hom(QX,Y)/\sim_{stable}. $$
\item \label{triangulated} The functors $F,U$ preserve distinguished triangles and the shift functor.
\end{enumerate}
\end{thm}

\begin{proof} \ref{frob} follows from \cite[Chapter VI Theorem 2.1]{Beligiannis-Reiten}. \ref{stable-cat} follows from \cite[Chapter VIII Theorem 4.2, Corollary 4.5]{Beligiannis-Reiten}. \ref{triangulated} follows from \ref{frob} and \ref{stable-cat}.
\end{proof}



\section{  Comodules and supercommutative Hopf algebras}


The main example of a Frobenius pair in this paper comes from representation theory. 
Let $\pi: A\to B$ be a surjective homomorphism of supercommutative Hopf algebras. It corresponds to an embedding $H \subset G$ of supergroups. This embeddings induces a restriction functor $Res: Rep(G) \to Rep(H)$ which is an exact faithful tensor functor. We discuss its left and right adjoint (called coinduction and induction). Under suitable conditions on $H$ and $G$ this will give rise to a Frobenius pair.

\subsection{$k$-linear tensor categories with finiteness conditions} \label{sec:finiteness-conditions}

Let $k$ be a field of characteristic $0$. For a $k$-linear abelian category $\mathcal T$ consider the following finiteness conditions 

\medskip\noindent
{\it Condition (F)}. Every object has finite length and all morphism spaces $Hom(X,Y)$ have finite $k$-dimension. 

\medskip\noindent
{\it Condition (G).} The category $\mathcal{T}$ is monoidal and there exists an object $X$ such that any object is a subquotient of some tensor power of $X$.

\medskip\noindent
{\it Modules}. If conditions (F) and (G) hold, then 
the abelian category admits a projective generator (\cite[Proposition 2.14]{Deligne-Festschrift}). Hence by Morita's theorem it is  equivalent to the category of right modules of some $k$-algebra of finite rank (\cite[Corollary 2.17]{Deligne-Festschrift}). In particular there exist enough injectives and projectives. By assumption (F) the category is nice in the sense
of \cite[p.371]{Germonie}: Indecomposable injective objects in $\mathcal T$ have unique simple quotients (and by duality unique simple submodules) and have local endomorphism rings. Any object satisfies the Krull-Schmidt theorem, i.e. is isomorphic to a finite direct sum of indecomposable objects. The isomorphism classes in $\mathcal T$ of these indecomposable objects (up to permutation of the summands) do not depend on the decomposition. Up to isomorphism objects have unique projective covers resp. injective hulls. For any object $X$ there are only finitely many indecomposable objects $Y$ with
$Hom(Y,X)\neq 0$. 

\medskip\noindent
{\it Comodules}. By \cite[Theorem 2.13]{Simson} a $k$-linear abelian category $\mathcal T$ is $k$-linear equivalent to the category of 
of $k$-finite dimensional comodules over some $k$-coalgebra $A$ if and only the finiteness conditions (F) hold.
By \cite[p.141 -146]{Batchelor} in the category $\mathcal{C}$ of all comodules of a fixed $k$-coalgebra $A$ an injective comodule $L$ in $\mathcal C'$ is a direct sum $L =\bigoplus_\mu J_\mu$ of indecomposable injective subcomodules $J_\mu$, and each $J_\mu$ contains a unique simple subcomodule $X_\mu$. Under some further assumptions (see Section \ref{sec:frobenius} below) all simple 
$A$-comodules are finite dimensional, and the comodules $J_\mu$ are injective hulls in $\mathcal T$ and hence are $k$-finite dimensional. Hence injective objects $L\in \mathcal C$ are direct sums of $k$-finite dimensional injective
objects. In particular $\mathcal C$ has enough injectives.


\subsection{Supercommutative Hopf algebras}
For a supercommutative Hopf algebra $A$ \cite{Masuoka-hopf} \cite{Weissauer-semisimple} over an arbitrary field $k$
let now  
\begin{itemize}
\item $\mathcal T$ the category
of $k$-finite dimensional graded $A$-comodules, and
\item $\mathcal C=\mathcal T^\infty$ be the abelian category of all graded $A$-comodules.
\end{itemize}
These $k$-linear tensor categories $\mathcal T$ always satisfy the finiteness condition (F).
Condition (G) holds if $A$ is finitely generated over $k$ and $char(k)\neq 2$. 
Condition (G) follows from the fact that every faithful representation of an algebraic supergroup is a tensor generator as in the algebraic group case \cite[Proposition 2.20]{Deligne-Milne} \cite[Section 3.5]{Waterhouse}, and its proof is virtually the same as in the classical case \cite[Proposition 3.1]{Deligne-hodge} using \cite[Proposition 9.3.1]{Westra} that every finite-dimensional representation $V$ of $G$ is a submodule of $k[G]^{\dim V}$ (See \cite[Section 7]{Comes-Heidersdorf}).
We list some general important properties of these categories 
\begin{enumerate}
\item $\mathcal C$ has all small limits and colimits (\cite[Corollary 2.5.6]{Hovey}).
\item All objects in $\mathcal C$ are small
(\cite[cor 2.5.7]{Hovey}).
\item Any $X\in\mathcal C$ is an inductive union $X = co\lim_i X_i$ of  $X_i\in \mathcal T$  (\cite[Lemma 2.5.1]{Hovey}). It is easy to see that $\mathcal C= \mathcal T^\infty$ is the ind-category of $\mathcal T$
in the sense of \cite{Deligne-tensorielles}.
\item $\mathcal C$ is a closed symmetric monoidal category
(\cite[ p. 63]{Hovey} and \cite[2.5.1]{Hovey}).
\item The abelian category $\mathcal C$ only depends on the
coalgebra structure of $A$. 
\item An $A$-module $L\in \mathcal C$ is injective if and only if $L$ has the lifting property 
for all monomorphisms $i:A\to B$ in $\mathcal T$ (\cite[Proposition 2.5.8]{Hovey}). The isomorphism classes
of monomorphisms in $\mathcal T$ are a set (\cite[Definition 2.5.12]{Hovey}).
\item $\mathcal I_{\mathcal C}$ is a tensor ideal (\cite[Proposition 2.5.8]{Hovey}).
\item $A$ is injective (\cite[Proposition 2.5.8]{Hovey}).
\item The isomorphism classes of indecomposable injective modules are a set
of generators. Each of them is a direct summand of $A$
\item Coproducts of injectives are injective (\cite[Proposition 2.5.8]{Hovey}).  $\mathcal C$ has enough injectives.
\item Every comodule $M$ has a nontrivial socle $soc(M)$, the sum of the simple subcomodules
of $M$. The socle can be described by a wedge construction, and 
defines a left exact functor  $soc: \mathcal C \to \mathcal C$ \cite[p.142]{Batchelor}.
\item Indecomposable injective comodules have irreducible socle.
\item Properties \ref{stable-f-1}-\ref{stable-f-4} hold for the category $\mathcal C$.
\end{enumerate}
Let $\mathcal C'$ be a full subcategory of $\mathcal C$ containing $\mathcal T$ closed under direct summands. Then an injective module $L\in \mathcal I_{\mathcal C'}$ is injective in the category of all $A$-comodules.

\begin{remark} We point out that all direct limits (colimits) in the comodule category are sequential direct limits. In particular direct limits are exact. We will use this frequently without further justification. 
\end{remark}

\subsection{Frobenius categories}\label{sec:frobenius}

By the usual dictionary between supercommutative Hopf algebras and affine supergroup schemes, the category $\mathcal T$ 
of $k$-finite dimensional graded $A$-comodules is equivalent to the finite-dimensional representations $Rep(G)$ for an affine supergroup scheme $G$, and $\mathcal C=\mathcal T^\infty$ is equivalent to the category of all algebraic representations. For an affine supergroup scheme denote by $G_0$ the underlying even group scheme.

From now on we also make the crucial 

\medskip\noindent
{\bf Assumption}. {\it $G$ is quasireductive, i.e. $G_0$ is reductive.}

\medskip\noindent
Since $Rep(G_0)$ is reductive, this implies that $\mathcal{T} \cong Rep(G)$ is a Frobenius category. This also applies to the subcategories $Rep(G,\epsilon)$ in the sense of Deligne \cite{Deligne-tensorielles}.


Under these assumptions one has the following lemma.

\begin{lem} \label{inj-lemma} Any injective comodule $I$ is a direct sum of finite dimensional indecomposable injective comodules.
\end{lem}

\begin{proof} In the context of representations of quasireductive
groups we now also use the language of modules instead of comodules following \cite{Serganova-quasireductive} \cite{Jantzen}.
 By \cite{Serganova-quasireductive} (see Theorem 9.2) \[ k[G] \cong \bigoplus_L I(L)^{m(L)} \] for some multiplicity $m(L)$ where the sum runs over all finite dimensional injective hulls $I(L)$ for finite dimensional irreducible $L$. We can embed $I$ into a module of the form $V \otimes k[G]$ for a super vector space $V$. Indeed, by the adjointness of induction and coinduction between modules for $k[G]$ and $k[G_0]$ this is reduced to the well-known case $G= G_0$  \cite[Proposition 3.10a]{Jantzen} in view of $Ind_{G_0}^G (k[G_0]) = Ind_{G_0}^G (Ind_1^{G_0} (1)) = k[G]$. Since $I$ is injective, it is the image of an idempotent $\varphi \in End(V \otimes_k k[G])$. 
Let $\overline\varphi$ denote its restriction to the socle of $V \otimes_k k[G])$.
Note that the usual properties of socle functors carry over from \cite[I.3]{Jantzen} to the super setting; in particular
\begin{enumerate}
\item Any indecomposable injective module is determined by its simple socle up to isomorphism.
\item If $M$ is an injective module  and $\bar{\varphi} \in End(soc(M))$ an idempotent, then there exists an idempotent lift $\psi \in End(M)$ with $\psi|_{soc(M)} = \bar{\varphi}$.
\end{enumerate}
Therefore $I$ is isomorphic to the image of an idempotent lift $\psi$ of an idempotent $\bar{\varphi}$ of $V \otimes_k \bigoplus L^{m(L)}$. The idempotent $\bar{\varphi}$
preserves the isotopic components $V \otimes_k L^{m(L)}$. Hence we can assume that the same holds for the lift $\psi$. 
If we replace  $V^{m(L)}$ by $V$, this allows to assume that $I$ is the image of an idempotent $\psi$ of some \lq$L$-{isotypic}\rq\ injective module $V\otimes_k I(L)$ for a fixed simple module $L$. The endomorphism ring of $L$ defines a finite field extension $K$ of $k$. 
The functor $Hom(L,-)$ is a functor from the category of $L$-isotypic modules of the form $V \otimes_k L$ to the category of $K$-right vector spaces $W$, defining a quasi-equivalence of categories by the converse functor $W\mapsto W\otimes_K L$.
$Hom(L,-)$ maps the idempotent $\bar{\varphi}$ of $V\otimes_k L$ to an idempotent  $e$ of $W=Hom(L,V \otimes L) = V\otimes_k K$ 
in $End_K(W)$. The idempotent $e$ defines a $K$-vector subspace
$W'= e(W)$ of $W$. Using a $K$-basis of $W'$, 
we obtain $W'\cong \bigoplus_J K$ for some index set $J$ and hence  
 $\overline\varphi(V\otimes_k I(L))\cong  (\bigoplus_J K)\otimes_K L \cong  
 \bigoplus_J L$. This implies
 $\psi(V\otimes_k I(L))\cong (\bigoplus_J k) \otimes_k I(L) = \bigoplus_J I(L) $. Adding all isotypic summands indexed by $L$ gives the result for $I$. 
\end{proof}


\begin{lem} $\mathcal P_{\mathcal T}\subset \mathcal P_{\mathcal C}$. Projective objectives in $\mathcal C={\mathcal T}^\infty$ are injective $\mathcal P_{\mathcal C} \subset {\mathcal I}_{\mathcal C}$, and there are enough projectives
in $\mathcal C$. $P$ is projective in $\mathcal C$ if and only if $P$ is a retract of a direct sum of 
projectives in $\mathcal T$.
\end{lem}

\begin{proof} 1) Projectives $P$ in $\mathcal T$ are
projective in $\mathcal C$. Simply check the lifting property for $g:P\to Y$ with respect to an epimorphism
$p:X\to Y$ in $\mathcal C$. For this replace $Y$ by $g(P)$ and $X$ by $p^{-1}(g(P))$. 
Since $g(P)\in \mathcal T$, there exists a  subobject $X'\in \mathcal T$ of $X$ such that  $p:X'\to Y'$
is an epimorphism. So the lifting property in $\mathcal C$ reduces to the lifting property in $\mathcal T$.

\medskip\noindent
2) Given $X = co \lim_i X_i$ in $\mathcal C$ (for objects $X_i\in \mathcal T$) with injective transition maps $f_i:X_i \to X_{j}$
choose  $P_i\in \mathcal P_{\mathcal T}$ together with surjections
$\pi_i: P_i \to X_i$
$$ \xymatrix@+7mm{ \tilde P_{i+1}= P_{i+1} \oplus P_{i} \ar@{->>}[r] & P_{i+1} \ar@{->>}[r]^{\pi_{i+1}} &   X_{i+1} \cr
& P_{i} \ar@{_{(}->}[ul]^{g_i\oplus id}\ar@{.>}[u]^{g_i}\ar[ur]\ar@{->>}[r]^{\pi_i} &   X_{i} \ar@{^{(}->}[u]_{f_i} \cr }$$
The transition morphisms $f_i$ can be lifted to morphisms $g_i:P_i \to P_{i+1}$. [$f_i\circ \pi_i: P_i \to X_{i+1}$ can be lifted
to a morphism $g_i:P_i \to P_{i+1}$ such that 
$\pi_{i+1}\circ g_i = f_i\circ \pi_i$,
since $P_i$ is projective.] Now replace $P_{i+1}$ by $\tilde P_{i+1} = P_{i+1} \oplus P_{i}$ and  
$\pi_{i+1}$ by $\tilde\pi_{i+1} = \pi_{i+1}\circ proj_1$ and $g_i$
by $\tilde g_i= g_i\oplus id_{P_i}$. Repeating this inductively all $\tilde g_i$ become monomorphisms between
finite dimensional projective modules. Its limit $P= co\lim_i \tilde P_i$
is isomorphic to a direct sum 
of finite dimensional projective modules. Choose a well ordering.
The image of $\tilde P_i$ in $\tilde P_{i+1}$ is isomorphic to $\tilde P_i$, hence injective.
Therefore it has a complement $P'_{i+1}$, again finite dimensional
projective and injective. If $P_0=0$, then $P= \bigoplus P'_i$;
and we have constructed a surjection $P \to X$. Direct sums of projective objects are
projective, hence $P$ is projective and $\mathcal C$ has enough projectives.

\medskip\noindent
3) Now all $P'_i$ are projective and in $\mathcal T$, hence also injective. A direct sum of injective objects in the category of $A$-comodules is injective. This proves that
$P$ is also injective. Now suppose $X$ was a projective object. Then $X$ is a summand of $P$, since
$P$ surjects onto $X$ by our construction.  Hence $X$ is a summand of an injective object, hence itself
injective.
\end{proof}

Since any $I\in \mathcal I_{\mathcal C}$ is a direct sum of finite dimensional injective objects (and each of them is in $\mathcal I_{\mathcal T}= \mathcal P_{\mathcal T}$), any injective object is a direct sum of projective objects. A direct sum of projectives is projective. Hence $\mathcal I_{\mathcal T} \subset \mathcal P_{\mathcal T}$. Together with the previous lemma this gives

\begin{cor} If $\mathcal T$ is a Frobenius category, then $\mathcal C$ is a Frobenius category.
\end{cor}


\subsection{The stable model structure} The assumption for the existence of the stable model structure  are satisfied (see lemma \ref{lemma-stable-cat}): Let $J$ be the set of isomorphism classes of indecomposable
injective $A$-comodules (these are all finite dimensional) (or better the morphisms $0 \to Inj$). Let $I$ be the set of isomorphism classes of monomorphisms $i:A\to B$ between finite dimensional $A$-comodules. Then
$I$ and $J$ define the cofibrantly generated stable model structure.  We will later show (in a slightly more general context) that this stable model structure is a monoidal model structure.


\subsection{Induced model structure}\label{sec:induced-model-structure} 

We now specialize the construction of the induced model structure of subsection \ref{sec:induced-1} to the comodule category $\mathcal C = \mathcal{T}^{\infty}$. We assume in this section that the comodule categories are Frobenius categories. Let $\pi: A\to B$ be a surjective homomorphism of supercommutative Hopf algebras. It corresponds to an embedding $H \subset G$ of supergroups. This embeddings induces a restriction functor $Res: Rep(G) \to Rep(H)$ which is an exact faithful tensor functor. We discuss its left and right adjoint (called coinduction and induction) and their extensions to the ind-category.

\subsection{Frobenius extensions and adjunctions}

We review some facts about Frobenius extensions following \cite{Bell-Farnsteiner}. Let $S \subset R$ be a subring and $\alpha$ an automorphism of $S$. If $M$ is an $S$-module, let $_{\alpha}M$ denote the $S$-module with action $s * m = \alpha(s)m$. Then $Hom_S(R, _{\alpha}S)$ is the set of additive maps $f:R \to S$ such that $f(rs) = \alpha(s)f(r)$ for all $s \in S, r \in R$. It is an $(R,S)$-bimodule via $(r \cdot f \cdot s) = f(xr)s$.

\begin{definition} $R$ is an $\alpha$-Frobenius extension of $S$ if 
\begin{enumerate}
\item $R$ is a finitely generatedd $S$-module and
\item there exists an isomorphism $\varphi:R \to Hom_S(R, _{\alpha}S)$ of $(R,S)$-bimodules.
\end{enumerate}
The Frobenius extenions is said to be free if $R$ is a free $S$-module.
\end{definition}

Now let $\mathfrak{g} = \mathfrak{g}_0 \oplus \mathfrak{g}_1$ be a finite dimensional Lie superalgebra over  $k$ and let $\mathfrak{h} = \mathfrak{h}_0 \oplus \mathfrak{h}_1$ be a subalgebra with $\mathfrak{g}_0 \subset \mathfrak{h}$ (and hence $\mathfrak{g}_0 = \mathfrak{h}_0$).

\medskip\noindent
The map $f: \mathfrak{h}_0 \to \mathfrak{gl}(\mathfrak{g}/\mathfrak{h})$ defined by $f(a)(y + \mathfrak{h}) = [a,y] + \mathfrak{h}$ for $a \in \mathfrak{h}, y \in \mathfrak{g}$ is a homomorphism of Lie superalgebras. We also define the linear functional \[ \lambda(a) = - str(f(a))\] which vanishes on $[\mathfrak{h},\mathfrak{h}] + \mathfrak{h}_1$. There is a unique automorphism $\alpha$ of $\mathcal{U}(\mathfrak{h})$ such that \[ \alpha(a) = \begin{cases} a + \lambda(a) 1 & \text{ for } a \in \mathfrak{h}_0 \\ (-1)^n a & \text{ for } a \in \mathfrak{h}_1 \end{cases} \] where $n$ is the codimension of $\mathfrak{h}$ in $\mathfrak{g}$.

\begin{thm} (\cite[Theorem 2.2]{Bell-Farnsteiner}) \label{thm:frob-ext} Let $\mathfrak{h} \subset \mathfrak{g}$ be a sub Lie superalgebra of codimension $n$ that contains $\mathfrak{g}_0$. Then the extension $\mathcal{U}(\mathfrak{g}): \mathcal{U}(\mathfrak{h})$ is a free $\alpha$-Frobenius extension.
\end{thm}

\begin{thm} (\cite[Page 96]{Nakayama-Tsuzuku}) For any $\alpha$-Frobenius extension there is a natural equivalence \[ R \otimes_S V \cong Hom_S(R, \ _{\alpha}V).\]
\end{thm}

In particular the left and right adjoint of $Res: Rep(\mathfrak{g}) \to Rep(\mathfrak{h})$ are isomorphic and therefore exact. For the special case of induction and coinduction in the $\mathfrak{gl}(m|n)$-case see \cite[Proposition 2.1.1]{Germonie} where also the automorphism $\alpha$ is described explicitely.

\medskip\noindent
If $G$ is an affine supergroup and $\mathfrak{g}$ its Lie superalgebra, we have a fully faithful tensor functor $Rep(G) \to Rep(\mathfrak{g})$ (note that $Rep(G) =  Rep(G_0,\mathfrak{g})$, the $\mathfrak{g}$-representations such that the restriction to $\mathfrak{g}_0$ is an algebraic representation of $G_0$). Consider an embedding $G_0 \subset H \subset G$ of affine supergroups. Then $\mathcal{U}(\mathfrak{h}) \subset \mathcal{U}(\mathfrak{g})$ is a Frobenius extension by theorem \ref{thm:frob-ext}. For $V \in Rep(H) \subset Rep(\mathfrak{h})$, its induction and coinduction is again algebraic, i.e. in $Rep(G)$.

\begin{cor} Let $H$ be an affine supergroup such that $G_0 \subset H \subset G$. Then $Res$ has a left and right adjoint and $Coind \cong Ind$. Both adjoints are exact.
\end{cor}

\subsection{Limit to the ind-category and the induced model structure} \label{sec:ind-passage} We recall from \cite[Proposition 6.1.9]{Kashiwara-Shapira} that the passage from a category $\mathcal{C}$ to $Ind(\mathcal{C})$ is functorial in the sense that there is a unique extension of a functor $F:\mathcal{C} \to \mathcal{C}'$ to the ind-category \[ \xymatrix{ \mathcal{C} \ar[d] \ar[r]^F & \mathcal{C}' \ar[d] \\ Ind(\mathcal{C}) \ar[r]^{IF} & Ind(\mathcal C'). } \] On objects $X = (X_i)_i$ in $Ind(\mathcal C)$ the extension $IF$ of $F$ is defined by \[ IF(X) = \varinjlim_i F(X_i).\] Recall that the morphism spaces in $Ind(\mathcal{C})$ are given by \[ Hom_{Ind(\mathcal{C})} (X,Y) = \varprojlim_i \varinjlim_j Hom_{\mathcal{C}}(X_i,Y_j).\] Then the map \[ IF: Hom_{Ind(\mathcal{C})} (X,Y) \to Hom_{Ind(\mathcal{C}')}(IF(X),IF(Y))\] is given by \[ \varprojlim_i \varinjlim_j Hom_{\mathcal{C}}(X_i,Y_j) \to \varprojlim_i \varinjlim_j Hom_{\mathcal{C}}(FX_i,FY_j).\]

\medskip\noindent
Therefore the functors $Res$ and $Ind \cong Coind$ extend to the ind-category of $Rep(G)$ with the same adjunction properties. Since $\varinjlim$ is an exact functor, the extended functors are exact as well. Using the equivalence between $Rep(G)$ and $Comod(A)$ for $A = k[G]$ we get the same adjunction properties for the Hopf algebra quotient $A \to B$ corresponding to $H \subset G$. We use the notation \begin{align*} U & = I Res: Rep(G)^{\infty} \to Rep(H)^{\infty}, \\ F & = I Coind: Rep(H)^{\infty} \to Rep(G)^{\infty}\end{align*} and likewise for the corresponding functors on the comodule categories.

\begin{cor} Let $H$ be an affine supergroup such that $G_0 \subset H \subset G$ and the associated Hopf algebra quotient $A \to B$.  For $\mathcal{C} = Comod(A)$ and $\mathcal{D} = Comod(B)$ we have \[ Hom_{\mathcal C}(FX,Y) = Hom_{\mathcal D}(X,UY) \ \] where $U$ and $F$ are exact and $U$ is faithful.
\end{cor}

\noindent
Now let $H \subset G$ be the pair of affine supergroups with Hopf algebra quotient $A \to B$.

\begin{definition} \label{def-frob-pair} We call $(H,G)$ as well as the pair of corresponding supercommutative Hopf algebras $(A,B)$  a \textit{Frobenius pair} if $G_0 \subset H \subset G$ and the categories $\mathcal{C} = Comod(A)$ and $\mathcal{D} = Comod(B)$ are Frobenius categories.
\end{definition}

In particular for any Frobenius pair the functor $Res: Rep(G)^{\infty} \to Rep(H)^{\infty}$ has a left and right adjoint which are isomorphic. By definition any Frobenius pair $(H,G)$ defines a pair of comodule categories $(\mathcal{C},\mathcal{D})$ which is a Frobenius pair in the sense of section \ref{sec:induced}. In particular the homotopy category of $(H,G)$ is defined.

\begin{example} Our main example of a Frobenius pair $(H,G)$ in part \ref{part-gl} is the following: Let $H=P$ be the  parabolic subgroup of upper triangular block matrices in the general linear supergroup $G=GL(m|n)$.
\end{example}


\begin{definition} \label{def:induced} {\bf The induced model structure}. Let $A$ be a supercommutative Hopfalgebra $A$ over a field $k$ with a quotient $A\to B$ such that $(A,B)$ is a Frobenius pair. Let $\mathcal C$ be the category of all graded comodules over $A$ endowed with the induced model category structure via the quotient $A\to B$ where we use the stable model structure on $Comod(B)$. We call this model structure on $\mathcal C$ the \it{induced model structure} attached to $(A,B)$.
\end{definition}

\subsection{Setup and further conventions} \label{sec:conventions}
In the following sections \ref{sec:clean} - \ref{sec:isogenies} we will assume that we are in the setting of definition \ref{def:induced}, i.e. $\mathcal{C} = \mathcal{T}^{\infty}$ is the category of comodules of a supercommutative Hopf algebra. However many results hold if we just assume that $\mathcal{T}$ is a $k$-linear abelian Frobenius category satisfying conditions (F) and (G). Then $\mathcal{T}$ is equivalent to the category of finite dimensional comodules of some coalgebra $A$ (see section \ref{sec:finiteness-conditions}) and we denote by $\mathcal{C}$ its ind completion. We assume then that $\mathcal{C}$ carries a model structure induced by the stable model structure from a Frobenius pair $(\mathcal{C},\mathcal{D})$ in the sense of definition \ref{def:frobenius-pair}. An exception are the results of section \ref{sec:monoidal-model-structure}: We have not verified that the model structure on $\mathcal{C}$ is a monoidal model structure in the more general situation. 

\medskip\noindent
For a $k$-linear abelian Frobenius category $\mathcal{T}$ satisfying conditions (F) and (G) we put \[ \mathcal{T}_- = \mathcal{C}_- \cap \mathcal{T}, \ \mathcal{T}_+ = \mathcal{C}_+ \cap \mathcal{T}. \] We denote the full triangulated subcategory generated by the image of $\mathcal{T}$ in $Ho\mathcal{C}$ by $Ho\mathcal{T}$.


\section{  The monoidal model structure} \label{sec:monoidal-model-structure}

\subsection{Monoidal model structures} We assume in this section that we are in the setting of definition \ref{def:induced}.

\begin{thm} \label{thm:monoidal} The induced model structure \ref{def:induced} is a monoidal model structure.
\end{thm}

In order to prove this theorem we have to verify  that the model structure on $\mathcal C$ satisfies the pushout-product axiom and the unit axiom \cite[Definition 4.2.6]{Hovey} as in section \ref{sec:monoidal-model-def}. The unit axiom - a condition on the cofibrant replacement $q:Q\one\to \one$ of the unit object - is easily verified: Indeed $f\otimes id_Z\in \mathcal W$ if and only if $U(f\otimes id_Z)=U(f)\otimes id_{U(Z)}\in \mathcal W_{\mathcal D}$. The latter  is obvious, since stable equivalence is preserved by tensor products with $id_K$ (for any $K\in \mathcal D$).
So it remains to show that the pushout-product axiom \cite[Definition 4.2.6.1]{Hovey} holds.
For this it is very convenient  (\cite[Cor.4.2.5]{Hovey}) that the model structure is cofibrantly generated by $J$ and $I$, where
 $J$ is the set of morphisms $0\to Z$ for $Z\in F(\mathcal I_{\mathcal D})$
 and $I\subset \mathcal L_+$ defined in section \ref{sec:induced-1}.

\medskip 
We start with a technical lemma.

\begin{lem} \label{induction-lemma} We have natural isomorphisms
$\varphi: F(A\otimes UB) \cong F(A)\otimes B$ in $\mathcal C$.
\end{lem}

\begin{proof} We construct $\varphi $  by adjunction. For this it suffices to construct
a homomorphism in $Hom_{\mathcal D}(A\otimes UB, UFA \otimes UB)$, namely
$ ad(A\to UFA) \otimes id_{UB} $. It is enough to show that $U(\varphi)$ is an isomorphism. Alternatively use twice adjunction
$$ Hom(F(A\otimes B),C) = Hom(A,{\mathcal H}om(UB,UC)) $$
and similarly
$$ Hom(F(A)\otimes B,C) = Hom(A,U{\mathcal H}om(B,C)) \ ,$$
so it suffices to give an isomorphism
$$ U{\mathcal H}om(B,C) \cong {\mathcal H}om(UB,UC) \ .$$ 
The proof is now basically an unravelling of the definition of the comodule structure of ${\mathcal H}om(B,C)$ as in \cite[ p. 63]{Hovey}. First notice that $B^*\hookrightarrow A^*$ becomes a subalgebra where $A^*$, $B^*$ are the duals of $A$, $B$ \cite[2.5.1]{Hovey}.
Comodules $V$ of $A$ define tame $A^*$-modules $V^*$, and this defines an equivalence of categories (see \cite[Proposition 2.5.5]{Hovey}). The functor $U$ corresponds to the restricting the tame $A^*$-module to the corresponding tame $B^*$-module. The internal Hom-module
is constructed in loc. cit. as a tame $A^*$-module $Hom_k(M,N)$ attached to tame $A^*$-modules $M,N$. From the definition of the $A^*$-action it can be verified that the defining action commutes with the restriction to the subalgebra $B^*$ (the restriction functor $U$). The formula involves a basis $b_i$ of $B$ with dual basis $b_i^*$ of $B^*$
and
\begin{itemize}
\item $\chi^*: B^* \to B^*$ (dual of the antipode $\chi:B\to B$)
\item $\Delta^*: B^* \to Hom_k(B,B^*)$ (dual of comultiplication)
\item $f:M\to N$ $k$-linear
\item $u\in B^*$, so that $ b^u_i:= \chi*(\delta^*(u)(b_i)) \in B^*$
\end{itemize}
with $uf \in Hom_k(M,N)$ defined by 
$$ (uf)(x) = \sum_i b_i^*(f( b^u_i x )) $$ 
for $x\in M$.
The similar formula for the action of $A^*$ reduces to this formula defining the action of $B^*$, if
there is a $k$-basis $a_i$ of $A$, such that the dual basis $a_i^*$ contains the dual
basis $b_i^*$ as a subset. This is possible provided there exists a
splitting $A =B \otimes_k V$ for a finite dimensional $k$-algebra $V$ (see \cite{Weissauer-semisimple}) such that the quotient map $A \to B$ is induced from an algebra map $V \to k$. This is the situation
we are considering. Indeed by \cite[Page 16]{Weissauer-semisimple} $A$ can be written as the tensor product of two supercommutative Hopf algebras \[ A = A_0 \otimes \Lambda^{\bullet}(\theta_1,\ldots \theta_s) \] where $A_0 = k[G_0] = A/J$ and the $\{ \theta_i \}$ are an $A_0$-basis of $J/J^2$. 
\end{proof}

We now prove theorem \ref{thm:monoidal}

\begin{proof} We use \cite[Corollary 4.2.5]{Hovey}. For the definition of $f \square g$ (the pushout-product) we refer to \cite[Definition 4.2.1]{Hovey}.

\medskip\noindent
1) We first verify $f \square g \subset \mathcal L\cap \mathcal W$ for $f \in I$ and
 $g\in J$ or vice versa: For $g: 0\to Z$ and $f:X\to Y$ we have
 $$ f\square g = f \otimes id_Z : X\otimes Z \to Y\otimes Z \ .$$
Now $Z \in \mathcal P_{\mathcal C} = \mathcal I_{\mathcal C}$, since $F$ maps projectives to projectives.
Hence $X\otimes Z$ and $Y\otimes Z$ are in $\mathcal I_{\mathcal D}$. Since any $f\in I$ is a monomorphism,
also the morphism $f \otimes id_Z $ is a monomorphism. This shows $f \otimes id_Z $ is a split 
monomorphims with projective kokernel. Hence it is in $\mathcal L\cap \mathcal W$. 

\medskip\noindent
2) Next we show $f\square g\in \mathcal L$ for $f,g \in I$.
For this we use the previous lemma \ref{induction-lemma}. For $f=X\to Y\in I$ and $f'=F(i):F(A)\to F(B)$ in $I$ we find
that $f\square f'$ 
$$  X\otimes F(B) \bigoplus_{X \otimes F(A)} Y \otimes F(A) \to Y\otimes F(B) $$
can be identified with the morphism
$$  F(UX\otimes B) \bigoplus_{F(UX \otimes A)} F(UY \otimes A) \to F(UY\otimes B) $$
using lemma \ref{induction-lemma}. But this last morphism is $F(U(f)\square i )$ for the pushout product $U(f)\square i$ of
 $U(f):UX\to UY$  and $i:A\to B$. Since the pushout square $u\square v$ of two monomorphisms $u,v$ is a monomorphism, $U(f)\square i$ is a monomorphism. Hence
 we see that $F(U(f)\square i )\in I \subset \mathcal L$. 
 \end{proof}
 
\begin{cor} $C\in \mathcal C_-$ implies ${\mathcal H}om(B,C)\in \mathcal C_-$.
\end{cor}

\begin{proof} Since $UC\in \mathcal I_{\mathcal D}$ also ${\mathcal H}om(UB,UC) \in \mathcal I_{\mathcal D}$
using that ${\mathcal H}om$ is right adjoint to the tensor functor.
Hence $U{\mathcal H}om(B,C) \in \mathcal I_{\mathcal D}$ by the isomorphism established in the proof of lemma \ref{induction-lemma}. 
\end{proof}

\begin{cor} $B\in \mathcal C_-$ implies ${\mathcal H}om(B,C)\in \mathcal C_-$. In particular, 
$\mathcal T_- = \mathcal{C}_- \cap \mathcal{T}$ is closed under duality.
\end{cor}

\begin{proof} $UB\in \mathcal I_{\mathcal D} =  \mathcal P_{\mathcal D}$ implies also ${\mathcal H}om(UB,UC) \in \mathcal I_{\mathcal D}$ by adjunction, since $\mathcal P_{\mathcal D}$ is a tensor ideal.
Hence $U{\mathcal H}om(B,C) \in \mathcal I_{\mathcal D}$.
\end{proof}

\subsection{Monoidal model categories and homotopy quotients} For a symmetric monoidal model category the homotopy functor 
$$ \gamma: {\mathcal C} \to Ho\mathcal C $$
is a tensor functor \cite[Page 116, Theorem 4.3.2]{Hovey}. If $\mathcal C$ is a pointed symmetric monoidal model category, then
$Ho \mathcal C$ is a closed monoidal pre-triangulated category (\cite[Page 174, Theorem 6.6.3]{Hovey}). 

\medskip\noindent
We recall some basic facts about rigid categories. If $X^{\vee}$ is a left-dual object to an object $X$ in a monoidal category in the sense of \cite[Definition 2.10.1]{EGNO}, there exist morphisms $ev_X: X^{\vee} \otimes X \to \one$ (the evaluation) and $coev_X: \one \to X \otimes X^{\vee}$ (the coevaluation) and similarly for right dual objects. A left or right-dual object is unique up to isomorphism. If $X,Y$ admit left or right duals, the dual morphism $f^{\vee}: Y^{\vee} \to X^{\vee}$ is defined. An object in a monoidal category is called rigid if it has
left and right duals. A monoidal category $\mathcal{C}$ is called rigid if every object is
rigid. For $\mathcal{C} = Comod(A)$ the notions of left and right dual coincice.

\begin{lem} \label{lem:rigidity} The homotopy category $Ho\mathcal{T}$ is rigid. 
\end{lem}

\begin{proof} If $X$ is rigid and $F$ a tensor functor, then it is easy to see that $F(X^{\vee})$ is dual to $F(X)$. Suppose $X$ is rigid. If $X = X_1 \oplus X_2$ we obtain $id_X = e + (1-e)$ for $e^2 = e$ where $e$ is the idempotent projecting to $X_1$. Then $(e^{\vee})^2 =  e^{\vee}$ where $e^{\vee}: X_1^{\vee} \to X^{\vee}$. Furthermore \[ id_{X^{\vee}} = (id_X)^{\vee} = e^{\vee} + (id_{X^{\vee}} - e^{\vee}).\] Since $X$ is rigid, we have natural adjunction morphisms \cite[Proposition 2.10.8]{EGNO} \[ Hom(A \otimes X,B) \cong Hom(A,X^{\vee} \otimes B).\] We then obtain the commutative diagram \[ \xymatrix{ Hom(A \otimes X,B) \ar[r]^{\cong} & Hom(A,X^{\vee} \otimes B) \\ Hom(A \otimes X_1,B) \ar[u] \ar[r]^{\cong} & Hom(A,X_1^{\vee} \otimes B) \ar[u] }\] where the left vertical map is induced by $e$ and the right vertical map by $e^{\vee}$. This implies that $eX = X_1$ and $(1-e)X = X_2$ are rigid. 
\end{proof}

\begin{lem} For objects $Y$ in $\mathcal T$ and $X,Z\in {\mathcal C}$ one has
$$ [X\otimes Y, Z]  \ = \ [X, Y^\vee \otimes Z] \ .$$
\end{lem}
 

\subsection{Some monoidal properties of $\mathcal{C}_{\pm}$}

The injective $A$-comodules $\mathcal I_{\mathcal C}$ define a  tensor ideal: $A\in \mathcal I_{\mathcal C} $ and $B\in \mathcal C$ implies $A\otimes B \in \mathcal I_{\mathcal C}$. Therefore the tensor product functor $- \otimes Y$ is an exact functor and preserves injectives. By the Frobenius property it also preserves projectives.

\begin{lem}$$Z\in {\mathcal I_{\mathcal C}} \quad  \Longrightarrow \quad {\mathcal  H}om(Y, Z)\in \mathcal I_{\mathcal C} \ $$ and $$ {\mathcal  H}om(Y, Z)\in \mathcal P_{\mathcal C}$$ for $Y\in \mathcal P_{\mathcal C}$
\end{lem}

\begin{proof} Since the
functors $Hom_{\mathcal C}(X\otimes Y,Z)$  and $Hom_{\mathcal C}(X,{\mathcal  H}om(Y, Z))$
are right exact in $X$, using projective resolutions of $X$ in $\mathcal C$, we obtain  
$$ Ext_{\mathcal C}^i(X\otimes Y,Z) \cong  Ext^i_{\mathcal C}(X,{\mathcal  H}om(Y, Z)) \ . $$ 
Now  $Z\in \mathcal I_{\mathcal C}$ is equivalent to $Ext^1({\mathcal C},Z)=0$. By the last  formula this implies $$Z\in {\mathcal I_{\mathcal C}} \quad  \Longrightarrow \quad {\mathcal  H}om(Y, Z)\in \mathcal I_{\mathcal C}\ .$$
An analogous argument gives the second statement.
\end{proof}

\begin{lem} ${\mathcal H}om(\mathcal C_+, \mathcal C_-) \subset \mathcal I_{\mathcal C}$. 
\end{lem}

\begin{proof} For
$X\in\mathcal C$ and fixed $X_\pm\in \mathcal C_\pm$  we have $Ext^1(X,{\mathcal H}om(X_+,X_-))=
Ext^1_{\mathcal C}(X\otimes  X_+,X_-) \subset Ext^1({\mathcal C_+,\mathcal C_-})=0$.
\end{proof}

\begin{lem} $\mathcal C_{\pm}$ are tensor ideals
\[ X \in {\mathcal C_{\pm}}, Y \in {\mathcal C} \ \Longrightarrow\ X\otimes Y \in \mathcal C_{\pm}.\]
\end{lem}

\begin{proof}
For $X \in \mathcal{C}_-$ this follows from $U(X\otimes Y)=UX \otimes UY \in \mathcal I_{\mathcal D} \otimes D \subset \mathcal I_{\mathcal D}$. The statement about $\mathcal{C}_+$ follows from $X\in {\mathcal C_- \Rightarrow H}om(Z,X)\in \mathcal C_-$ and
$Ext^1_{\mathcal C}(Y\otimes  Z,{\mathcal C_-})=Ext^1(Y,{\mathcal H}om(Z,{\mathcal C_-}))=0$.
\end{proof}

\begin{lem} \label{lem:dim0} Suppose a full subcategory  $\mathcal C' \subset \mathcal C$ is a tensor ideal, closed under retracts such that $\mathcal C'$ is not quasi-equivalent to $\mathcal C$. Then for rigid objects $X\in \mathcal C'$ the dimension $\dim(X)$ vanishes.
\end{lem}

\begin{proof}  $\dim(X)\neq 0$ implies that the maps $coev$ and $\chi(X)^{-1}ev$ defines
a split summand $\one \subset X\otimes X^\vee$. If $X\in \mathcal C'$, then also $X\otimes X^\vee
\in \mathcal C'$, hence also any direct summand is in $\mathcal C'$. But $\one\in \mathcal C'$ would imply
$X = \one \otimes X \in \mathcal C'$ for all $X\in \mathcal C$. This proves the claim.
\end{proof} 

\begin{lem} Morphisms from $\mathcal C_-$ to $\mathcal  T_+$ are stably equivalent to
zero if ${\mathcal T_+}^\vee = {\mathcal T}_+$. \
\end{lem}

\begin{proof} For $X_+\in \mathcal T_+$ suppose $X_+^\vee = {\mathcal H}om(X_+,\one) \in \mathcal T_+$.
For $X_+\in \mathcal C_+$ and $$  \varphi: X_- \to X_+ $$
the associated morphism $\Phi= ev\circ (\varphi\otimes X_-^\vee)$
$$ \Phi: X_-\otimes X_+^\vee \to \one \ $$
factorizes uniquely over the projective hull $p:P(\one)\to \one$, since
$X_-\otimes X_+^\vee \subset \mathcal I_{\mathcal C} = \mathcal P_{\mathcal C}$. Notice that here we use that $\one$
(the class of the trivial comodule) is simple.
Hence $\varphi$, which is the composite
$ X_- \to X_+\otimes X_+^\vee \otimes X_+ \to P(\one)\otimes X_+ \to X_+ $,
factorizes over the projective object $P(\one)\otimes X_+$.
\end{proof}




\section{  Clean decompositions} \label{sec:clean} 

We retain the conventions of section \ref{sec:conventions}.

\begin{definition}
An object $X$ of $\mathcal C$ is \emph{clean} if it does not contain an injective subobject. 
\end{definition}

\begin{lem}Suppose $N,M$ are clean objects.
Then $M\oplus N$ is clean.
\end{lem}

\begin{proof} Suppose $I$ is an injective summand of $M\oplus N$, hence a direct factor.
$I=\bigoplus_\nu I_\nu$ decomposes into a direct sum of indecomposable finite dimensional injective objects $I_\nu$. Hence we can assume that $I$ is the
the finite dimensional indecomposable injective hull of a simple object $S$. Then $S$ injects into $M_0\oplus N_0$ for some finite dimensional subcomodules $M_0\subset M$ and $N_0\subset N$. Now $M_0\oplus N_0 = I \oplus K$. Decomposing $M_0$ and $N_0$ into a direct sum of indecomposables then allows to apply the Krull-Schmidt theorem. It shows that $I$
must be isomorphic to one of the composition factors of $M_0$ or $N_0$, say of $M_0$. Hence, being injective
and projective, $I$ must be a direct summand of $M_0$. Hence $I\subset M$ contradicting our assumptions.
\end{proof}

\begin{lem} Any object $M\in \mathcal C$
decomposes into a direct sum $M= I \oplus N$, where $I$ is injective and $N$ 
is clean, i.e. does not contain an injective subobject. For two such decompositions
 $M=I\oplus N = I'\oplus N'$ the morphisms $\beta,\gamma$ are isomorphisms
$$ \xymatrix{  &   & 0 &  &  \cr
&  & I' \ar[u] & &   \cr
 0 \ar[r] &  I \ar[ur]^\beta_\sim\ar[r] & M \ar[u]\ar[r] & N \ar[r] & 0  \cr
 &   & N' \ar[u]\ar[ur]_\gamma^\sim & &   \cr
 & & 0 \ar[u] & &  \cr} $$
Finally, $M=I\oplus N$ and $M' = I' \oplus N'$ in $\mathcal C$ are stably
equivalent if and only if their clean components $N$ and $N'$ are isomorphic in $\mathcal C$.
\end{lem}

\begin{proof}
Existence. Any $M\in \mathcal T$ decomposes by Krull-Schmidt into a finite direct sum of indecomposable objects and the claim follows easily. In general $M =co\lim M_i$ for $M_i\in \mathcal T$ with monomorphic transition morphisms. Since injectives are projective, the argument above shows 
that clean decompositions $M_i = I_i \oplus N_i$ can be chosen such that the transition morphism
map the injective subobjects $I_i$ into the chosen injective subobject
using transfinite induction. Hence $I=co\lim I_i$ is a subobject of $M$
with quotient $M/I \cong N= co\lim N_i$. Since $I$ is injective, we get $M\cong I\oplus N$
and $N$ is clean. If $N$ contains a nontrivial injective object, then in particular an indecomposable
injective object. Therefore it suffices that $N$ does not contain an injective object $I$ from $\mathcal T$.
But this is obvious, since such an object $I$ must be contained in some $N_i$. This is impossible, since all $N_i$ are clean.

Suppose $I\oplus N = M = I' \oplus N'$ are two decompositions. 
Let the projections from $M$ to $N'$ and $I'$ denote $\alpha$ and $\beta$.
Any injective object decompses as $I=\bigoplus_\nu I_\nu$ for indecomposable injectives $I_\nu\in \mathcal T$.
The socle of each $I_\nu$ is a simple object $L_\nu \in \mathcal T$.
Hence if $\alpha(soc(I_\nu))\neq 0$, this implies that $\alpha: I_\nu \to N'$
is injective. But this is impossible, since $N'$ is clean, and implies $\alpha(soc(I_\nu))=0$.
Since this holds for all summands $I_\nu$, it implies $\alpha(soc(I'))=0$.
Therefore $\beta$ must be injective on $soc(I)$, and hence $\beta: I\to I'$
is injective, since otherwise $0\neq soc(Kern(\beta)) \subset soc(I)$ would be annihilated
by $\beta$. On the other hand $R'=Kern(\beta)$ which forces $R'\cap \beta(I)=0$. 
By the snake lemma we get an exact kokernel sequence
$$ 0 \to R' \to M/I \to I'/\beta(I) \to 0 \ .$$
Since $M/I = R$ and since $I'/\beta(I)$ is injective, this implies
$R \cong R' \oplus (I'/\beta(I))$. Accordingly $R'$ is clean and we obtain
$R\cong R'$ and $I'/\beta(I)=0$. Therefore $\beta: I \cong I'$   
must be an isomorphism. The last assertion is obvious. 
\end{proof}




\section{ Construction of cofibrant replacements} \label{sec:cofibrant}

Suppose $\mathcal{T}$ is an abelian $k$-linear Frobenius category satisfying $(F)$ and $(G)$ as in section \ref{sec:conventions} and $\mathcal{C} = \mathcal{T}^{\infty}$. Recall that $\mathcal{C}$ is equipped with an induced model structure coming from a Frobenius pair. In particular we have two abelian model categories $\mathcal C$ and $\mathcal D$ 
with a Quillen adjuntion given by $F$ and $U$. Let
$\mathcal C$ be cofibrantly generated by $J=F({\mathcal F_D})$ and $I=F({\mathcal I_{\mathcal D}})$.

\begin{definition} \label{def:weights} Let ${\mathcal C}^{\leq w}$ and ${\mathcal D}^{\leq w}$ be full subcategories of $\mathcal C$ 
and  $\mathcal D$ which are  
closed under limits and colimits, hence in particular under pushouts. Assume the following two conditions:
\begin{enumerate}
\item \it The restriction functor $U: \mathcal C \to  \mathcal D$
satisfies $U: {\mathcal C}^{\leq w} \to {\mathcal D}^{\leq w}$. 

\item \it The induction functor $F:\mathcal D \to \mathcal C$ satisfies
$F: {\mathcal D}^{\leq w} \to {\mathcal C}^{\leq w}$.
\end{enumerate}
Then we say that $(\mathcal C$, $\mathcal D$, $U$, $F$,$\leq$) forms a \emph{Quillen adjunction with weight truncation}.
\end{definition}

\begin{remark} Consider the category $\mathcal T$ of finite dimensional representations of a
(connected) algebraic supergroup $G$ with reductive $G_0$ over an algebraically closed field
$k$ of characteristic $char(k)=0$.  Let $\mathcal C$ be its ind-category. There is a notion of weights 
with an ordering $\leq $ defined between weights. Then we can define ${\mathcal T}^{\leq w}$ and ${\mathcal C}^{\leq w}$ to be the full subcategories
of objects whose simple subquotients $L(\lambda)$ all satisfy $\lambda \leq w$. 
This ordering need not be  the usual weight ordering, but could be quite arbitrary.
\end{remark}

\bigskip\noindent
{\it Weight truncation}. Given $M\in \mathcal C$ there exist objects $M_{\leq w}$ and $M^{\leq w}$ in ${\mathcal T}^{\leq w}$, such that
$M_{\leq w}$ is the maximal subobject of $M$ in ${\mathcal T}^{\leq w}$ and $M^{\leq w}$ is the maximal quotient object
in ${\mathcal T}^{\leq w}$. There exists an obvious morphism $can: M_{\leq w} \to M^{\leq w} $
functorial in $M$. These constructions are functorial in the following sense:
For a morphism $f:M\to N$ the composed morphism $M\to N\to N^{\leq w}$ uniquely factorizes
over $M\to M^{\leq w}$. The induced morphism $M^{\leq w} \to N^{\leq w}$ will again be denoted $f^{\leq w}$
by abuse of notation
$$ \xymatrix@+3mm{  M_{\leq w} \ar[d]^f \ar@{^{(}->}[r] & M \ar[d]^f \ar@{->>}[r] & M^{\leq w} \ar[d]^{f^{\leq w}} \cr
N_{\leq w} \ar@{^{(}->}[r] & N \ar@{->>}[r] & N^{\leq w} \cr }$$
If $i$ is a monomorphism, $i^{\leq w}$ need not be a monomorphism.

\begin{thm} \label{thm:cofibrant} Any morphism $p$ in ${\mathcal C}^{\leq w}$ can be factorized  into a morphism in $\varphi \in \mathcal L$ and a morphism in $\psi \in\mathcal R\cap \mathcal W$, such that $\varphi:X\to Z$
where $Z$ is a direct sum of an injective object and an object in ${\mathcal C}^{\leq w}$. Every object in ${\mathcal C}^{\leq w}$ has a cofibrant replacement $Z$ of this form.
\end{thm}

\begin{proof} Step 1). Recall that $\mathcal C$ is cofibrantly generated by $J=F(J_{\mathcal D})$ 
and $I=F(I_{\mathcal D})$. For a morphism $p:X\to Y$ between
objects in ${\mathcal C}^{\leq w}$ we will construct a factorization
$p = \psi \circ \varphi$ with $\varphi \in \mathcal L$ and $\psi\in \mathcal R\cap \mathcal W$
for $\varphi: X\to Z$ such that
$$  Z \in {\mathcal C}^{\leq w} \ .$$   
Fix $p:G^0=X\to Y$. To each $i\in I$ corresponds  
a morphism $$ i: A_i \hookrightarrow B_i $$
in $\mathcal D$. Let $\nu_i: A_i \hookrightarrow I(A_i)$ denote the injective hull.
Let $I^{\leq w} \subset I$ be the subset of all monomorphisms 
$i:A_i\to B_i$ in $\mathcal D$, where $A_i \in {\mathcal D}^{\leq w}$. 
For $i$ consider the set $S(i)$ of all morphisms $f,g$ in $\mathcal C$, such that
 $(f,g)$ makes the following diagram 
commutative $$ \xymatrix@+5mm{ F(A_i) \ar@{^{(}->}[d]_{F(\nu_i\oplus i^{\leq w})}\ar[r]^f & G^0 \oplus A_i\ar[d]^{p,0} \cr
F(I(A_i)\oplus B_i^{\leq w}) \ar[r]^-{g} & Y \cr} $$ for
$F(i^{\leq w}):F(A_i^{\leq w})\to F(B_i^{\leq w})$ and $p:G^0=X\to Y$. 
Every $f=(f_X,f_{A_i})$ has two components, where $f_{A_i}$ can be chosen
unconditionally, for example $f_{A_i}=id_{A_i}$ or $f_{A_i}=0$.

\medskip\noindent
Step 2). Put \[ L= \bigoplus_{i\in I^{\leq w}}\bigoplus_{(f,g)\in S(i)} F(A_i)\] and  \[I= \bigoplus_{i\in I^{\leq w}}\bigoplus_{(f,g)\in S(i)} F(I(A_i)).\]
We extend the map $p:X\to Y$ to a map $X\oplus L \to Y$ which is zero on $L$. It factorizes over the 
the pushout $G^1 \in {\mathcal C}^{\leq w}$ as indicated in the next diagram
$$ \xymatrix{ \bigoplus_{i\in I^{\leq w}} \bigoplus_{(f,g)\in S(i)} F(A_i) \ar@{^{(}->}[d]_{\bigoplus_{I^{\leq w}, S(i)} F(\nu_i\oplus i^{\leq w})} 
\ar@{^{(}->}[r]^-{\bigoplus (f)}&   G^0\oplus L \ar@{^{(}->}[d]^{i_1} \cr
\bigoplus_{i\in I^{\leq w}} \bigoplus_{(f,g)\in S(i)} F(I(A_i)\oplus B_{i}^{\leq w})  \ar[r]^-{\bigoplus g} &   G^1 \ar[d]^{p_1}\cr
 &  Y \cr}
$$
Since $F(\nu_i)$ is a monomorphism,  the left vertical map (abbreviated $F(M)\to F(N)$ in the following) is in $\mathcal L$, and therefore also 
its pushout $i_1\in \mathcal L$. From the injectivity of $i_1$ we get an exact sequence
$$  0 \to (G^0\oplus L) \to G^1 \to F(N)/F(M) \to 0 \ .$$
Hence $G^1/(X\oplus L)$, and then also $G^1/X$, is in $\mathcal C_+$. 

\medskip\noindent
On the other hand
$$ 0 \to I \to G^1 \to (  (X\oplus L) \oplus (\bigoplus B_i^{\leq w}))/ F(M) \to 0 $$
since the upper horizontal map is injective.  The right hand side $R$ is in ${\mathcal C}^{\leq w}$
by our assumptions. Thus by a clean decomposition of $R$ we get
$$ G^1 \cong N^1  \oplus I^1 \quad , \quad (N^1\in {\mathcal C}^{\leq w} \mbox{ clean }  , I^1\in \mathcal I_{\mathcal C}) \ .$$

\medskip\noindent
Step 3). Now iterate the construction by replacing $p:G^0=X\to Y$ by $p_1:G^1 \to Y$, and so on.
This gives a chain of diagrams with $i_k\in \mathcal L$ and $G^k \in {\mathcal C}^{\leq w}$
$$ \xymatrix@+5mm{ G^k  \ar[d]^{p_k}\ar@{^{(}->}[r]^{i_k} & G^{k+1} \ar[d]^{p_{k+1}} \cr
Y \ar@{=}[r] & Y \cr} $$
Put $G^\infty = co\lim_k G_k$. Then there exists a canonical factorization of $p$
$$  X=G^0 \hookrightarrow \cdots G^k \hookrightarrow G^{k+1} \hookrightarrow  ...  G^\infty  \to  Y \ .$$
The morphism $X\to G^\infty$ is in $\mathcal L$: It is a monomorphism.
Its cokernel is a sequential colimit of objects in $\mathcal C_+$ with injective transition maps. Hence its cokernel is in $\mathcal C_+$. 

\medskip\noindent
Step 4). We claim $G^\infty \cong N \oplus I $ with clean $N\in {\mathcal C}^{\leq w}$
and injective $I$. To see this we write the transition monomorphism
$G^k \to G^{k+1}$ as a matrix. Then the matrix entry
$$ \gamma: I^k \to N^{k+1} $$
is not a monomorphism. Indeed, since $I^k$ is a direct sum of irreducible
injective objects with irreducible socle and since $N^{k+1}$ is clean,
 $\gamma(soc(I^k))=0$. Thus the matrix entry
$$ \alpha: I^k \to I^{k+1} $$
must satisfy $\alpha: soc(I^k) \hookrightarrow I^{k+1}$. 
This forces $$ \alpha: I^k \hookrightarrow I^{k+1} \ .$$
Hence we may write $I^{k+1}= \alpha(I^k)\oplus I'$. Define $$ \tilde\gamma: I^{k+1} \to \alpha(I^k) \cong I^k \to N^{k+1} \ $$
and the isomorphism
$$ \rho = \begin{pmatrix}  id_{I^{k+1}} & -\tilde\gamma \cr 0  &  id_{N^{k+1}} \cr \end{pmatrix}\ .$$
Then $\rho\circ i_k$ maps $I^k$ into $I^{k+1}$. 
Replacing $i_{k+1}$ by $i_{k+1}\circ \rho^{-1}$, we may therefore assume
$$ i_k : I^k  \hookrightarrow I^{k+1} \ .$$
In the limit we obtain an exact sequence
$$ 0 \to co\lim_k I^k \to co\lim_k G^k \to co\lim_k N_k \to 0 \ .$$
Since $I=co\lim_k I^k$ is injective, and since $N = co\lim_k N^k$ is clean
and in ${\mathcal C}^{\leq w}$, our claim follows. 

\medskip\noindent
Step 5). We want to show that $p_\infty: G^\infty \to Y$ is in $\mathcal R\cap W$.
For this consider a possible diagram  for some $j \in I$ (i.e. a monomorphism
$j:A_j\to B_j$ in $\mathcal D$)

\medskip\noindent
Diagram 1:
$$ \xymatrix@+5mm{ F(A_j) \ar[d]_{F(j)}\ar[r]^f & G^\infty \ar[d]^{p_\infty} \cr
F(B_j) \ar@{.>}[ru]^{?}\ar[r]^g & Y \cr} $$
We have to show that there exists a lifting $h$ in order to prove $p_\infty\in \mathcal R\cap W$. 
Since each $F(A_j)$ is finite dimensional, the morphism $f: F(A_j) \to
G^\infty$ factorizes in the form $f= p_k \circ f'$ over some $G^k$. This gives the left square of the next diagram
$$ \xymatrix@+5mm{ F(A_j) \ar[d]_{F(j)}\ar[r]^{f'} & G^k \ar[d]^{p_k} \ar[r]^{i_k} & G^{k+1} \ar[d]^{p_{k+1}} \cr
F(B_j) \ar@{.>}[rru]\ar[r]^g & Y \ar@{=}[r] & Y \cr} $$
To find the desired lift to $G^\infty$, it would suffice to find a lifting $h: F(B_j) \to G^{k+1}$ of the outer diagram. To discuss this we switch sides to obtain the new diagram
$$ \xymatrix@+5mm{ F(A_j) \ar[d]_{F(j)}\ar[r]^{f'} & G^k \ar[d]^{i_k} \ar[r]^{p_k} & Y \ar@{=}[d] \cr
F(B_j) \ar@/_10mm/[rr]_g\ar@{.>}[r]^h & G^{k+1} \ar[r]^{p_{k+1}}& Y \cr} $$
To construct such an $h$, making this diagram commutative, amounts to construct a morphism $$ h: F(B_j) \to G^{k+1}= \Bigl( G^k \oplus L^k \oplus F(N^k)\Bigr) / F(M^k)
 $$
 
\medskip\noindent
Step 6) ({\it Weight factorization}). Any morphism $f:F(A)\to X\in {\mathcal C}^{\leq w}$ corresponds to a morphism $A \to UX$.
Since $UX\in {\mathcal D}^{\leq w}$, this morphism uniquely factorizes over the quotient $A\to A^{\leq w}$. Hence by adjunction $f:F(A) \to X$ canonically factorizes in the form
$$ \xymatrix{ f: F(A) \ar@{.>}[r] & F(A^{\leq w}) \ar[r] & X }\ ,$$
and similar for $B$.

\medskip\noindent
Step 7). By the weight factorization property, and since $G^\infty$ and $Y$ are in ${\mathcal C}^{\leq w}$, 
diagram 1 factorizes over the following

\medskip\noindent
Diagram 2:
$$ \xymatrix@+5mm{ F(A_j^{\leq w}) \ar[d]_{F(j^{\leq w})}\ar[r]^f & G^\infty \ar[d]^{p_\infty} \cr
F(B_j^{\leq w}) \ar@{.>}[ur]^{h'}\ar[r]^-{g} & Y \cr} $$
Hence a lift $h'$ in diagram 2 provides us with a lift
in our original diagram 1. This reduction step allows us to assume
$$ A_j = A_j^{\leq w} $$
and $j$ by $j^{\leq w}:A_j=A_j^{\leq w} \to B_j^{\leq w}$ without restriction of generality.
Since this new $j^{\leq w}$ need not be a monomorphism any longer, we now replace diagram 2 by the more convenient

\medskip\noindent
Diagram 3: $h=(f,h')$
$$ \xymatrix@+8mm{ F(A_j^{\leq w}) \ar[d]_{F(\nu_j \ \oplus\ j^{\leq w})}\ar[r]^f & G^\infty \ar[d]^{p_\infty} \cr
F(I(A_j^{\leq w})\oplus B_j^{\leq w}) \ar@{.>}[ur]^h\ar[r]^-{g\circ pr_2} & Y \cr} $$
{\it Claim}. Any lift $h$ in diagram 3 gives a lift $h'$ of diagram 2.

\medskip\noindent
{\it Proof of the claim}. The lift $h=(h_1,h_2)$ is given by a pair of morphisms $h_1: F(I(A_j))\to G^\infty$
and $h_2: F(B_j^{\leq w})\to G^\infty$ such that
\begin{itemize}
\item $p_\infty\circ f = g \circ F(j^{\leq w})$ 
\item $p_\infty\circ h_2 = g$
\item $ p_\infty \circ h_1 =0$
\item $ h_2\circ F(j^{\leq w}) = f - h_1\circ F(\nu_j)$.
\end{itemize}
The desired lift $h'$ should satisfy
\begin{itemize}
\item $p_\infty\circ f = g \circ F(j^{\leq w})$ 
\item $p_\infty\circ h' = g$
\item $ h' \circ F(j^{\leq w}) = f $.
\end{itemize}
For this we put
$ h'  =  h_2 + \lambda$,
where $\lambda: F(B_j^{\leq w})\to G^\infty$ is chosen 
such that $$p_\infty\circ \lambda=0$$
$$ \lambda \circ F(j^{\leq w}) =  h_1 \circ F(\nu_j) \ .$$
To construct $\lambda$ by adjunction, we construct the corresponding
$\Lambda: B_j^{\leq w} \to UG^\infty$ via $\lambda = adj \circ F(\Lambda)$.  
For the definition of $\Lambda$
consider the next diagram
$$ \xymatrix@+5mm{ A_j^{\leq w}\ar@/_15mm/[dd]_{j^{\leq w}} \ar[r]^{\nu_j} \ar@{^{(}->}[d]^j& I(A_j^{\leq w}) \ar[d]\ar[r]^{H_1} &  UG^\infty \ar@{=}[d] \cr
 B_j  \ar@{->>}[d] \ar@{.>}[ur]^{\exists u}\ar[r] & I(A_j^{\leq w})^{\leq w} \ar@{.>}[r]^{\exists v} &  UG^\infty \ar@{=}[d] \cr
  B_j^{\leq w} \ar@{.>}[rr]^\Lambda & &  UG^\infty \cr} $$
where $u$ exists, since $I(A_j)$ is injective. The existence of the morphisms $v,\Lambda$ follows from
weight factorization, if $G^\infty$ is in ${\mathcal C}^{\leq w}$ and then $UG^\infty \in {\mathcal D}^{\leq w}$. Since
we only know that $G^\infty = N^\infty \oplus I^\infty$ with $N^\infty \in {\mathcal C}^{\leq w}$,
we still have to show that 
we may assume $G^\infty= N^\infty$ without restriction of generality.
For this we should temporily return to the beginning of step 5)
as we do in the next step.

\medskip\noindent
Step 8) (Replacing $G^\infty$ by $N=N^\infty$). Consider the lifting condition
for a diagram
$$    \xymatrix@+5mm{ F(A_j) \ar[d]^j\ar[r]^{(f_N,f_I)} &   N \oplus I \ar[d]^{p_N+p_I}\cr
F(B_j) \ar@{.>}[ur]_{(h_N,h_I)}\ar[r]^g &   Y  \cr} $$
Since $I$ is injective and $j$ is a monomorphism, we can choose once and for all $h_I: B_j
\to I$. The lifting condition then is equivalent to
$$ h_N \circ F(j) = f_N \quad , \quad p_N \circ h_N = g - p_I\circ h_I \ .$$
This allows us to replace our lifting problem for $G^\infty$ to a lifting property
with $G^\infty$ replaced by $N^\infty$ by using the new diagram
$$    \xymatrix@+5mm{ F(A_j) \ar[d]^j\ar[r]^{f_N} &   N  \ar[d]^{p_N}\cr
F(B_j) \ar@{.>}[ur]_{h_N}\ar[r]^{g-p_I\circ h_I} &   Y  \cr} $$
instead of the old one.

\medskip\noindent
Step 9). Now we combine the results from step 5) -step 8).
We can assume $A_j=A_j^{\leq w}$ and we can replace $j:A_j \to B_j$
by the redundant morphism $(\nu_j,j^{\leq w}): A_j \to I(A_j)\oplus B_j^{\leq w}$.
This defines an element in $j\in I^{\leq w}$. 
So it suffices now finally to find a lift for very special morphisms in $I^{\leq w}$ as in the next

\medskip\noindent
Diagram 4:
$$ \xymatrix@+5mm{ F(A_j) \ar[d]_{F(\nu_j\oplus j^{\leq w})}\ar[r]^f & G^\infty \ar[d]^{p_\infty} \cr
F(I(A_j)\oplus B_j^{\leq w}) \ar@{.>}[ur]^{?}\ar[r]^-{g} & Y \cr} \quad , \quad  j \in I^{\leq w}\ .$$
instead of diagram 1. As explained in step 5) the morphism $f$ factorizes over some
$f':F(A_j) \to G^k$, and it suffices to find the desired lift 
$h:F(I(A_j)\oplus B_j^{\leq w}) \to G^{k+1}$ on the next level $k+1$ of the construction of $G^\infty$.
 
\medskip\noindent 
The data from diagram 4 give a triple $(j,f,g)$ defining one of the summands that appear
in the construction of the pushout $G^{k+1}$. Notice $j\in I^{\leq w}$ and
choose $f=(f',0): F(A_j) \to G^k \oplus L^k$.
This being said,
we map $F(I(A_j)\oplus B_j^{\leq w})$  into the summand corresponding to the index $(j,f,g)$ by the identity map, and we map it 
into all other summands by zero. Thus we obtain a commutative diagram 
$$ \xymatrix@+2mm{ F(A_j) \ar@/^14mm/[rr]^{f}\ar[d]_{F(j)} \ar[r]^-{(j,f,g)} & \bigoplus_{i\in I^{\leq w}}\bigoplus_{(f,g)\in S(i)} F(A_i) \ar[d]_{\bigoplus_{I^{\leq w}, S(i)} F(\nu_i\oplus i^{\leq w})} \ar[r]^-{\bigoplus f}&   G^k\oplus L^k \ar[d]^{i_k} \cr
F(I(A_j)\oplus B_j^{\leq w})  \ar@/_14mm/[rr]^{h} \ar[r]^-{(j,f',g)} &  \bigoplus_{i\in I^{\leq w}} \bigoplus_{(f,g)\in S(i)} F(I(A_i)\oplus B_i^{\leq w}) \ar[r]^-{\bigoplus g} &   G^{k+1} \ar[d]^{p_k}\cr
  & &  Y \cr}
$$
Then $ i_k \circ f = h\circ F(j)$ and $p_{k+1}\circ h = g$. Composing $h$ with the morphism
$G^{k+1} \to G^\infty$ defines the lift we were looking for. 
\end{proof}

%



\section{Minimal models} \label{sec:minimal}

Suppose $\mathcal{T}$ is an abelian $k$-linear Frobenius category satisfying (F) and (G) as in section \ref{sec:conventions} and $\mathcal{C}_+ \cap \mathcal{C}_- = \mathcal{P} = \mathcal{I}$.

\subsection{Minimal cofibrant replacements} Let $\mathcal C$ be an abelian model category in  which every object is fibrant.

\begin{definition} \label{def:minimal-model} Let $\mathcal C$ be an abelian model category in  which every object is fibrant. A cofibrant replacement  $q:QX\to X $ of $X\in \mathcal C$ is called
\emph{minimal} (or \emph{minimal model}) if an element $\varphi\in End_{\mathcal C}(QX)$ is a unit, i.e. in $Aut_{\mathcal C}(QX)$, if and only if $\varphi$ is in $\mathcal W$. 
\end{definition}

\begin{lem} \label{lem:mini} Suppose $q:QX \to X$ is minimal. Then any cofibrant replacement
$q':Q'X \to X$ is of the following form: there exists a projective object $P$
and an isomorphism $\tau: Q'X \cong P \oplus QX$ such that $q'=q\circ pr_{QX}\circ \tau$.
\end{lem}

\begin{proof} Recall from Corollary \ref{model-structure-descr} that the fibrations are the epimorphisms. Suppose $q:QX\to X$ is minimal. Suppose $q': Q'X \to X$ is another cofibrant replacement. Then
by the lifting property there exist morphisms $\alpha\in \mathcal W$ and $\beta\in \mathcal W$ (not uniquely) such that
$ q = q' \circ \alpha$ and $q' = q \circ \beta$. Hence
$q = q \circ (\beta \circ \alpha)$. Notice that $\varphi= \beta \circ \alpha \in End_{\mathcal C}(QX)$
is in $\mathcal W$. If
$q:QX\to X$ is minimal, the $\varphi$ is an isomorphism of $QX$. Then 
$$  \alpha: QX \to Q'X \quad , \quad s\circ \alpha = id_{QX} $$
is a retract of $Q'X$
for  $ s=\varphi^{-1}\circ \beta $. Hence $\alpha$ is a 
monomorphism in $\mathcal W$, and $QX$ is 
a direct summand of $Q'X$
$$  Q'X = QX \oplus P $$
where $P$ is in $\mathcal C_-$. But $P$ is in $\mathcal C_+$, hence in $\mathcal P_{\mathcal C}$.
Then $P= Kern(s)$ is also in the kernel of $q':Q'X\to X$,
since  $q' = q \circ \beta = q \circ \varphi \circ s$. Hence
$$ Kern(q') = Kern(q) \oplus P \quad , \quad P \in \mathcal P_{\mathcal C} \ .$$ 
\end{proof}

\begin{remark} \label{rem:minimal} Suppose $QX = A \oplus P$ for some projective $P\neq 0$.
Then $QX \to A \to QX$ is in $End_{\mathcal C}(QX)\cap \mathcal W$, but not
in $Aut_{\mathcal C}(QX)$. Hence $QX$ is not minimal. Hence minimal models are unique
up to isomorphism.
\end{remark}

\begin{remark} Not every indecomposable object has a minimal model, see example \ref{ex:minimal} for an example in the $GL(m|1)$-case.
\end{remark}

\begin{lem} A cofibrant replacement $q:QX \to X$ is minimal if and only if $QX$ is clean (i.e. without
injective subobject). A minimal model $q:QX\to X$ is unique up to isomorphism.
\end{lem}

\begin{proof} If $QX$ is minimal, then $X$ is clean by remark \ref{rem:minimal}. If $QX$ is clean,
let $\gamma :QX \to QX$ to be an endomorphism, and assume $\gamma\in \mathcal W$. It can be factorized
$\gamma= u\circ v$, where $u\in \mathcal R\cap \mathcal W$ and $v\in \mathcal L\cap \mathcal W$.
Hence $v: QX \to M$ is a monomorphism with injective kokernel $I'$, and $u:M\to QX$ 
is an epimorphism
with kernel $I\in \mathcal C_-$.  Then $M\in \mathcal C_+$, since $QX$ and $I'$ are in $\mathcal C_+$,
This implies $I\in \mathcal C_-\cap \mathcal C_+$, hence $I$ is injective too.
$$ \xymatrix@C=1.8em@R=1.8em{  &   & 0 &  &  \cr
&  & I' \ar[u] & &   \cr
 0 \ar[r] &  I \ar[ur]^\beta\ar[r] & M \ar[u]\ar[r]^u & QX \ar[r] & 0  \cr
 &   & QX \ar[u]^v\ar[ur]_\gamma & &   \cr
 & & 0 \ar[u] & &  \cr} $$
Since $QX$ is clean, the morphisms $\beta$ and $\gamma$ are isomorphisms.
This proves $\gamma\in Aut_{\mathcal C}(QX)$.
\end{proof}

Consider an arbitrary cofibrant replacement
$$ 0 \to K \to QX \to X \to 0 $$
with $QX\in\mathcal C_+$ and $K\in\mathcal C_-$. 
We decompose $QX = I\oplus N$ into an injective $I$ and a clean
summand $N$. Notice $N,I \in \mathcal C_+$. Then there are the following exact sequences
$$ \xymatrix{ 0 \ar[r] & Q_1   \ar[r] & X \ar[r] & Q_2 \ar[r] &   0 \cr
 0 \ar[r] & N   \ar@{->>}[u] \ar[r] & N\oplus I \ar@{->>}[u]\ar[r] & I \ar[r] \ar@{->>}[u] &   0 \cr
  0 \ar[r] & K\cap N \ar@{^{(}->}[u]  \ar[r] & K \ar@{^{(}->}[u] \ar[r] & im(K)\ar@{^{(}->}[u] \ar[r] &   0 \cr} $$
  
\begin{lem} \label{lem:existence-minimal-model} Every simple object $X\in \mathcal T$ not in $\mathcal C_-$ has a minimal model in $\mathcal C$. If $(\mathcal C$, $\mathcal D$, $U$, $F$,$\leq$) forms a \emph{Quillen adjunction with weight truncation} and if the simple object $X$ has weight $\leq w$, then the minimal model
$QX$ of $X$ is in the subcategory ${\mathcal C}^{\leq w}$.
\end{lem}

\begin{proof} For the first assertion assume that $X$ is a simple object in ${\mathcal C}$ in the diagram above. Then the top horizontal exact sequence implies that either $Q_1=0,Q_2\cong X$ 
or $Q_1\cong X, Q_2=0$:

\bigskip\noindent
a) In the case $Q_1=0$ we have $N=K\cap N$, hence $K = (K\cap N)\oplus (K\cap I)$.
Thus $K\cap N \in \mathcal C_-$ is a summand of $K\in \mathcal C_-$. 
Hence $K\cap N$ and $im(K)$ are both in $\mathcal C_-$.
Secondly $q$ factorizes $ N\oplus I \to (N\oplus I)/K \cong I/im(K) \cong X$.
Since $I$ and $im(K)$ are in $\mathcal C_-$, this would imply $$ X\in \mathcal C_- \ .$$
Then $QX\in \mathcal  C_-$, hence $QX \in \mathcal C_-\cap \mathcal C_+ = \mathcal I_{\mathcal C}$.
Thus $$ X \cong 0 \quad , \quad \mbox{ in } Ho\mathcal C \ .$$
b) In the other case $Q_2=0$ we have $Im(K)=I$, hence $Im(K)$ is projective.
The projection $K\to im(K)$ therefore splits by a morphism $s:im(K)\to K$.
Then $K\cong s(im(K)) \oplus (K\cap N) \in \mathcal C_-$ again
implies $(K\cap N)\in \mathcal C_-$. On the other hand there is the
exact sequence
$$ 0 \to (K\cap N) \to N \to N/(N\cap K) \cong X \to 0 \ .$$
Since $N\in \mathcal C_+$ and $(K\cap N)\in \mathcal C_-$, this implies
that $N\to X$ is a cofibrant replacement. Since $N$ is clean,
this proves the first statement. On the other hand there exist a cofibrant replacement
$q:QX \to X$ with $QX\cong N\oplus I$, where $I$ is injective and $N$ is of weight $\leq w$ as shown for arbitrary objects in ${\mathcal C}^{\leq w}$ in section \ref{sec:cofibrant}. 
Since we know that $X$ admits a minimal model, then $q(I)=0$ and $(QX=N, q\vert_N)$ defines a minimal model, where $QX \in {\mathcal C}^{\leq w}$.
\end{proof}

\subsection{More on morphisms in $Ho \mathcal{T}$}

By theorem \ref{stable-cat} $[X,Y] = Hom(QX,Y)/\sim_{stable}$. Two maps $f:QX \to Y, \ g:QX \to Y$ are equivalent if and only if their difference factors over a projective module $P$ \[ \xymatrix{ f-g:QX \ar[r] \ar[dr] & Y \\ & P \ar[u] } \] Suppose $Y\in\mathcal T$. It suffices 
to consider any projective cover $\pi_Y: P(Y) \to Y$ instead of arbitrary $P$. Indeed, assume  $QX \to Y$ factorizes over $\pi: P\to Y$ as above. Then $\pi$ factorizes over the surjection $\pi_Y: P(Y) \to Y$ (since $P$ is projective).
If  $\lambda: QX \to P(Y)$ is surjective, then
$P(Y)$ splits and becomes a summand of $QX$. Suppose $QX$ is clean and $Y$ is simple, then the image of
$\lambda$ is contained in the radical of $P(Y)$ and hence $f=g$. 

\begin{cor} \label{thm:hom-formula} If $QX$ is clean and $Y\in \mathcal T$ is simple, then $$[X,Y]= Hom_{\mathcal C}(QX,Y)\ .$$
\end{cor}

In particular \[ [V,Y]  = Hom_{\mathcal C}(V,Y) \] for clean $V \in \mathcal{C}_+$.

\begin{cor} Suppose $X$ admits a minimal model and suppose $Y$ is simple.
Then a morphism $f:X\to Y$ becomes zero in $Ho\mathcal C$ if and only if $f$ is zero in $\mathcal C$.
\end{cor}

\begin{remark} In the setting where $(\mathcal C$, $\mathcal D$, $U$, $F$,$\leq$) forms a Quillen adjunction with weight truncation this applies if $X$ and $Y$ are simple objects in $\mathcal C$.
\end{remark}

\begin{proof}
The category $Ho\mathcal C$ is obtained as the localization of the 
stable category $\overline{\mathcal C}$ by $\mathcal R\cap W$. Hence $f$ becomes zero
if and only if there exists $s\in \mathcal R\cap W$ such that $s\circ f=0$ in the stable
category. Indeed we may even assume that $s=q$ is the cofibrant replacement
$q:QX \to X$. In other words $q\circ f=0$ in $Hom_{\overline C}(QX,Y)$. 
Now suppose that $X$ admits a minimal model, i.e. suppose
that $QX$ is clean, and suppose that $Y$ is simple. Then this amounts
to $q\circ f=0$ in $Hom_{\mathcal C}(QX,Y)$ by corollary \ref{thm:hom-formula}.
But $QX\to X$ is surjective. Hence $q\circ f=0$ implies $f=0$ in $\mathcal C$.
\end{proof}


\section{Categories with weights and vanishing theorems} \label{sec:vanishing}

\subsection{Vanishing theorem} We retain the setup of section \ref{sec:conventions}. The ind-category $\mathcal C$ decomposes into blocks ${\mathcal C}^\Lambda$.
Consider a poset structure on the set $\Lambda$ of isomorphism classes $\lambda$ of simple objects
in a block. 
For $\lambda\in \Lambda$ define full subcategories ${\mathcal C}^{\leq \lambda} \subset {\mathcal C}^\Lambda$  and similarly for $\mathcal T$ to consist of all objects
with simple subquotients $\leq \lambda$. We define ${\mathcal C}^{< \lambda} \subset {\mathcal C}^\Lambda$   and similarly for $\mathcal T$ to consist of all objects
with simple subquotients $\leq \lambda$ but $\neq \lambda$. Then by definition
$$Hom_{\mathcal C}({\mathcal C}^{< \lambda}, L(\mu))=0$$ for all $\lambda \leq \mu$.
The subcategories ${\mathcal C}^{< \lambda}$ and ${\mathcal C}^{\leq \lambda}$ are closed under limits and colimits. Suppose that the blocks $\Lambda_{\mathcal{C}}$ of $\mathcal{C}$ can be identified with the blocks $\Lambda_{\mathcal{C}}$ of $\mathcal{D}$ via the functors $F$ and $U$. Then it makes sense to assume $F: \mathcal{D}^{\leq \lambda} \to \mathcal{C}^{\leq \lambda}$ and likewise for $U$. The tuple $(\mathcal C$, $\mathcal D$, $U$, $F$,$\leq$) forms a Quillen adjunction with weight truncation.


\begin{thm} \label{thm:main} Under these assumptions  $$ Hom_{Ho\mathcal C}(QL(\mu),L(\lambda))=0 \ \ \mbox{ hence } \ \ [L(\mu),L(\lambda)]=0$$  for any irreducible objects $L(\lambda), \ L(\mu)$ for which $\mu<\lambda$.
\end{thm}

\begin{proof}  For  $X\in {\mathcal C}^{\leq \mu}$ we choose a cofibrant replacement $QX=N\oplus I$, whose
clean component $N$ is in ${\mathcal C}^{\leq \mu}$ as shown in section \ref{sec:cofibrant}. Then $Hom_{Ho\mathcal C}(QL(\mu),L(\lambda))$ is $$  Hom_{\overline{\mathcal C}}(QL(\mu),L(\lambda))= Hom_{\overline{\mathcal C}}(I,L(\lambda)) \oplus 
Hom_{\overline{\mathcal C}}(N,L(\lambda))=0 \ ,$$ 
since $I=0$ in $\overline{\mathcal C}$ and since $Hom_{\overline{\mathcal C}}(N,L(\lambda))$ 
is a quotient of $Hom_{\mathcal C}(N,L(\lambda))$, which is zero by weight reasons.
\end{proof}

\subsection{Finiteness theorems} \label{sec:finiteness-theorems}

We now discuss some finiteness theorems or conjectures under certain additional conditions. These theorems hold in our main example ($P(m|n)^+$, $GL(m|n)$). It is plausible that they all hold if  $(\mathcal C$, $\mathcal D$, $U$, $F$,$\leq$) forms a Quillen adjunction with weight truncation. 

\medskip
For a Frobenius pair ($\mathcal C, \mathcal D$) consider the following chain of functors
$$ \gamma: {\mathcal C} \to Ho{\mathcal C} = {\mathcal C_+}/\sim_{stable} .$$
Let ${\mathcal H}=Ho\mathcal T$ be the full triangulated tensor subcategory of $Ho\mathcal C$
generated by the image of $\mathcal T$ under $\gamma$.
Then there is the functor
$$ \gamma: {\mathcal T \to \mathcal H }=Ho\mathcal T \ $$
which in general is neither surjective nor injective on the set of morphisms.
We assume now the following additional assumptions on our categories with weights:

\medskip\noindent
{\bf Assumptions I}. {Let $\mathcal C$ and $\mathcal D$ be Frobenius categories with an interval finite poset $\Lambda$ of weights, such that
\begin{enumerate}[label=\textbf{A.\arabic*}]
\item \label{A-1} Every simple object $X$ has a weight $\lambda$ in this poset such that $X\in {\mathcal T}^{\leq \lambda}$ and
$Hom_{\mathcal C}(X,{\mathcal C}^{\leq \lambda'})=0$ for all $\lambda'< \lambda$
\item ${\mathcal C}^{\leq \lambda}$ is closed under extensions. 
\item For every simple object $X=L(\lambda)$ of weight $\lambda$
there exists a monomorphism $L(\lambda)\hookrightarrow  K_-(\lambda) $ in $\mathcal C$ with
$K_-(\lambda)\in \mathcal T_-$ and cokernel in ${\mathcal C}^{< \lambda}$. Here ${\mathcal C}^{< \lambda}$ denotes the full subcategory generated by the objects in ${\mathcal C}^{\leq \lambda'}$ for $\lambda'\leq \lambda$.
Let $Ho{\mathcal C}^{\leq \lambda}$ denote the full image of ${\mathcal C}^{\leq \lambda}$.
\item \label{A-4} For every simple object $X=L(\lambda)$ of weight $\lambda$
there exists an epimorphism $K_+(\lambda)\rightarrow  L(\lambda) $ in $\mathcal C$ with
$K_+(\lambda)\in \mathcal T_+$ and kernel in ${\mathcal C}^{< \lambda}$. 
\end{enumerate}

\begin{example} For the Frobenius pair ($P(m|n)^+$, $GL(m|n)$) put $K_-(\lambda) =V(\lambda)^*$ (anti Kac module) and $K_+(\lambda) =V(\lambda)$ (Kac module).
\end{example}

\begin{thm} \label{thm:translation} If assumptions \ref{A-1} - \ref{A-4} hold, then the shift functor induces a functor
$$    [1]:  Ho{\mathcal T}^{\leq \lambda} \to Ho{\mathcal T}^{< \lambda} \ . $$
\end{thm}

\begin{proof} Obviously we obtain $L(\lambda)[1] \cong cokern(L(\lambda)\to K_+(\lambda)) \in {\mathcal T}^{< \lambda'}$ for simple objects $X=L(\lambda)$ of weight $\lambda'$ in ${\mathcal T}^{\leq \lambda}$. Since
$\lambda' \leq \lambda$ this implies the claim. 
\end{proof}

\begin{thm} \label{thm:end(1)} If assumptions \ref{A-1} - \ref{A-4} hold and $X\in \mathcal T$ is simple,
$$ [X,X] \cong k\cdot id_X.$$
\end{thm}

\begin{proof} We use the exact sequence \[ \xymatrix{ 0 \ar[r] & K \ar[r] & K_+(\lambda) \ar[r] & L(\lambda) \ar[r] & 0 } \] for $K \in \mathcal{C}^{< \lambda}$. Since $\gamma: \mathcal{C} \to Ho \mathcal{C}$ is exact, it induces a distinguished triangle. We apply $Hom_{Ho \mathcal{C}}(-,L(\lambda))$ and obtain \[ \xymatrix{ [K[1],L(\lambda)] \ar[r] & [L(\lambda),L(\lambda)] \ar[r] & [K_+(\lambda),L(\lambda)] \ar[r] & [K,L(\lambda)].}\] Since $K, K[1] \in \mathcal{C}^{< \lambda}$ we obtain \[ Hom_{Ho \mathcal{C}}(L(\lambda),L(\lambda)) \simeq Hom_{Ho \mathcal{C}}(K_-(\lambda),L(\lambda)).\] Since we can suppose that $K_+(\lambda)$ is clean, the latter equals $Hom_{\mathcal{C}}(K_+(\lambda),L(\lambda))$ which is one-dimensional.
\end{proof}



\subsection{Irreducible objects in $Ho \mathcal T$}

We now add another assumption (the typicality axiom):

\begin{enumerate}[label=\textbf{T.\arabic*}]
\item \label{T-1} For simple $X=L(\lambda)$ let $P(\lambda) \to L(\lambda)$ denote the projective hull. Then there
 exists a chain of surjections $P(\lambda) \to V(\lambda) \to L(\lambda)$
with $V(\lambda) \in \mathcal T_+$ such that
$V(\lambda) \cong P(\lambda)$ implies 
$L(\lambda)\cong V(\lambda)$, hence $L(\lambda)\cong P(\lambda)$
is projective.
\end{enumerate}

We remark that a simple object $X=L(\lambda)$ becomes zero in $Ho\mathcal T$ if and only if $X\in \mathcal T_-$. 

\begin{thm} If assumption \ref{T-1} holds, a simple object $X$ becomes zero in $Ho\mathcal T$ if and only if  $X$ is projective in $\mathcal T$.
\end{thm}

\begin{proof} If $X$ is projective, then $X$ becomes zero in $\overline{\mathcal C}$ and hence zero in $Ho\mathcal T$. Conversely, if $id_X$ becomes zero in $Ho\mathcal T$, then
the epimorphism 
$$ p: V=V(\lambda) \to L(\lambda)=X$$
factorizes over a projective module. This implies 
$p=0$ in $\mathcal C$, and hence a contradiction if
$V$ is clean. If $V$ is not clean, then the typicality axiom implies $P(\lambda)\cong  V(\lambda)$, and $V(\lambda)$ is projective. The typicality axiom then furthermore implies
that $L(\lambda)$ is projective in $\mathcal C$.    
\end{proof}

\subsection{Representations of supergroups}\label{sec:further-conj}

Consider an algebraic supergroup $G$ and a subgroup $H$ such that $(H,G)$ is a Frobenius pair. We assume
that the reduced groups $G_0$ and $H_0$ are reductive, e.g. $G$ is a basic classical supergroup such as $GL(m|n)$, $OSp(m|2n)$, $P(n)$ or $Q(n)$.
Put ${\mathcal T}=Rep_k(G)$ and ${\mathcal T}_H=Rep_k(H)$ (or a related tensor category 
$Rep_k(\mu,G)$ etc.). 
Attached to the pair $(H,G)$ we consider the ind-categories $\mathcal C$ of $\mathcal T$ and $\mathcal D$ of ${\mathcal T}_H$, and the associated tensor functors
\begin{align*} \gamma: {\mathcal C} & \to Ho{\mathcal C} \\
\gamma: {\mathcal T} & \to \mathcal H = Ho\mathcal T \ .\end{align*}

Recall that $Ho\mathcal{T}$ is rigid. Hence

\begin{cor} The homotopy category $Ho\mathcal{T}$ is a $k$-linear symmetric rigid monoidal category satisfying $End(\one) = k$.
\end{cor}

\begin{cor}For $X$ in $\calT$ and $Y\in \calT_+$ we have
$[X,Y] \cong Hom_{\overline{\calT}}(X,Y)$. If $Y=V(\lambda)$ and $X$ is simple
this is equal to $Hom_{{\calT}}(X,Y)$. 
\end{cor}

\begin{proof} Use $[X,Y] \cong [Y^\vee,X^\vee] = Hom_{\overline{\calT}}(Y^\vee,X^\vee) \cong Hom_{\overline{\calT}}(X,Y)$, since $Y^\vee$ is cofibrant. If $Y=V(\lambda)$, then $Y^\vee\in \calC_+$ is clean. If $X^\vee$ is simple, hence $[Y^\vee,X^\vee] = Hom_{\calT}(Y^\vee,X^\vee) =
Hom_{\calT}(X,Y)$. 
\end{proof}

\begin{conj} \label{thm:hom-finite} If assumptions \ref{A-1} - \ref{A-4} hold, for all objects $X,Y\in \mathcal T$ 
$$ \dim_k([X,Y]) \ < \ \infty \ .$$
\end{conj}

\begin{remark} It is enough to prove the conjecture for $[L,\one]$ for irreducible $L$ whose weight is neither bigger or smaller than the weight of $\one$. Using $[X,Y] \simeq [X \otimes Y^{\vee}, \one]$ we can assume that $Y$ is trivial. We claim that it then suffices to show that $\dim_k [X,\one]$ is finite-dimensional for irreducible $X \simeq L(\lambda)$. We induct on the length of $X$ and suppose $\dim_k [X,\one] < \infty$ for $l(X) < n$ for $n \geq 2$. Let $l(X) = n$. Then embedding of the socle gives a distinguished triangle \[ \xymatrix{ soc(X) \ar[r] & X \ar[r] & X' \ar[r] & soc(X)[1] } \] with $l(soc(X)) < n$ and $l(X') < n$. If we apply the functor $[,\one]$ to this triangle, then the finite-dimensionality of $[soc(X),\one]$ and $[X',\one]$ forces $\dim_k [X,\one] < \infty$. Now consider $[L(\lambda), \one]$. For irreducible objects $[L(\lambda),L(\mu)] = 0$ if $\mu > \lambda$ by theorem \ref{thm:main} and $[L(\lambda),L(\lambda)] = k \cdot id$. Hence we assume now $\lambda > w(\one)$. Then there are finitely many weights $w'$ satisfying $w(\one) \leq w' < \lambda$. Assume now by induction that the statement holds for all $[L(w'),\one]$ with $w(\one) \leq w' < \lambda$. Then the morphism $K_+(\lambda) \to L(\lambda)$ gives a distinguished triangle in $Ho \mathcal{C}$ \[ \xymatrix{ K \ar[r] & K_+(\lambda) \ar[r] & L(\lambda) \ar[r] & K[1]. }\] We apply the functor $[,\one]$. Now $K$ and $K[1]$ are in $\mathcal{C}^{<\lambda}$ by assumption and theorem \ref{thm:translation}, and therefore $[K[1],\one]$ is finite dimensional. But $[K_+(\lambda),\one]$ is also finite dimensional since $K_+{\lambda}$ is cofibrant and clean.  
\end{remark}

\begin{lem} The conjecture holds for ($P(m|n)^+$, $GL(m|n)$).
\end{lem}

\begin{proof} Follows immediately from the explicit construction of cofibrant replacements in theorem \ref{thm:good-replacement}.
\end{proof}

Since $Ho {\mathcal T}$ is rigid,  the monoidal ideal ${\mathcal N}$ of negligible morphisms is defined (see section \ref{sec:negligible} for more details).

\begin{conj} Assume that $G$ is basic classical and $H$ satisfies $G_0 \subset H \subset G$. The quotient $Ho \mathcal T / \mathcal N$ is the semisimple representation category of an affine supergroup scheme.
\end{conj}

We prove this in the $G = GL(m|n)$ and $H = P(m|n)^+$-case in theorem \ref{thm:semisimple}.

\begin{remark} An indecomposable object $X \in Ho \mathcal T$ is in the kernel of $Ho \mathcal T \to Ho \mathcal T/\mathcal N$ if and only if $\sdim(X) = 0$ \cite{Heidersdorf-semisimple}. Suppose $\mathcal C_\pm\neq\mathcal C$ to exclude trivial cases.
Then $X\in \mathcal T_\pm = \mathcal T\cap \mathcal C_\pm$ implies $\sdim(X)= 0$ by lemma \ref{lem:dim0}.
\end{remark}

\begin{example} If $\mathcal T = \mathcal T_{m|n}$ and $\mathcal T_{\pm}$ are the tensor ideals of Kac and anti Kac modules in $\mathcal{T}_{m|n}$ respectively, then it is well-known \cite{Heidersdorf-semisimple} that the superdimension is zero for every object in $\mathcal{T}_{\pm}$.
\end{example}



\section{Isogenies} \label{sec:isogenies}

We discuss a further localization ${\mathcal H} \to  {\mathcal H}[\Sigma^{-1}] $ of the homotopy category.

\subsection{Isogenies I}
Assume $\mathcal T$ is as in section \ref{sec:conventions}.
For $\mathcal E=\mathcal C_+$ 
consider the full triangulated subcategory 
${\mathcal F}$ of objects stably equivalent to objects in $\mathcal T_+$. 
Its image $\overline{\mathcal F}$ in $\overline{\mathcal E}$ is quasi-equivalent to the image of $\mathcal T_+$
in $\overline{\mathcal E}$. Moreover, every object in ${\mathcal E}$ isomorphic in $\overline{\mathcal E}$
to an object in $\mathcal T_+$ is in ${\mathcal F}$ by definition.   

\begin{lem} $\overline{\mathcal F}$ is a thick triangulated subcategory of $\overline{\mathcal E}$.
\end{lem}

\begin{remark} Similarly the full image of $\mathcal T_-$ defines a thick triangulated subcategory of the stable category $\overline{\mathcal T}$ of $\mathcal T$.
\end{remark}

\begin{proof} $\overline{\mathcal F}$ is closed under the suspension and 
loop functor. Thickness: We have to show $\overline{\mathcal F}$ is closed under direct summands
and $A,B\in \overline{\mathcal F}$ implies $C\in \overline{\mathcal F}$ for distinguished triangles $(A,B,C,\alpha,\beta,\gamma)$
in $\overline{\mathcal E}$.
An equivalent characterization is: 
For distinguished triangles $(A,B,C,\alpha,\beta,\gamma)$
in $\overline{\mathcal E}$ such that $C\in \overline{\mathcal F}$ and such that
$\alpha:A\to B$ factorizes over an object in $\overline{\mathcal F}$ it follows that
$A,B$ are also in $\overline{\mathcal F}$.

\medskip\noindent
Suppose $U\oplus V = C \in \overline{\mathcal F}$. Then $C=C_{clean} \oplus I$ for injective $I$
and finite dimensional $C_{clean}\in \mathcal E$.
Using the clean
decompositions $U=U_{clean}\oplus I_U$ and $V=V_{clean}\oplus I_V$
we can assume $U,V$ to be clean. We already have shown that $U,V$ clean implies
$U\oplus V$ clean. Hence $U\oplus V= C_{clean}$ is finite dimensional. Hence $U$ and $V$ are finite dimensional. 

\medskip\noindent
Suppose we are given a distinguished triangle $(A',B',C',\alpha,\beta,\gamma)$ in $\overline {\mathcal E}$ such that $A',B'$ are in $\overline{\mathcal F}$. We have to show $C'\in \overline{\mathcal F}$. Obviously we may replace the triangle by an isomorphic standard triangle.  So let us assume $(A,B,C,\alpha,\beta,\gamma)$ is a standard triangle in $\overline{\mathcal E}$, hence
there is an exact sequence $$ 0 \to A \to B \to C \to 0 $$ in $\mathcal E$
defined by morphisms $\alpha:A\to B$ and $\beta:B\to C$ in ${\mathcal E}$ such that $A,B\in \mathcal F$.
We have to show $C\in \mathcal F$.
Using a clean decomposition $A_{clean}\oplus I$ of $A$, we may write $B= I\oplus B'$
such that $B'_{clean}$ is in $\mathcal T_+$. Hence replacing $A$ by $A_{clean}$ and $B$ by $B'$
we may assume that $A\in \mathcal T_+$, and $B=B_{clean}\oplus J$ for injective $J$ and 
$B_{clean}\in \mathcal T_+$. Hence there exists an exact sequence in $\mathcal C$
$$  0 \to B_{clean}/(A\cap B_{clean})  \to C \to J/im(A) \to 0 \ .$$  
Since $im(A)$ is finite dimensional and $J= \bigoplus_{\nu\in X} I_\nu$, we can assume
that $im(A)$ is contained in a finite sum  $\bigoplus_{\nu\in X_0} I_\nu$ given by suitable finite
subset $X_0$ of the index set $X$. Hence $J/im(A)\cong J' \oplus J''$ isomorphic to the direct sum  of the injective comodule $J'=\bigoplus_{\nu\notin X_0} I_\nu $ and the finite dimensional comodule
$J'' = (\bigoplus_{\nu\in X_0} I_\nu)/im(A)$. The summand $J'$ is projective, hence splits in the exact
sequence above; hence $C \cong J'' \oplus C'$, where $C'$ is a finite dimensional extension
of the finite dimensional comodules $B_{clean}/(A\cap B_{clean})$ and $J''$. Hence
$C_{clean}$ is a finite dimensional. Since $C$ is in $\mathcal E$ this implies $C\in {\mathcal F}$.
\end{proof}


\subsection{Isogenies II} \label{sec:isogenies}  Let $\Sigma$ denote the class of morphisms $s$ in $\overline{\mathcal E}$, whose
cone is in the subcategory $\overline{\mathcal F}$. We call morphisms in $\Sigma$ isogenies.
As shown in \cite[p.279ff]{Verdier},  the class of morphisms $\Sigma$ admits a calculus of right and left fractions, since $\overline{\mathcal F}$ is a thick subcategory.  
This defines a triangulated localization functor
$$ \overline{\mathcal E} \to  \overline{\mathcal E}[\Sigma^{-1}] \ .$$
Let $\mathcal H$ denote the full subcategory of objects in $\overline{\mathcal E}$
which are isomorphic to cofibrant replacements of objects which are stably isomorphic 
to objects in $\mathcal T$.
This is a full triangulated subcategory of $\overline{\mathcal E}$, i.e stable under suspension and loop
functor, 
and closed under extension meaning that for a distinguished triangle $(A,B,C)$
the condition $A,C\in \mathcal H$ implies $B\in \mathcal H$.

\bigskip\noindent
Then 
$$   \overline{\mathcal F} \subset  \mathcal H   \ ,$$
and  $\mathcal H$ is a symmetric monoidal rigid subcategory of $\overline{\mathcal E}$
 such that $\overline{\mathcal F}$ is a tensor ideal in $\mathcal H$.
 Indeed for $X=I\otimes N$ and $q: QY \to Y$ with $N\in \mathcal T_+$ and
 $Y\in \mathcal T$ we have $X\otimes QY \to X\otimes Y$ is a cofibrant replacement
of $X\otimes Y = I\otimes Y \oplus N\otimes Y$, which is stably equivalent to $N\otimes Y\in \mathcal T_+$ and hence is in $\overline{\mathcal F}$. This implies that 
the localization functor  
$$  {\mathcal H} \to  {\mathcal H}[\Sigma^{-1}] $$ 
is a triangulated tensor functor.






\bigskip
\part{The homotopy category associated to $GL(m|n)$} \label{part-gl}

We now study one particular Frobenius pair: For the general linear \mbox{supergroup} $GL(m|n)$ over an algebraically closed field $k$ of characteristic zero and its parabolic subgroup $P^+$ of upper triangular block matrices we construct an explicit cofibrant replacement for any $X \in \mathcal T$ and deduce from this an \mbox{another} description of $Ho \mathcal{T}$. The semisimplification $Ho \mathcal{T}_{m|n}/\mathcal{N}$ is the semisimple representation category of a supergroup scheme. In the $GL(m|1)$-case we determine this semisimple quotient. For more background on $\mathcal{T}_{m|n} = Rep(GL(m|n))$ we refer to \cite{Heidersdorf-Weissauer-tensor}.


\section{Cofibrant replacements and an explicit description of $Ho \mathcal T$} \label{sec:gl-m-n-cofibrant}

\subsection{Representations of $GL(m|n)$} \label{sec:repr-of-gl} \subsubsection{Various representation categories}  For an algebraic supergroup $G$ the category of rational representations is identified with the category of comodules of $k[G]$ as in the classical case \cite{Waterhouse} \cite{Jantzen}. The finite dimensional rational representations $Rep(G)$ can alternatively be identified with the finite dimensional algebraic representations of $\mathfrak{g} = Lie(G)$, those representations whose restriction to $\mathfrak{g}_0$ lifts to a rational representation of $G_0$ \cite{Serganova-quasireductive} \cite{Weissauer-semisimple}. Considering $\mathcal{C} = Comod(k[G])$ therefore amounts to study inductive limits of finite dimensional algebraic representations of $\mathfrak{g} = Lie(G)$.

\subsubsection{The case $\mathfrak{gl}(m|n)$} We adopt the notations of \cite{Heidersdorf-Weissauer-tensor}.  Let $\mathfrak{g} = \mathfrak{gl}(m|n)$ the Lie superalgebra of $GL(m|n)$. We assume without loss of generality $m \geq n$. A finite dimensional rational representation $\rho$ of $GL(m|n)$  is by the conventions above a representation of $\mathfrak{g}$ whose restriction to $\frakg_{0}$ defines an algebraic representation of $G_{0} = GL(m) \times GL(n)$. We denote by $\mathcal{T} = \mathcal{T}_{m|n}$ the category of all finite dimensional representations with parity preserving morphisms.

\subsubsection{Weights} The irreducible representations in $\mathcal T_{m|n}$ are parametrized by their highest weight with respect to the Borel subalgebra of upper triangular matrices. A weight $\lambda=(\lambda_1,...,\lambda_m \ | \ \lambda_{m+1}, \cdots, \lambda_{m+n})$ of an irreducible representation in $\mathcal R_n$ satisfies $\lambda_1 \geq \ldots \lambda_m$, $\lambda_{m+1} \geq \ldots \lambda_{m+n}$ with integer entries. The Berezin determinant of the supergroup $G$
defines a one dimensional representation $Ber$. Its weight is
is given by $\lambda_i=1$ and $\lambda_{m+i}=-1$ for $i=1,..,n$. For each weight $\lambda$ we also have the parity shifted irreducible representation $\Pi L(\lambda)$.  Both $\vee$ and $*$ (the twisted dual)
define contravariant functors on $\mathcal{T}_{m|n}$. 

\subsubsection{The upper parabolic} The supergroup $GL(m|n)$ contains the parabolic subgroup $P(m|n)^+$ of upper triangular block matrices. As in section \ref{sec:intro-ov} we denote by $\mathcal{C}_-$ the trivially fibrant objects of the model structure induced by the stable model structure on $Rep(P(m|n)^+)$ (see \ref{sec:induced-model-structure}), i.e. those algebraic representations of $GL(m|n)$ that become projective when restricting to $P(m|n)^+$.

\subsubsection{Kac objects} An irreducible representation of weight $\lambda$ of the even subgroup $G_0 = GL(m) \times GL(n)$ can be trivially extended to $P(m|n)^+$ and then induced to $GL(m|n)$. This parabolic induction  yields the Kac modules \cite{Kac-Rep} \[ V(\lambda) = Ind_{P(m|n)^+}^{\ GL(m|n)} L_P(\lambda),\] the universal highest weight modules in $\mathcal{T}_{m|n} = Rep(GL(m|n))$. They are the standard modules in the highest weight category $\mathcal{T}_{m|n}$. The twisted dual $V(\lambda)^*$  (also called anti Kac module)  is then the corresponding costandard module. 


\subsection{Two Frobenius pairs for $GL(m|n)$}

We abbreviate $P^+$ for the maximal parabolic subgroup of upper triangular block matrices and $P^-$ for the maximal parabolic of lower triangular block matrices. By \cite[Lemma 3.3.1]{Germonie} \[ V(\lambda)^* = Coind_{P^-}^G L_{P^-}(\lambda) = Ind_{P^-}^G L_{P^-}(\lambda - 2 \rho_1).\]

\begin{lem} (\cite[Proposition 3.6.2]{Germonie} \begin{enumerate} \item For $M \in \mathcal{T}_{m|n}$ the following are equivalent: \begin{itemize} \item $M$ has a filtration by Kac modules. \item $Ext^1(M,V^*(\mu)) = 0$ for all $\mu \in X^+$. \item $Res_{P^-}M$ is projective in $Rep(P^-)$.\end{itemize} \item  For $M \in \mathcal{T}_{m|n}$ the following are equivalent: \begin{itemize} \item $M$ has a filtration by anti Kac modules. \item $Ext^1(M,V(\mu)) = 0$ for all $\mu \in X^+$. \item $Res_{P^+}M$ is projective in $Rep(P^+)$.\end{itemize} \end{enumerate}
\end{lem}

We denote by $\mathcal T_+$ the tensor ideal of modules with a filtration by Kac modules in $\mathcal T$ and by $\mathcal T_-$ the tensor ideal of modules with a filtration by anti Kac modules. The full subcategory $\mathcal{T}_+$  of representations with a filtration by Kac modules (simply called Kac objects) and the full subcategory $\mathcal{T}_-$ of representations with a filtration by anti Kac modules (called anti Kac objects) are then orthogonal in the sense that \[ Ext^1(\mathcal{T}_+, \mathcal{T}_-) =0\] Furthermore $\mathcal{T}_+ \cap \mathcal{T}_- = Proj$ (every tilting module is projective).

\begin{cor} For the Frobenius pair $(P^+,G)$ (see definition \ref{def:frobenius-pair})  $\mathcal{C}_- \cap \mathcal{T}_{m|n} =\mathcal{T}_-$ and $\mathcal{C}_+ \cap \mathcal{T}_{m|n} = \mathcal{T}_+$.
\end{cor}


\begin{remark} Note that our notation for modules with Kac respectively anti Kac filtrations shows that we always consider the case $H = P^+$. If we would exchange $P^+$ with $P^-$, this would switch the roles of $\mathcal{T}_+$ and $\mathcal{T}_-$ and our notation in section \ref{sec:repr-of-gl} for modules with a Kac or anti Kac filtration would be inconsistent.
\end{remark}


\subsection{Axiomatic description of the highest weight structure}

We fix until the end of the article $\mathcal D$ such that $\mathcal T_-$ is the
category of representations with anti Kac flags. In other words: $\mathcal{D}$ is the ind-category of $Rep(P(m|n)^+)$ where $P(m|n)^+$ is the parabolic subgroup of upper triangular block matrices in $GL(m|n))$. 

\medskip\noindent
As an  abelian category $\mathcal T$ splits into blocks ${\mathcal T}_\Lambda$, each of which is a highest weight category with duality \cite{CPS}. The standard modules in this highest weight structure are the Kac modules $V(\lambda)$. We now axiomatize the situation of the $\mathcal{T}_{m|n}$-case and consider an abelian category $\mathcal{T} = Rep(G)$ for some supergroup $G$ satisfying the following sets of assumptions.

\medskip\noindent
{\bf First list of assumptions}. As an  abelian category $\mathcal T$ splits into blocks ${\mathcal T}_\Lambda$, each of which is a highest weight category with duality \cite{CPS} in the following way:  
Each block $\Lambda$ has the structure of an interval finite poset 
such that the elements $\lambda$ correspond to representatives $L(\lambda)$ of isomorphism classes of simple objects 
in the block ${\mathcal T}_{\Lambda}$. Each $L(\lambda)$ has a projective cover $P(\lambda)\in {\mathcal T} _\Lambda$.
Furthermore for $\lambda\in \Lambda$
there exist objects $V(\lambda)\in {\mathcal T_+\cap \mathcal T}_\Lambda$
such that

\medskip\noindent
\begin{enumerate}
\item 
There exists an epimorphism $P(\lambda)\to V(\lambda)$, and the kernel has
 a finite filtration whose successive quotients are of the form $V(\nu)$ for certain $\nu\in \Lambda$
 such that $\nu > \lambda$, 
\item There exists an epimorphism $V(\lambda) \to L(\lambda)$ such that $L(\lambda)$ is the cosocle of $V(\lambda)$,  so that the kernel (radical) has a finite filtration whose successive quotients are of the form $L(\mu)$ for certain $\mu\in \Lambda$
 such that $\mu < \lambda$.
\end{enumerate}

\medskip\noindent
{\bf Second list of assumptions}. Let $V=V(\one)$ in $\mathcal T_+$ be the standard module corresponding to the trivial object and $P=P(\one)$
the projective hull of $\one$. These objects are defined for a highest weight category.
We abbreviate $V({\mathcal L})=V\otimes {\mathcal L}$.
We now assume that the following additional assumptions hold: 

\begin{enumerate}
\item There exists an antiinvolutive $\otimes$-functor $*$ inducing
an equivalence of the tensor categories ${}^*: {\mathcal T} \to {\mathcal T}^{op}$, which permutes the subcategories $\mathcal T_-$ and $\mathcal T_+$ so that
$Y^*\cong Y$, if $Y$ is simple or if $Y$ is an indecomposable projective object.
\item There exists an invertible simple object $\mathcal L$ in $\mathcal T$, such that
\item  there exists an injection $i: V({\mathcal L}) \to P$,
\item and there exists a surjection $\pi: P \to V({\mathcal L})^*$.
\item $V$ (hence $V({\mathcal L})$) is rigid with a Loewy filtration
of length $r$ with $r$ pairwise non-isomorphic simple constituents $L_i$.
\item The kernel of $\pi\circ i: V({\mathcal L}) \to V({\mathcal L})^*$ is 
the radical of $V({\mathcal L})$, i.e. $V({\mathcal L})$ divided by the kernel
is isomorphic to ${\mathcal L}$.
\item The Jordan-H\"older constituent $\mathcal L$ is the highest weight
representation in $P$ and has multiplicity one.
\end{enumerate}

\medskip\noindent
Property 5 is of auxiliary nature. It will not be used in the following except
that it allows to verify the other properties in the case where ${\mathcal T}={\mathcal T}_{m\vert n}$.

\begin{lem} Under the assumption 1. above the subcategories $\mathcal T_-$
and $\mathcal T_+$ are stable under the Tannaka duality functor ${}^\vee$.
\end{lem}

\begin{proof} Assumption 1) implies that as a tensor functor ${}^*$ commutes with the 
Tannaka duality ${}^\vee$. Since ${\mathcal T}_-$ is preserved by ${}^\vee$, therefore
1) implies that also ${\mathcal T}_+$ is preserved by ${}^\vee$. Indeed, for
$X \in \mathcal T_+$ we get $X^*\in \mathcal T_-$ and hence $(X^\vee)^* \cong (X^*)^{\vee}\in \mathcal T_-$.
Therefore $X^\vee = (X^\vee)^{**} \in \mathcal T_+$. 
\end{proof}

\begin{lem} Axioms 1.-7. are satisfied in the case $GL(m\vert n)$.
\end{lem} 

\begin{proof} For ${\mathcal T}_{m\vert n}$ these conditions hold for ${\mathcal L} = Ber^n$. 
Then $\mathcal L$ is the dual of the socle of the Kac module $V$ of the trivial representation 
$$ \xymatrix{ 0 \ar[r] & I \ \ar[r]^a & \ V \ \ar[r]^b & \ \one \ar[r] &  0 \cr} ,$$
hence ${\mathcal L}^{-1}$ is the cosocle of $V^*$, and $\one$ is the cosocle 
$$ \xymatrix{ 0 \ar[r] & J \ar[r] & V^*({\mathcal L})  \ar[r] &  \one \ar[r] &  0 \cr} ,$$
of $V^*({\mathcal L})= V^*\otimes {\mathcal L}$. The rigidity assertion 5) has been shown in \cite{Brundan-Stroppel-4},\cite{Brundan-Stroppel-1},\cite{Brundan-Stroppel-2}.  Notice also $V^\vee \cong {\mathcal L}\otimes V $. 
In particular by property  5)
the constituents $L_i$ of the Loewy filtration of $V$ satisfy
$$L_i^\vee \ \cong\  L_{r-i} \otimes {\mathcal L}\ .$$ 
Property 7) follows from the fact that $\mathcal L$ is the highest weight constituent of $P$
and also follows from loc. cit. The Loewy length is $r=n$
by \cite[Theorem 3.2]{Su-Zhang}. Property 3) and 4) and also 7) follow for $m+n$ from $${\mathcal L}\otimes V \otimes V^* = P \oplus Q\ ,$$ where
$Q$ is a projective object of atypicality $<n$ (see lemma \ref{thm:V-tp}). If 3), 4), 5) hold for $\mathcal{T}_{n|n}$, they hold for $\mathcal{T}_{m|n}$ by using the block equivalence to the principal block in $\mathcal{T}_{n|n}$. 
The inclusion $\one \to V^*$ induces the embedding $ i: {\mathcal L}\otimes V  \to P \subset P\oplus Q$; similarly the projection $V \to \one$ induces the surjection
$ \pi:  P \oplus Q \to {\mathcal L}\otimes V^*$, since $\pi$ is necessarily  trivial on $Q$.
\end{proof}  

So let us now take all these properties for granted. For simplicity we could assume ${\mathcal T=\mathcal T}_{m\vert n}$ (for $m\geq n$)
in order to ensure that these conditions hold.
Then we obtain

\begin{lem} The restriction of $V^*\otimes \mathcal L$ under $U:{\mathcal C} \to {\mathcal D}$ is projective
and the restriction of $V$ to $\mathcal D$ decomposes in the following way
$$  U(V) \ \cong \ U(I) \ \oplus \ \one  \ .$$
\end{lem}

\begin{proof} Since $V\in \mathcal T_+$ the first property 1) implies $V^*\in \mathcal T_-$, hence
$V^*\otimes{\mathcal L}\in \mathcal T_-$. Concerning the second assertion
this implies that $U(V^*\otimes{\mathcal L})$ is projective and that 
$  U(\pi):   U(P) \to U( V({\mathcal L})^*) = U(V^* \otimes {\mathcal L}) $
splits
$$  U(P) \ \cong \ U(Kern(\pi))\ \oplus\ U( V({\mathcal L})^*) \ .$$
Now
$ D = U(i(V({\mathcal L})))   \ \cap \ U( V({\mathcal L})^*) \ \neq 0 $,
since by the properties 6) and 7) the Jordan-H\"older constituent $U({\mathcal L})$ of $U(P)$
is obtained from $U(i(V({\mathcal L}))) \subset U(P)$ but not obtained from $U(Kern(\pi))$. To proof our second assertion it would suffice to show $D \cong U({\mathcal L})$. Indeed, since $Kern(\pi)$ contains the radical of $i(V({\mathcal L}))$ by property 6),
the splitting of $U(\pi)$ then induces a splitting of $U(i(V({\mathcal L})))$
$$ U(i(V({\mathcal L})))  \ \cong \ U(Kern(\pi) \cap i(V({\mathcal L}))) \ \oplus \  U({\mathcal L}) \ .$$
 Tensoring with ${\mathcal L}^{-1}$  gives
the required isomorphism $U(V) \cong U(I) \oplus \one$. 
\end{proof}


\subsection{Construction of cofibrant replacements}

Recall that $V = V(\one)$ and $I = V(\one)/\one$. Now consider in the category $\mathcal C$ the objects 
$P=\bigoplus_{i=0}^\infty P_i$, $Q=\bigoplus_{i=0}^\infty Q_i$  and
$R=\bigoplus_{i=0}^\infty R_i$ 
for
\begin{align*} R_i \ & = \ (I\otimes I^*)^{\otimes (i+1)} \\
Q_i \ & = \ (I\otimes V^*)\otimes (I\otimes I^*)^{\otimes i}\ \bigoplus\ V \otimes (I\otimes I^*)^{\otimes i} \ \\  
 P_i \ & = \ I \otimes (I\otimes I^*)^{\otimes i} .\end{align*}
We define morphisms $\alpha_i: P_i \to Q_i$ by $\alpha \otimes id_{ (I\otimes I^*)^{\otimes i}}$ for
$$ \alpha:\  I \ \hookrightarrow \ (I\otimes V^*) \ \bigoplus \ V \ ,$$ 
where $\alpha$ is the diagonal map obtained from the two morphisms $id_I \otimes b^*: I = I\otimes \one  \hookrightarrow I \otimes V^*$
and $a: I \hookrightarrow V$.
Similarly
define morphisms $\beta_i: Q_i \to R_i$ by $\beta \otimes id_{ (I\otimes I^*)^{\otimes i}}$ for
the epimorphism
$$ \beta:\ \  (I\otimes V^*) \bigoplus V \ \longrightarrow \ I\otimes I^* \ ,$$ 
where $\beta$ is of projection onto $(I\otimes V^*)$ followed by the epimorphism $id_I \otimes a^*: I \otimes V^*= I\otimes I^* $. Finally define for $i\geq 1$
morphisms $\gamma_i: Q_i \to R_{i-1}$ by $\gamma \otimes id_{ (I\otimes I^*)^{\otimes i}}$ for
the epimorphism
$$ \gamma:\ \  (I\otimes V^*) \bigoplus V \ \longrightarrow \ V \to  \one  \ ,$$ 
where $\gamma$ is of projection onto $V$ followed by the epimorphism $b: V\to \one $. Put $\gamma_0=0$.
The maps $(\alpha_i)_{i\geq 0}$ and $(\gamma_i)_{i\geq 1} - (\beta_i)_{i\geq 0}$ define a complex in $\mathcal C$
$$ 0 \to P \to Q \to R \to 0 \ .$$
Let $\Omega = Kern(\beta)/Im(\alpha)$ be its cohomology.
The composition of epimorphisms $\Omega \to \Omega_0 \to V \to \one$ defines
an epimorphism $q: \Omega \longrightarrow \one$.

\begin{lem} \label{thm:cofib-of-1}$\Omega$ is cofibrant in $\mathcal C$.
There exists an epimorphism $$q: \Omega \to \one$$ with kernel
in $\mathcal C_-$. Hence $\Omega$ is a cofibrant replacement of $1$.
\end{lem}

\begin{proof} The inclusion $V \to Q_0$ on the second summand of $Q_0$ 
induces a complex map, which defines a monomorphism on cohomology
$ V \hookrightarrow \Omega $
with quotient 
$\Omega/V \ \cong \ \Omega \ \otimes \ (I\otimes I^*) $.
Similarly $V\otimes  (I\otimes I^*) \hookrightarrow 
\Omega \otimes (I\otimes I^*)$ has quotient isomorphic to
$\Omega \otimes (I\otimes I^*)^{\otimes 2}$.
Iterating this gives short exact sequences
$$ 0 \to \Omega_i \to \Omega \to \Omega\otimes (I\otimes I^*)^{\otimes i} \to 0 \ .$$
The kernels $\Omega_i$ define an increasing sequence of sub-comodules of $\Omega$
$$ V =\Omega_0 \subset \Omega_1 \subset \Omega_2 \subset ... $$
such that
$ \Omega  =  co\lim \Omega_i $. 
Since $$ \Omega_{i+1}/\Omega_{i}\ \cong\ V \ \otimes\ (I\otimes I^*)^{\otimes i}$$
is in $\mathcal T_+$, all the comodules $\Omega_i$ are in $\mathcal T_+$.
Hence $\Omega\in \mathcal C_+$. This shows that $\Omega$ is cofibrant.
The kernel $K$ of $q:\Omega\to \one$ 
is the cohomology of the complex
$$ P \to Q' \to R \ ,$$
for $Q'=Kern(Q \to \one)$.
Since $U$ is an exact functor
and commutes with direct sums, we can compute $U(R)$ from the complex
$U(P)\to U(Q')\to U(R)$. Since $U(V)\cong U(I)\oplus \one$ splits, $U(Q)$ simplifies 
$$ U(Q_i) \cong U(V^*)\otimes U(I\otimes I^*)^{\otimes i} \ \bigoplus \ (U(I)\oplus \one)\otimes U(I\otimes I^*)^{\otimes i}
$$
Thus we obtain $$ U(Q')/U(P)\ \cong \ \bigoplus_{i=0}^\infty  \ U(V^*)\otimes U(I\otimes I^*)^{\otimes i}
\ \  \oplus \ \ \bigoplus_{i=1}^\infty  \  U(I\otimes I^*)^{\otimes i} $$
so that the kernel of $U(Q')/U(P) \to U(R)$ becomes
$$ U(R) \ \cong \ \bigoplus_{i=0}^\infty \  U(V^*)\otimes U(I\otimes I^*)^{\otimes i} \ .$$
Since $U(V^*)$ is projective by the last lemma, $U(V^*)\otimes U(I\otimes I^*)^{\otimes i}$
is projective as well. Hence $U(R)$ is a direct sum of projectives objects in $\mathcal D$, hence projective
in $\mathcal D$. Thus $R\in \mathcal C_-$. 
\end{proof}

\begin{example} \label{ex:gl-1-1} In the $GL(1|1)$-case $V(\one)$ has the composition factors $\one$ and $Ber^{-1}$. Therefore $I \otimes I^* \cong Ber^{-2}$. Accordingly $\Omega_{i+1}/\Omega_{i}$ equals \[ V  \otimes\ (I\otimes I^*)^{\otimes i} \cong V \otimes (Ber^{-2})^{\otimes i} \cong V(Ber^{-2i}).\] 
\end{example}

\medskip\noindent
We can find embeddings
$  K\otimes (I\otimes I^*) \hookrightarrow K $
with kokernel isomorphic to $V^*\otimes I$. Indeed
$P_{i+1}=P_i\otimes (I\otimes I^*)$, $Q_{i+1}=Q_i\otimes (I\otimes I^*)$
and $R_{i+1}=R_i\otimes (I\otimes I^*)$, and similarly for the
complex maps. This defines a short exact sequences of complexes
$$ \xymatrix{ 0 \ar[r] & P \otimes (I\otimes I^*) \ar@{^{(}->}[d]\ar[r] & Q' \otimes (I\otimes I^*) \ar[r]\ar@{^{(}->}[d] & R \otimes (I\otimes I^*)  \ar@{^{(}->}[d]\ar[r] & 0  \cr
0 \ar[r] & P \ar[r]\ar@{->>}[d] & Q' \ar[r]\ar@{->>}[d] & R \ar@{->>}[d]\ar[r] & 0\cr
0 \ar[r] & P_0 \ar[r] & P_0\oplus (I\otimes V^*) \oplus R_0\ar[r] & R_0 \ar[r] & 0\cr} $$
whose cohomology sequence gives the short exact sequence in $\mathcal C$
$$  0 \to K\otimes (I\otimes I^*) \to K \to I\otimes V^* \to 0 \ .$$
By $i$-fold iteration this gives short exact sequences
$$ 0 \to K^{(i)} \to K \to K_i \to 0 \ $$
$$ K^{(i)} \ \cong \ K\otimes (I\otimes I^*)^{\otimes i} $$
Hence the objects $K^{(i)}$ and $K_i$ are in $\mathcal C_-$. 

\medskip\noindent
Hence $K$ has a descending chain of subcomodules $K^{(i)}$ in $\mathcal C$
$$ K=K^{(0)}  \supset K^{(1)} \supset K^{(2)} \supset K^{(3)} \cdots   $$
whose successive quotients $K_i$ are in $\mathcal T_-$. 
By the construction the weights of all irreducible constituents
of $K$ and of $I\otimes I^*$ are $<0$. Hence we get
from simple weight reasons

\begin{thm} For any $Y$ in $\mathcal T$ there exists an integer
$n$ such that $Hom_{\mathcal C}(K^{(n)},Y) =0$.
\end{thm}

\begin{thm} \label{thm:good-replacement} For any object $X$ in $\mathcal T$ there exists 
a cofibrant replacement $q_X:QX \to X$ in $\mathcal C$ with the following property.
For any $Y$ in $\mathcal T$ there exists a subobject $K' \in QX$ of finite codimension
contained
in $Kern(q_X)$ such that $K'\in \mathcal C_-$ and such that  $Hom_{\mathcal C}(K',Y) =0$.
\end{thm}

\begin{proof} Consider $QX=\Omega \otimes X$ for $q_X=q\otimes X$ and $K'=K^{(n)}\otimes X$
for $n$ large enough, such that $Hom_{\mathcal C}(K^{(n)},X^\vee \otimes Y) =0$. \end{proof}

Since $[X,Y] = Hom_{Ho\mathcal{C}}(QX,Y)/\sim$, the two theorems imply immediately the following important corollary.

\begin{cor} For $X,Y \in \mathcal{T}$ we have $\dim [X,Y] < \infty$.
\end{cor}

\begin{remark} For a way to see the cofibrant replacement as a Kac resolution see section \ref{sec:intro-ov}. An estimate for the dimension of $[L(\lambda),L(\mu)]$ can be found in section \ref{sec:intro-gl}. We do not know a direct representation theoretic meaning of this dimension.
\end{remark}


\subsection{A second interpretation of $Ho \mathcal T$}

The full image category of $\mathcal T_-$ in $\overline{\mathcal T}$ is
a triangulated subcategory. It is thick, since $X\cong A\oplus B$ for $X\in \mathcal T_-$
implies $P\oplus X \cong P'\oplus A \oplus B$ in $\mathcal T$, hence
$X'\cong A' \oplus B'$ for the clean components $X',A',B'$ of $X,A,B$.
Let $ho\mathcal T$ be the quotient category of the triangulated stable category
$\overline{\mathcal T}$ by the thick subcategory $\overline{\mathcal T}_-$.
There is a natural
tensor functor $$ ho{\mathcal T} \to Ho\mathcal T \ .$$
Fix objects $X$ and $Y$ in $\mathcal T$.
Then morphisms $X\to Y$ in $ho{\mathcal T}$ are (certain equivalence classes of diagrams) of the form (see \cite{Neeman})
$$ \xymatrix{ &  Z \ar[dl]_s\ar[dr]^f &  \cr X &  & Y} $$
for morphisms $s:Z\to X$ and $f:Z\to X$ in $\overline{\mathcal T}$ with $Z\in \mathcal T$
(hence  $s$ and $f$ are classes of morphisms in $\mathcal T$ that are still denoted $s$ and $f$
by abuse of notation) such that
the cone of $s$ is in $\mathcal T_-$. 
Since $\overline{\mathcal T}_-$ maps to zero under the functor $\gamma: \overline{\mathcal T} \to Ho{\mathcal T}$, defined as a full
subcategory of $Ho\mathcal C$, the morphism $s:Z\to X$ in $\overline{\mathcal T}$
becomes an isomorphism $\gamma(s)$ in $Ho{\mathcal T}$. In the manner
diagrams are composed and the equivalence classes are defined, it is easy to see
that we obtain an induced functor
$$  \gamma: ho{\mathcal T} \to Ho\mathcal T \ ,$$
which maps the equivalence class of the diagram $X \leftarrow Z \rightarrow Y$
to $\gamma(f)\circ\gamma(s)^{-1}$.
Let us show

\begin{thm} \label{thm:gl-m-n-homotopy} The functor $\gamma$ induces
a $k$-linear equivalence of tensor categories \[ \overline{\mathcal{T}}/\overline{\mathcal{T}}_- =: ho{\mathcal T} \ \cong \ Ho\mathcal T\] between the quotient of the stable category by the thick ideal of anti Kac modules and the homotopy category $Ho\mathcal{T}$.
\end{thm}

\begin{proof}
We have to
show that the induced map
$$ \gamma: Hom_{ho\mathcal T}(X,Y) \to Hom_{Ho\mathcal T}(X,Y) $$
is an isomorphism. 

\medskip\noindent
$X \leftarrow Z \rightarrow Y$  
for $f:Z\to Y$ is equivalent to zero if and only if there exists
a morphism $s'$ in $\overline{\mathcal T}$ with cone in $\overline{\mathcal T}_-$
such that $f\circ s' =0$ in $\overline{\mathcal T}$ (see \cite{Neeman}). 
On the other hand
recall  $Hom_{Ho\mathcal T}(X,Y) = Hom_{\overline{\mathcal C}}(QX,Y)$.
Since $\gamma(s): Z \to X $ is an isomorphism, there exists a commutative diagram
in $Ho\mathcal T$
$$ \xymatrix{ QX\ar[dd]_{q_X}\ar@{.>}[dr]^q &  &  \cr
& Z \ar[dl]_s\ar[dr]^f & \cr
X & & Y \cr} $$
such that $\gamma(s)^*: [X,Y]\cong [Z,Y]$ and $[Z,Y] = Hom_{\overline{\mathcal C}}(QX,Z) $,
since $QX$ is cofibrant and $Z$ is fibrant.
Hence $\gamma(f)\circ\gamma(s)^{-1}$ is equivalent to zero in $Ho\mathcal T$ if and only if 
$f\circ q =0$. This is where the last theorem comes in. Since $Z\in \mathcal T$ it implies
that $q$ is trivial on a subobject $K'$ of $QX$ such that $Z'=QX/K' \in \mathcal T$.
Hence $s':Z'\to Z$ is well defined in $\mathcal T$, such that $f\circ s'=0$.
But there also exists a distinguished triangle
$$ C_q \to C_{q_X} \to C_{s} \to C_q[1] \ .$$  
Since $C_{q_X}=K[1]\in \mathcal C_-$ and $C_s\in \mathcal C_-$ this implies
$C_q\in \mathcal C_-$. Hence $C_q\in \mathcal T_-$.
Therefore $K'\in \mathcal C_-$ and 
$$ K' \to C_{q} \to C_{s'} \to K'[1] \ $$
implies $C_{s'}\in \mathcal C_-$. But then already $C_{s'}\in \mathcal T_-$.
Therefore $f\circ s'=0$ implies that the class of $X \leftarrow Z \rightarrow Y$  
is the zero morphism $X\to Y$ in $ho\mathcal T$. This shows that $\gamma$ is faithful.
The fullness of $\gamma$ is shown similarly. Any morphism in $[X, Y]$ is represented
by a morphism $q:QX \to Y$ similarly as in the diagram above. Since $q_X\oplus q$ is trivial
on some $K' \subset Kern(q_X)$ with finite quotient $Z' = QX/K'$
we obtain a diagram in $\mathcal T$  
$$ \xymatrix{ &  Z' \ar[dl]_{s'}\ar[dr]^{f'} &  \cr X &  & Y} $$
($s'$ is induced by $q_X$ and $f$ is induced by $q$)
such that $\gamma(s')^{-1} \circ \gamma(f) \in [X,Y]$ represents the morphism
we started from. 
\end{proof} 

\medskip\noindent
Accordingly we will identify the categories $ho\mathcal T$ and $Ho\mathcal T$
in the following. Note however that it is important for us to have both interpretations of $Ho\mathcal{T}$. While the interpretation of $Ho\mathcal{T}$ as a Verdier quotient looks more down to earth, the cofibrant replacements are only visible when we use the model structure on $Ind(\mathcal{T})$.


\subsection{Remarks on the Balmer spectrum}\label{sec:balmer-spectrum}

Balmer \cite{Balmer} defined for a tensor triangulated category the notion of its spectrum by equipping the set of all prime ideals (proper thick tensor ideals such that $a \otimes b \in \mathcal P$ implies $a \in \mathcal P$ or $b \in \mathcal P$) with a Zariski topology. The category $Ho \mathcal T$ is a triangulated tensor category in the sense of Balmer \cite{Balmer}. Hence its spectrum $Spc(Ho \mathcal T)$ is defined. By \cite{BKN-spectrum} the spectrum of the stable category $\overline{ \mathcal T}$ is homeomorphic \[ Spc(\overline{\mathcal T}) \simeq Proj(N-Spec(S^{\bullet}(\mathfrak{f}_1^*)))\] where $N = Norm_{G_0(\mathfrak{f}_1)}$ and the detecting subalgebra $\mathfrak{f}$. Formation of the spectrum is a contravariant functor, and if $F: \mathcal K \to \mathcal L$ is an essentially surjective tensor triangulated functor, the induced map $Spc(F): Spc(\mathcal L) \to Spc (\mathcal K)$ of locally ringed spaces is injective. More specifically, let $q: \mathcal K \to \mathcal L = \mathcal K / \mathcal J$ be the localization functor where $\mathcal J$ is a thick tensor ideal. Then the associated map $Spc(q): Spc(\mathcal L) \to Spc(\mathcal K)$ induces a homeomorphism between $Spc(\mathcal L)$ and the subspace \[ \{\mathcal P \in Spc(\mathcal K) \ | \ \mathcal J \subset \mathcal P \} \subset Spc(\mathcal K) \] of those thick prime ideals containing $\mathcal J$. In our case this applies to $Ho \mathcal T \cong \overline{T}/ \overline{\mathcal T}_-$, but doesn't give a concrete description of $Spc(Ho \mathcal T)$ in this way. Note that by \cite{BKN-spectrum} the thick tensor ideals of $\mathcal T$ are in bijection with specialization closed (union of closed sets) subsets of $N-Proj(S^{\bullet}(\mathfrak{f}_1^*)))$ by assigning to a thick tensor ideal the union of the support varieties of its elements. However the support varieties of anti Kac modules (or modules with a filtration by anti Kac modules) don't seem to have a known description. 



\section{The degree filtration and cofibrant replacements}

\subsection{Degree filtration of Kac objects}\label{sec:degree-filtration}

We show that every Kac object has a canonical degree filtration. By using the cofibrant replacement of an arbitrary $X$ we can also endow $X$ with such a filtration in the ind-category. This filtration could be seen as an analogue of Deligne's weight filtration.

\medskip\noindent
To $\lambda = (\lambda_1,\ldots,\lambda_m \ |  \ \lambda_{m+1},\ldots,\lambda_{m+n})$ we associate the bidegree \[ (d,d') = (\sum_{i=0}^m \lambda_i, \sum_{i=1}^n \lambda_{m+i}).\]  By the description of the blocks \cite{Brundan-Stroppel-4} $d - d'$ only depends on the block of $L(\lambda)$. If we fix the block, we can therefore think of $d$ as the relevant degree and we define therefore \[ deg(\lambda) = \sum_i^n \lambda_i.\] Recall that $\mathcal{T}_+$ denotes the tensor ideal of modules with a filtration by Kac modules in $\calT_{m|n}$ and $\calT_-$ the tensor ideal of modules with a filtration by anti Kac modules in $\calT_{m|n}$. 

\begin{lem} \label{degree-filtration} Each $M \in \calT_+$ has a canonical degree filtration, i.e. a filtration by submodules $F_i(M) \in \calT_+$ such that \[ \ldots \subseteq F_{i-1}(M) \subseteq F_i(M) \subseteq F_{i+1}(M) \subseteq \ldots \] and \[ F_i(M)/ F_{i-1}(M) = \bigoplus_{\lambda} V(\lambda) \] holds for certain Kac modules $V(\lambda) \in \calT_+$ of degree $deg(\lambda) = i$. This filtration is inherited to retracts $N$ of $M$ so that $F_i(N) = N \cap F_i(M)$. The filtration is functorial with respect to morphisms.
\end{lem}

\begin{proof} Every $M$ in $\calT_+$ admits a filtration by objects in $\calT_+$ whose graded pieces are Kac modules. To show the existence as in our claim it suffices to show that $Ext^1(V(\rho_1),V(\rho_2)) = 0$ holds for $deg(\rho_1) \leq deg(\rho_2)$. Since all Jordan-Hoelder constituents $L(\tau)$ of $V(\rho_1)$ have degree $deg(\tau) \leq deg(\rho_1)$, it suffices to show $Ext^1(V(\rho_1),L(\tau)) = 0$ for $deg(\rho_1) \leq deg(\tau)$. The dimension of $Ext^1 (V(\rho),L(\tau))$ can be expressed as the coefficient $p^{(1)}_{\rho,\tau}$ of the linear term of the Kazhdan-Lusztig polynomial $p_{\rho,\tau}$ \cite{Brundan-Stroppel-2}. By \cite[Lemma 6.10]{Musson-Serganova} and \cite[Lemma 5.2]{Brundan-Stroppel-2} $p_{\rho,\tau}^{(1)} \neq 0$ if and only if $\tau$ is obtained from $\rho$ by interchanging the labels at the ends of one of the cups in the cup diagram of $\rho$. Since this operation increases the degree of $\rho$, we must have $deg(\tau) > deg(\rho)$ to get a nonvanishing $p_{\rho,\tau}^{(1)}$. Hence the coefficient must be zero for $deg(\rho) \geq deg(\tau)$. The uniqueness is proved by induction on the length of such filtrations. The minimal nontrivial filtration submodule $N$ is uniquely characterized by the maximal degree highest weight vectors in $M$. Then consider $M/N$ and proceed by induction. Concerning retracts it suffices that $\mathcal T_+$ is closed under retracts, and hence so is $\calT_+$. 
\end{proof}

\begin{lem} \label{C-ind-limit}
Every object $M \in \mathcal{C}_+$ is isomorphic to an inductive limit of finite dimensional Kac objects. In particular the degree filtration extends to $\mathcal{C}_+$.
\end{lem}

\begin{proof} The given object $M \in \mathcal{C}_+$ is stably isomorphic to $\Omega \otimes M$, hence we may replace $M$ by $\Omega \otimes M$. Hence we may suppose \[ M \cong \bigcup_i \Omega \otimes M_i \cong \bigcup_{i,j} \Omega_j \otimes M_i \] for $\Omega  = \bigcup_j \Omega_j$. Since the $\Omega_j$ are finite dimensional Kac objects $\Omega_j$, $\Omega_j \otimes M_i$ is a Kac object as well.
\end{proof}

Given an object $M \in \mathcal{C}_-$, we can dualize it via $()^*$ (the extension of the twisted dual to the ind completion) to obtain an object in $\mathcal{C}_+$ with its canonical degree filtration. Dualizing the filtration steps, equips $M \in \mathcal{C}_-$ with a dual degree filtration.

\begin{cor} \label{cor:anti-kac-lim} Every object $M \in \mathcal{C}_-$ is a sequential projective limit of finite dimensional anti Kac objects. It carries a descending degree filtration whose graded pieces are direct sums of finite dimensional anti Kac modules.
\end{cor}

\begin{example} (\textit{Degree filtration of projective objects}) The degree filtration of a maximal atypical projective cover $P(\tau)$ is as follows using the known filtration of $P(\tau)$ by Kac modules as in \cite[Theorem 5.1]{Brundan-Stroppel-1}. Let $L(\rho)$ denote the constituent of highest weight in $P(\tau)$. Then there are $2^n$ weights $\mu_1,\ldots, \mu_{2^n}$ whose weight diagrams are obtained from the labeled cup diagram of $\tau$ by interchanging the labels at the ends of the $n$ cups in all possible ways. Enumerate these $2^n$ distinct weights as $\mu_1,\ldots,\mu_{2^n}$ so that $\mu_i > \mu_j$ in the Bruhat order implies $i < j$. Then $\rho = \mu_1$ and $\mu_{2^n} = \tau$. The projective cover has then a filtration by submodules $M(i)$ \[ \{0\}  = M(0) \subset M(1) \subset \ldots \subset M(2^n) = P(\tau)  \] such that \[ M(i)/M(i-1) \cong V(\mu_i) \] for each $i=1,\ldots,2^n$. Note that the enumeration in the Bruhat order also implies that $\mu_i > \mu_j$ implies $i < j$ in the degree ordering. The quotient $F_i(P(\tau))/F_{i-1}(P(\tau)) = \bigoplus_{\lambda} V(\lambda)$ is the direct sum of the $V(\mu_j)$ with $deg(\mu_j) = i$. If $deg(\rho) = k$, then $M(2^n)/M(2^n - 1) = V(\tau)$ and $M(1)/M(0) = M(1) = V(\rho)$.
\end{example}

\begin{example} (\textit{Degree filtration of $V(\nu) \otimes L(\mu)$}) (see \cite[Corollary 5.2]{Serganova-character} for a variant) For maximal atypical $L(\lambda) \in \calT_n$ we write $L_0(\lambda)$ for the irreducible $Gl(n) \times Gl(n)$-module with highest weight $(\lambda_1,\ldots,\lambda_n) \times (\lambda_{n+1},\ldots,\lambda_{2n})$. We denote the restriction of $L(\lambda)$ to $G_0 = Gl(n) \times Gl(n)$ by $Res_{G_0}(L(\lambda)) = L_{G_0}(\lambda)$. The restriction decomposes into a direct sum of irreducible representations, and the representation $L_0(\lambda)$ is the irreducible representation in this decomposition of largest degree.

For maximal atypical $L(\nu), \ L(\mu)$ we determine the canonical degree filtration of the maximal atypical summand $V$ of $V(\nu) \otimes L(\mu)$. The restriction $L_{G_0}(\nu) = Res_{G_0}(L(\nu))$ of $L(\nu)$ to the subgroup $G_0$ decomposes into a direct sum of irreducible representations and we denote the direct sum of the irreducible $G_0$ representations with weight of degree $i$ by $L_{G_0}(\nu)^i$. Then $i \leq deg(\nu)$. Since $\bigoplus_{j \geq i} L_{G_0}(\nu)^i$ is stable under the super parabolic $P \subset G$ and since $Ind_P^G$ is an exact functor we obtain from Frobenius reciprocity \[ \tilde{V} := L(\nu) \otimes V(\mu) = L(\nu) \otimes Ind_P^G L_0(\mu) \simeq Ind_P^G(L_{G_0}(\nu) \otimes L_0(\mu)).\] Then the degree filtration of $\tilde{V}$ has the form \[ F_i(\tilde{V}) = Ind_P^G(\bigoplus_{j \leq i} L_{G_0}(\nu)^{j - deg(\mu)} \otimes L_0(\mu)).\] The associated graded modules are the Kac objects \[ gr^i(\tilde{V}) = \bigoplus V(L_{G_0}(\nu)^{i - deg(\mu)} \otimes L_0(\mu)) \] in $\calC^+$. The projection $V$ of $\tilde{V}$ onto the principal block $\Gamma$ has the same structure except that only those irreducible $G_0$-representations in $L_{G_0}(\nu)^{i - deg(\mu)}$ contribute which give Kac modules in $\Gamma$. 
\end{example}


\subsection{Polynomial growth and power series}

We now define two subcategories \[ \mathcal{T}_+ \subset \mathcal{C}_+^{pol} \subset \mathcal{C}_+^{fin} \subset \mathcal{C}_+.\] 

\begin{definition} Let $\mathcal{C}^{pol}_+$, $\mathcal{C}^{fin}_+$ be the full subcategories of $\mathcal{C}_+$ with the following objects:
\begin{itemize}
\item $\mathcal{C}_+^{fin}$: objects $M$ with a degree filtration $F$ such that $F_k(M) = 0$ for some $k \in \mathbb{N}$ and $\dim gr_i^F(M) < \infty$ for all $i$.
\item $\mathcal{C}_+^{pol}$: objects $M$ with a degree filtration $F$ such that $F_k(M) = 0$ for some $k \in \mathbb{N}$ and $\dim gr_i^F(M) < C \cdot P(i)$ for all $i$ where $C = C(M)$ is a constant and $P = P(M)$ a polynomial.
\end{itemize}
\end{definition}        
    
\begin{lem} The subcategories $\mathcal{C}_+^{pol}$, $\mathcal{C}_+^{fin}$ are exact subcategories in $\mathcal{C}_+$ and closed under tensor products. Their images in the homotopy category are triangulated monoidal categories.
\end{lem}

Both statements are obvious (closure under tensor product follows from the classical behaviour of weigths in tensor products over $GL(m) \times GL(n)$). In particular the Grothendieck group $K_0$ of $\mathcal{C}_+^{pol}$, $\mathcal{C}_+^{fin}$ is defined.    
 
\medskip\noindent    
We consider formal power series of the form \[ \sum_{i<k} [M_i] q^i \] for $[M_i] \in K_0(\mathcal{T})$. Then we have a homomorphism \[ K_0(\mathcal{C}_+^{fin}) \to K_0(\mathcal{T})[[q^{-1}]], \quad [M] \mapsto [gr_i^F(M)] q^i.\] Any $M \in \mathcal{T}$ is isomorphic in $Ho\mathcal{T}$ to $M \otimes \Omega$ and lies in the image of $C^{fin}_+$. Therefore we can assign to $M$ a formal power series as above. If $M$ is irreducible, or more generally, if $M$ has a minimal model, then we obtain a canonical power series in $K_0(\mathcal{T})[[q^{-1}]]$ associated to $M$. We remark that the minimal model lies in $C^{pol}_+$. Indeed this is already true for $M \otimes \Omega$ by the description of $\Omega_{i+1}/\Omega_i$ in section \ref{sec:gl-m-n-cofibrant}. By lemma \ref{lem:mini} the minimal model for $M$ is obtained from $M \otimes \Omega$ by  projecting to the clean component. As for Kac objects we can define analogs of $\mathcal{C}^{fin}_+$ and $\mathcal{C}^{pol}_+$ for anti Kac objects using corollary \ref{cor:anti-kac-lim} and then define an associated power series   
 
\begin{example} \label{ex-gl-1-1-power} We compute the minimal model $\Omega(L(a))$ of an irreducible representation $L(a) = L(a|-a)$ in the principal block of $GL(1|1)$ in section \ref{gl-1-1-cofib}. The Kac modules in $\Omega(L(a))$ are $V(a)$, $V(a-2)$, $V(a-4), \ldots$ and the anti Kac modules in $A$ are $V(a-1)^*$, $V(a-3)^*$, $V(a-5)^*,\ldots$. Therefore the power series associated to $\Omega(L(a))$ is \begin{align*} [V(a)]q^a + [V(a-2)]q^{a-2} + [V(a-4)q^{a-4} + \ldots \\ = ([L(a)] + [L(a-1)])q^a + ([L(a-2)] + [L(a-3)])q^{a-2} + \ldots \end{align*} The power series associated to $A$ is similar, but involveses only odd powers of $q$.  
\end{example}


\medskip\noindent 
\textit{A variant}. We would like to read the exact sequence associated to a cofibrant replacement as an equality between power series. As example \ref{ex-gl-1-1-power} shows, the power series of $\Omega(L(a))$ and $A$ might not have any cancellations.  We identify $\mathcal{C}_+^{pol}$, $\mathcal{C}_+^{fin}$ and $\mathcal{C}_-^{pol}$, $\mathcal{C}_-^{fin}$  with their image in the ring of formal power series. We also use $q^{-deg(\lambda)}$ to obtain a formal power series with finite principal part. Then we have three different non-unital rings of formal power series:

\begin{itemize}
\item The $\mathcal{C}_+^{pol}$ version: Here we give $V(\lambda)$ degree $deg(\lambda)$ and assign to $V(\lambda)$ the power series $q^{-deg(\lambda)} [V(\lambda)]$. This construction extends to sequential inductive limits of Kac-modules of polynomial growth.

\item The $\mathcal{C}_-^{pol}$ version: Here we give $V(\lambda)^*$ degree $deg(\lambda)$ and assign to $V(\lambda)^*$ the power series  $q^{-deg(\lambda)} [V(\lambda)^*]$. This extends  to sequential projective limits of anti Kac-modules of polynomial growth.

\item The $\mathcal{C}^{pol}$ version: Here we give $L(\lambda)$ degree $deg(\lambda)$ and assign to $L(\lambda)$ the power series $q^{-d(\lambda)} [L(\lambda)]$. This extends to inductive limits of polynomial growth of finite dimensional modules.
\end{itemize}

There is a natural ring isomorphism between $K_0(\mathcal{C}_+^{pol})$ and the $K_0(\mathcal{C}_-^{pol})$ induced by $()^*$. There is a natural ring homomorphism from $K_0(\mathcal{C}_+^{pol})$ to $K_0(\mathcal{C}^{pol})$ given by \[ q^{-deg(\lambda)}[V(\lambda)] \mapsto \sum_L  q^{-deg(L)} [L]\] where $L$ runs over the irreducible constituents of $V(\lambda)$. Now $0 \to A \to \Omega(M) \to  M \to 0 $ for finite dimensional $M$ and its cofibrant replacement of polynomial growth $\Omega(M)$ with $A \in \mathcal{C}_-$ 
gives via identifications in the power series ring $\mathcal{C}^{pol}$  the formula
\[ [M] = [\Omega(M)] - [A].\]


\subsection{The tensor product $V(\one) \otimes V(\one)^*$}\label{sec:tp-V}

We give some estimates on the weights appearing in the cofibrament replacement $\Omega$ of $\one$. Recall that $\Omega$ has an increasing sequence of sub-comodules of $\Omega$
$$ V =\Omega_0 \subset \Omega_1 \subset \Omega_2 \subset ... $$
such that $$ \Omega_{i+1}/\Omega_{i}\ \cong\ V \ \otimes\ (I\otimes I^*)^{\otimes i}.$$ We analyse now the filtration step $\Omega_{i+1}/\Omega_{i}$. For simplicity we specialize in this section to the $m=n$ (and assume $m=n \geq 2$) since many calculations in a maximal atypical block can be reduced to calculations in the principal block of $\mathcal{T}_{n|n}$. In order to understand the tensor product $I \otimes I^*$ better, we first analyze the $V \otimes V^*$ tensor product. We recall from \cite[Proposition 27.4]{Heidersdorf-Weissauer-tensor}:

\begin{prop} \cite[Theorem B.17]{Bump-Schilling} The space of matrices  $M_{nn}(k)$ is a $GL(n,k)\times GL(n,k)$-module in a natural way by left and right multiplication, hence also the Gra\ss mann algebra $\Lambda:=\Lambda^\bullet(M_{n}(k))$. As a representation
of $GL(n,k)\times GL(n,k)$ we have
$$ \Lambda^\bullet(M_{nn}(k)) \ \cong \ \bigoplus_{\rho}
  \rho^\vee \boxtimes \rho^*$$
  where $\rho=\rho_\lambda$ runs over all partitions in \[ P(n,n)=  \{ \lambda \in {\bf Z}^n\ \vert \ n \geq \lambda_1 \geq \lambda_2 \geq ... \geq \lambda_n \geq 0 \} \ .\]
\end{prop}  

We also note that there exist $2^n$ symmetric Young diagrams with $\lambda=(\lambda_1,...,\lambda_n)$ and  $\lambda_1\leq n$.

\medskip\noindent

We also recall from \cite{Heidersdorf-Weissauer-tensor} that $V(\one)$ has a decreasing filtration (the radical filtration) of $GL(n\vert n)$-subrepresentations 
with $n+1$ irreducible graded pieces $L_i$ such that $L_0=k$ is the maximal
irreducible quotient representation. The highest weights of the $L_i$ can be computed
from \cite[Theorem 5.2]{Brundan-Stroppel-1} to be the duals
$$   \lambda_{(i)}^\vee = (0,\cdots,0,-i,...,-i)   \quad  , \quad  \mbox{ for }\  i=0,...,n  \  $$
of the basic selftransposed weights $\lambda_i$ in $P(n,n)$. In particular the cosocle of $V^*$ consists of $Ber^{-n}$.

\begin{lem} \label{thm:V-tp} The tensor product $V \otimes V^*$ decomposes as \[ V \otimes V^* \cong P(Ber^{-n}) \oplus Q \] where $Q$ is of atypicality less than $n$.
\end{lem}

\begin{proof}
First note that $$ V \otimes V^* \in \calC_+ \otimes \calC_- \subset P_{\calC}$$
is projective in $\calT$ since $\calC_+ \cap \calC_- = Proj$ and $\calC_{\pm}$ are tensor ideals. From $V=F(\one)$ and $F(X)^* \cong Ind_{P_-}^{G}(X^* \otimes Ber^{-n})$ \cite[Proposition 2.1.1]{Germonie} we obtain as a $P$-module \[W^*=F(\one)^* = \Lambda^\bullet(G/P_-) \otimes Ber^{-n} \cong  \Lambda^\bullet(P_+) \otimes Ber^{-n} = Ind_{G_0}^{P_+}(Ber^{-n}).\] In other words \[F(U(V^*)) \cong F(Ind_{G_0}^{P_+}(Ber^{-n})) \cong F_0(Ber^{-n})\] where $F_0$ denotes induction from the group $G_0$. Hence we get
$$ V  \otimes V^* \cong F(\one) \otimes V^* \cong F(U(V^*)) = F_0(Ber^{-n}) \ .$$

\medskip\noindent
Any filtration of $U(V^*)$ by irreducible $P$-modules,
for instance the one
by the Grassmann degree, has as irreducible graded pieces the ${2n \choose n}$ modules $\rho_\alpha \boxtimes \rho_{\alpha^*}$
of the quotient group $G_0$ of $P$ for $\alpha\in P(n,n)$ (considered as representations of $P$).
It induces a filtration of $$V\otimes V^* = F(U(V^*))$$ by the ${2n \choose n}$ Kac-modules corresponding to the representations $\rho_\alpha \boxtimes \rho_{\alpha^*}$ of $G$ for
$\alpha\in P(n,n)$. 
Hence $V\otimes V^*$ inherits a Kac filtration whose maximal atypical graded pieces correspond
to the $2^n$ self transposed $\alpha = \alpha^*$ weights.
Since the degree of atypicality is a block invariant
we can decompose $V\otimes V^*$ in the form
$$ V \otimes V^*  = P \oplus Q $$
where $P\in {\calT}^n$ has atypicality $n$
and $Q\in {\calT}^{<n}$ is in the direct sum of blocks of atypicality $<n$.

\medskip\noindent
Therefore $P$ is projective with a filtration containing as graded pieces the $2^n$  
maximal atypical Kac-modules defined by the $2^n$ highest weights $\alpha=\alpha^* \in P(n,n)$ (each with multiplicity one). 
Obviously $P(Ber^{-n})$ must be a summand of $P$, since $Ber^{-n}$
is in the cosocle of $P$. The claim now  follows since $P(Ber^{-n})$ contains $2^n$ Kac-constituents. 

\end{proof}

\bigskip\noindent
Since $V\otimes V^* \in \calC_+$ and $V\otimes V^* \in \calC_-$,
by the vanishing of $Ext(\calC_+,\calC_-)$-groups and $$ \dim_k Hom_{\calT}(V(\lambda), V(\mu)^*) =\delta_{\lambda\mu}$$ we obtain from the fact that $P$ has a filtration by $2^n$ Kac modules (and similarly a filtration by $2^n$ anti Kac modules)
$$ \dim_k End_{\calT}(P) = \dim_k Hom_{{\calT}^{n}}(V\otimes V^*, V\otimes V^*) = 2^n \ . $$
For $V(\lambda)\in {\calT}^n$ notice $\dim_k   Hom_{{\calT}^{r}}(V\otimes V^*, V(\lambda)^*) = 0$ or $1$ depending on whether $\lambda$ is one of the self transposed weights $\alpha \in P(n,n)$. Hence for maximal atypical $\lambda$ we get
$$\dim_k   Hom_{\calT}(P, V(\lambda)^*) = 1 \ or \ 0  \ ,$$  
depending on whether $\lambda$ is one of the self transposed weights $\alpha \in P(n,n)$ or not.
On the other hand for $V(\lambda)\in {\calT}^n$ we have
$$Hom_{{\calT}^{n}}(P, V(\lambda)^*) =  Hom_{{\calT}^{n}}(V\otimes V^*, V(\lambda)^*)    =
Hom_{{\calT}}(V\otimes V^*, V(\lambda)^*) $$ $$ = 
Hom_{{\calT}}(F_0(Ber^{-n}), V(\lambda)^*)   = Hom_{G}(Ber^{-n}, V(\lambda)^*)   \ $$
by Frobenius reciprocity. We have deduced the following result:

\begin{cor} For maximal atypical irreducible weights $\lambda$ the following
assertions are equivalent
\begin{enumerate}
\item $\lambda = \lambda^*$, and hence $\lambda \in P(n,n)$ (and there are $2^n$ such $\lambda$).
\item $V(\lambda)$ is a Kac-constituent of the projective module $P=P(Ber^{-n})$  in the category $\calT$.
\item The restriction of $V(\lambda)^*$ to $G_0$ contains the representation $Ber^{-n}$.
\item The restriction of $V(\lambda)^*$ to $G_0$ contains the representation $Ber^{-n}$ with multiplicity one.
\item The restriction of $V(\lambda)$ to $G_0$ contains the representation $Ber^{-n}$.
\item $V(\lambda)$ contains $Ber^{-n}$ as a constituent in the category $\calT$. 
\end{enumerate}
\end{cor}

For the equivalence of 3. and 4. recall that $V(\lambda)$ and $V(\lambda)^*$ have the same simple constituents. For the last implication use that  property 2) is equivalent to
the last property by the BGG formula $[P(Ber^{-n}): V(\lambda)] = [V(\lambda): Ber^{-n}]$.

\subsection{Estimates for $I \otimes I^*$}

The $i$-th filtration step of $\Omega$ is given by $$ \Omega_{i+1}/\Omega_{i}\ \cong\ V \ \otimes\ (I\otimes I^*)^{\otimes i}.$$ Since $I \subset V(\one)$ we obtain $V(\one) \to I^*$ \[ \xymatrix@C=1.3em@R=1.3em{ & & 0 & & \\ 0 \ar[r] & I \otimes I^* \ar[r] & V \otimes I^* \ar[r] \ar[u] & I^* \ar[r] & 0 \\ & & V \otimes V^* \ar[u] & &  \\ & & V \otimes \one \cong V \ar[u] & & \\ & & 0 \ar[u] & & }\] where $V \otimes V^* \cong P(Ber^{-n})$ up to contributions of lower atypicality by lemma \ref{thm:V-tp}. The filtration of the module $P(Ber^{-n})$ via Kac modules $V(\one), \ldots V(Ber^n)$ has been described in section \ref{sec:tp-V}. The Kac filtration of $V \otimes I^*$ misses exactly the Kac module $V(\one)$ in $P(Ber^{-n})$; and the module $I \otimes I^*$ lacks exactly a copy of $I^*$ in comparison to $V \otimes I^*$.

\medskip
The highest weight in a Kac module is always in the top. The maximal atypical composition factors of $I \otimes I^*$ are those of $P(Ber^{-n})$ with those of one $V(\one)$ and $I$ missing. The constituents in $I$ are just the constituents in the tops of the Kac modules in $P(Ber^{-n})$. Therefore the largest weights in $I$ come from the Kac module $V(0,\ldots,0,-1)$: Its top $[0,\ldots,0,-1]$ is not in $I \otimes I^*$, but the constituents $[0,\ldots,0,-2]$ and $[0,\ldots,0,-1,-1]$ (see \cite[Theorem 5.2]{Brundan-Stroppel-1} for the description of the Loewy layers) in the next radical layer are.  The smallest weight in $P(Ber^{-n})$ and in $I \otimes I^*$ is the socle $Ber^{-2n}$ of $P(Ber^{-n})$ of degree $-2n^2$. 

\begin{cor} \label{cor:weight-estimate} Let $L(\lambda)$ be a composition factor of $I \otimes I^*$. Then \[ -2n^2 \leq deg(\lambda) \leq -2.\] In particular $V \ \otimes\ (I\otimes I^*)$ is a Kac object with weights between $-3n^2$ and $-2$. 
\end{cor}

\begin{example} We assumed in this section $m=n \geq 2$. The $GL(1|1)$-case was already treated in example \ref{ex:gl-1-1}. In this case $\Omega_{i+1}/\Omega_{i} \cong V(Ber^{-2i})$.

\end{example}



\section{Restriction and $DS$-cohomology}

\subsection{Restriction I}

For $m=m_1+m_2$ and $n=n_1+n_2$ consider the super subgroups 
$$ GL(m_1\vert n_1) \times GL(m_2\vert n_2) \hookrightarrow GL(m\vert n) $$
of the supergroup $GL(m\vert n)$
defined in terms of matrices by
$$ \begin{pmatrix} A' & 0  & B' & 0 \cr
0 & A'' & 0 & B'' \cr C' & 0 & D' & 0 \cr
0 & C'' & 0 & D'' \cr
\end{pmatrix}
$$
This defines a restriction functor $res$
$$ {\calT}_{m\vert n} \to {\calT}_{m_1\vert n_1} \times {\calT}_{m_2\vert n_2} \ . $$
This is an exact tensor functor. 
We easily see 
$$ res(Ber) \cong Ber \boxtimes Ber\ .$$
Similarly
 we obtain a functor for the corresponding ind-categories
$$ {\calC}_{m\vert n} \to {\calC}_{m_1\vert n_1} \times {\calC}_{m_2\vert n_2} \ ,$$
denoted $ res: \calC \to \calC' \times C'' $ for simplicity.

\medskip\noindent
These restriction functors are exact tensor functors. 
Since the restrictions of projective comodules are projective comodules, this induces
a monoidal triangulated functor between the triangulated tensor categories
defined by the stable categories
$$ \overline{\calC}\to \overline{\calC}' \times \overline{\calC}'' \ .$$

\medskip\noindent
Note that 
$$ P = \begin{pmatrix} A' & *  & B' & * \cr
* & A'' & * & B'' \cr 0 & 0 & D' & * \cr
0 & 0 & * & D'' \cr 
\end{pmatrix} \ \subset GL(m\vert n)
$$
contains
$$ P' \times P'' = \begin{pmatrix} A' & 0  & B' & 0 \cr
0 & A'' & 0 & B'' \cr 0 & 0 & D' & 0 \cr
0 & 0 & 0 & D'' \cr \end{pmatrix} 
\ \subset \ GL(m_1\vert n_1) \times GL(m_2\vert n_2) 
 $$
Therefore $U(X) \in \mathcal{I}_{\mathcal D}$ for $X\in {\calC}$
implies $  (U' \times U'')(res(X)) =
res(U(X)) \subset res({\mathcal{I}_{\mathcal D}}) \subset \mathcal{I}_{\mathcal D'} \times I_{\mathcal D''}$.
In other words
$$ res: {\calC}_- \to {\calC}'_- \times {\calC}''_- \ .$$
Thus we get an induced functor
between the triangulated quotient categories
$$ {\overline{\calE}} \to {\overline{\calE'}} \times {\overline{\calE''}} \ .$$

\subsection{Restriction II} \label{sec:restriction} Similarly we may embed $GL(m-k|n-k)$ as an outer  block matrix in $Gl(m\vert n)$ 
$$  \varphi_{n,m}: GL(m-k|n-k)  \hookrightarrow GL(m|n) \ . $$

Analogous to the preceeding discussion we obtain induced functors 
$$ res: Ho{\mathcal C}_{m|n} \to Ho{\mathcal C}_{m-k|n-k} $$
and similarly 
$$ res: Ho{\mathcal T}_{m|n} \to Ho{\mathcal T}_{m-k|n-k} \ .$$


\subsection{The functor $DS$} \label{sec:DS}

We recall from the article \cite{Heidersdorf-Weissauer-tensor} that we have a tensor functor $DS: \mathcal{T}_{m|n} \to \mathcal{T}_{m-1|n-1}$ attached to the choice of an odd element $x \in \mathfrak{g}_1$ satisfying $[x,x]=0$. Since $[x,x]=0$ we get $$2 \cdot \rho(x)^2 =[\rho(x),\rho(x)] =\rho([x,x]) =0 $$ for any algebraic representation
$(V,\rho)$ of $GL(m|n)$ in ${\calC}_{m|n}^{\infty}$. We fix now \[x = \begin{pmatrix} 0 & y \\ 0 & 0 \end{pmatrix} \text{ for } \ y = \begin{pmatrix} 0 & 0 & \ldots & 0 \\ 0 & 0 & \ldots & 0 \\ \ldots & & \ldots &  \\ 1 & 0  & 0 & 0 \\ \end{pmatrix} \] The cohomological tensor functor $DS$ is defined as \[ DS =  DS_{n,n-1}: \mathcal{C}_{m|n} \to \mathcal{C}_{m-1|n-1} \]
via  $DS_{n,n-1}(V,\rho)= V_x:=Kern(\rho(x))/Im(\rho(x))$.

\begin{lem} The functor $DS$ factorizes over the homotopy category and induces tensor functors \begin{align*} & DS: Ho{\calC}_{m|n} \to Ho{\calC}_{m-1|n-1} \\ & DS: Ho{\calT}_{m|n} \to Ho{\calT}_{m-1|n-1}.\end{align*}
\end{lem}

\begin{proof} It was proven in \cite[Theorem 4.1]{Heidersdorf-Weissauer-tensor} that the kernel of $DS: \mathcal{T}_{m|n} \to \mathcal{T}_{m-1|n-1}$ equals $\mathcal{T}_-$. A module $X$ is in $\mathcal{C}_-$ if and only if its restriction to $P(m|n)^+$ is projective and therefore injective. Since any injective module is a direct sum of injective finite dimensional modules, we obtain \[ DS(X) = DS(\bigoplus X_i) = \bigoplus DS(X_i) = 0\] where used that $x \in P(m|n)^+$. So $ker(DS) = \mathcal{C}_-$. A cofibrant replacement of $X \in \mathcal{C}$ defines an exact sequence \[ \xymatrix{ 0 \ar[r] & K^- \ar[r] & QX \ar[r]^q & X  \ar[r] & 0 } \] with $K^- \in \mathcal{C}_-$. We apply $DS$ to this sequence. Since $DS(\mathcal{C}_-) = 0$ and the functor $DS$ is weakly exact in the sense of \cite[Lemma 2.1]{Heidersdorf-Weissauer-tensor}, this implies that $DS(q): DS(QX) \to DS(X)$ is an isomorphism. Let $f \in [X,Y]$ be an arbitrary morphism and recall that $[X,Y] = Hom_{\mathcal{C}}(QX,Y)/\sim$. Therefore we obtain from $f:QX \to Y$ the commutative diagram \[ \xymatrix{ DS(QX) \ar[d] \ar[r]^{DS(q)} & DS(X) \ar[dl] \\ DS(Y) & } \] which shows that $DS: \mathcal{C}_{m|n} \to \mathcal{C}_{m-1|n-1}$ factorizes over $Ho\mathcal{C}_{m|n}$. We obtain all in all a commutative diagram \[ \xymatrix{  \mathcal{C}_{m|n} \ar[d]^{DS} \ar[r] &  Ho\mathcal{C}_{m|n} \ar@{.>}[d]^{DS} \ar[dl] \\  \mathcal{C}_{m-1|n-1} \ar[r] &  Ho\mathcal{C}_{m-1|n-1}  } \] which defines the induced functor $DS:  Ho\mathcal{C}_{m|n} \to  Ho\mathcal{C}_{m-1|n-1}$. This functor can be also simply restricted to the $Ho\mathcal{T}$-case.
 
\end{proof}

\begin{remark} This result extends to the more general functors $DS_{m-k|n-k}: \mathcal{T}_{m|n} \to \mathcal{T}_{m-k|n-k}$ considered in \cite{Heidersdorf-Weissauer-tensor}.
\end{remark}

\begin{remark} If we choose the other Frobenius pair (i.e. exchange $P^+$ with $P^-$) to define the homotopy category, we get a similar result for the $DS$ functor associated to the element $\sigma(x) \in \mathfrak{gl}(m|n)_{1}$.
\end{remark}




\section{Semisimple quotients}\label{sec:semisimple}

\subsection{Supertannakian categories} A $k$-linear tensor category $\mathcal T$ over a field $k$  (in the sense of \cite{Deligne-tensorielles}) is a small abelian $k$-linear symmetric closed monoidal rigid category with $End_{\mathcal T}(1)\cong k$. 
If $\mathcal T$ admits a super fibre functor over an extension feld of $k$, it is a 
supertannakian category and the finiteness condition (F) holds. 

\medskip\noindent
For a symmetric $k$-linear tensor category the symmetric group $S_m$ acts
on $X^{\otimes m}$ for any $X \in \mathcal{T}$. If $k$ is of characteristic zero, the irreducible representations $\sigma_\alpha$
of the group $S_m$ define the Schur functors $S=S^\alpha: X \to S^\alpha(X)=Hom_{S_m}(\sigma_\alpha,X^m)$ where $\alpha$ is a partition of $m$. Special 
cases are the symmetric or alternating $m$-th powers of $X$. By Deligne \cite{Deligne-tensorielles} a $k$-linear tensor category $\mathcal T$ over an algebraically closed field $k$ of characteristic 0 is supertannakian if and only if every object is annihilated by some Schur functor (Schur finiteness). Any supertannakian category over an algebraically closed field of characteristic 0 is tensor equivalent to the representation category $Rep(G,\epsilon)$ of an affine supergroup scheme over $k$. 

\subsection{Tensor generators} A tensor generator in the sense of \cite[0.1]{Deligne-tensorielles} is an object $Y\in \mathcal T$ such that any other object in  $\mathcal T$
is obtained by iterated application of the operations $\oplus$, $\otimes$, ${}^\vee$
and subquotients. A supertannakian category has a tensor generator if and only if $G$ is of finite type, i.e. an algebraic supergroup. We also say that it is an algebraic tensor category. By the usual comodules-representations correspondence, an algebraic tensor category
over an algebraically closed field of characteristic
zero is equivalent to the tensor category of finite dimensional graded $A$-comodules
of a supercommutative Hopf algebra $A$ finitely generated over $k$.

\medskip\noindent
If $\mathcal T\to \mathcal T'$ is a $k$-linear tensor functor, $\mathcal T$ is a $k$-linear tensor category
and $\mathcal T'$ is $k$-linear symmetric (closed) monoidal category, 
then the full image subcategory is a full $k$-linear symmetric (closed) monoidal rigid subcategory
of $\mathcal T'$. Schur finiteness will be inherited from $\mathcal T$, whereas properties (F) and (G) might not
be inherited to the full image subcategory of $\mathcal T$ in $\mathcal T'$.


\subsection{Ideals}

An ideal $\mathcal J$ of a $k$-category $\mathcal H$ is a collection of $k$-subvectorspaces
${\mathcal J}(X,Y)\subset Hom(X,Y)$ for all $X,Y\in \mathcal H$ such that $f{\mathcal J}(X,Y)g \in {\mathcal J}(X',Y')$ holds for all
$g\in Hom(X',X)$ and $f\in Hom(Y,Y')$. This defines the $k$-linear
quotient category $\mathcal H/J$
 with morphisms $Hom(X,Y)/{\mathcal J}(X,Y)$ and the same objects as in $\mathcal H$. 
If $\mathcal H$ is $k$-linear, so is $\mathcal H/J$.

\begin{example} The radical $rad(X,Y)$ is the ideal, which is defined by: $f\in rad(X,Y)$ if and only if
$id_X - gf$ is invertible for all $g\in Hom(Y,X)$.  
\end{example}


\subsection{Negligible morphisms} \label{sec:negligible} For a symmetric monoidal $k$-category $\mathcal H$ an ideal $\mathcal J$ is called a {\it monoidal ideal},
if it is stable under tensor products with $id_Z$ for all objects $Z\in\mathcal H$ \cite[section 6]{Andre-Kahn-nilpotence}. In this
case $\mathcal H/J$ inherits a symmetric monoidal structure and the 
quotient functor
$  \mathcal H \to H/J $
is a tensor $k$-functor, and
rigid objects in $\mathcal H$ map
to rigid objects in $\mathcal H/J$.

\begin{example} If $\mathcal H$ is a symmetric monoidal rigid $k$-linear category with $End(1)=k$, then the monoidal ideal ${\mathcal N}(X,Y)$ is defined by the morphisms $f\in Hom(X,Y)$ such that $tr(g\circ f)=0$ holds for all
$g\in Hom(Y,X)$. This is the ideal of negligible morphisms $\mathcal N= \mathcal N_{\mathcal{H}}$.
\end{example} 

\begin{lem} (\cite[7.1.4]{Andre-Kahn-nilpotence}) \enumerate \item The ideal $\mathcal N$ is the largest monoidal ideal  of $\mathcal H$
distinct from $\mathcal H$. \item If $\mathcal I$ is a monoidal ideal such that $\mathcal H/\mathcal I$ is semisimple, then $I = \mathcal N$.
\end{lem}


\subsection{Semisimplicity} By \cite[2.1.2]{Andre-Kahn-nilpotence} a small $k$-linear category $\mathcal H$ is semisimple  if and only if
\begin{itemize}
\item The radical ideal vanishes $rad({\mathcal H})=0$,
\item $Hom(X,X)$ is a semi-simple  Artin ring for all objects $X\in \mathcal H$.
\end{itemize}
In this situation, if $\mathcal H$ is semisimple pseudo-abelian and $k$-linear, then $\mathcal H$
is an abelian category \cite{Andre-Kahn-nilpotence}.

\subsection{The quotient by negligible morphisms} \label{sec:negligible-quotient}

In the $GL(m|n)$-case the vanishing and finiteness theorems of section \ref{sec:vanishing} hold. Then $\mathcal H = Ho \mathcal T$ is a $k$-linear rigid symmetric monoidal category, and we have shown
$End_{\mathcal H}(1)=k$. Hence the ideal of negligible morphisms
$\mathcal N$ is defined.

\begin{thm} \label{thm:semisimple} For $GL(m|n)$ over an algebraically closed field $k$ of characteristic 0 the following holds for $Ho {\mathcal T}$:
\begin{enumerate}
\item We have the relation $\mathcal N \supset R$ and $\mathcal H/ \mathcal N$ is semisimple.
\item The quotient $Ho \mathcal T / \mathcal N$ is the semisimple representation category of an affine supergroup scheme.
\end{enumerate}
\end{thm}

For the proof we use the following criterion due to Andr\'e and Kahn \cite[Th\'eor\`eme 1]{Andre-Kahn-erratum}:

\begin{prop} \label{ak-semisimplicity-criterion} Let $\mathcal A$ be $k$-linear symmetric monoidal category, rigid, with $End_{\mathcal A}(\one) = k$ and $char(k) = 0$.  Suppose there is an extension $L/k$ and a $k$-linear tensor functor $H: \mathcal A \to \mathcal V$ into an abelian $L$-linear symmetric monoidal rigid category, in which the $Hom$-spaces are finite-dimensional and the trace of a nilpotent endomorphism vanishes. Then $\mathcal R \subset \mathcal N$ and therefore $\mathcal A/ \mathcal N$ is semisimple.
\end{prop}

\begin{proof} We prove in section \ref{sec:gl-m-1-semisimple} by direct computations that $Ho \mathcal T_{m|1} / \mathcal N$ is a supertannakian category for the Frobenius pair ($GL(m|1), P(m|1)^+$) where $P(m|1)^+$ denotes the upper parabolic in $GL(m|1)$. For $G = GL(m|n)$ we then obtain an induced restriction functor $$ res: Ho {\mathcal T}_{m|n} \to Ho {\mathcal T}_{m-n+1|1} \ $$ as in section \ref{sec:restriction}. The functor \[  Ho {\mathcal T}_{m|n} \to Ho {\mathcal T}_{m-n+1|1} \to Ho {\mathcal T}_{m-n+1|1}/ \mathcal{N} \ \] satisfies the criterion of proposition \ref{ak-semisimplicity-criterion}. Therefore $\mathcal R \subset \mathcal N$ and $Ho \mathcal T/ \mathcal N$ is semisimple. This implies that $Ho \mathcal T/ \mathcal N$ is abelian. Since Schur finiteness is inherited via tensor functors and every object in $\mathcal T$ is Schur finite, these quotients are supertannakian categories and we can apply Deligne's theorem. 
\end{proof}




\section{The case $GL(m|1)$: Morphisms and cofibrant replacements}

Let $G= GL(m|1)$. As before we choose $A\to B$ to correspond to the inclusion $P=G_0\oplus G_{+1} \hookrightarrow
G$. For the corresponding model structure on the category $Ind(\mathcal{T}_{m|1})$
we obtain the homotopy category $Ho\mathcal C$.
We now compute the morphisms in $Ho\mathcal C$ between the simple
objects of $\mathcal T$.

\subsection{Representations of $GL(m|1)$}

The Kac modules in the category $\mathcal T$  are either irreducible projective ($\lambda$ typical) or have length two ($\lambda$ atypical) with two
atypical composition factors. The category $\mathcal T$ decomposes into blocks ${\mathcal T}^\Lambda$.
The $Ext$-quiver of an atypical block has been described in \cite{Germonie}. The irreducible representations in a given block can be parametrized by the integers, and we denote representatives of the simple objects of such a block ${\mathcal T}^\Lambda$ by
$L(i)$ for $i\in \Z$ for some arbitrarily chosen simple object $L(0)=L(\lambda)$ of this block.
The Kac module $V=V(0)$ is an extension with simple cosocle $L(0)$ and 
simple socle $L(-1)$.

\medskip\noindent
By the classification of the indecomposable objects in ${\mathcal T}^\Lambda$ the non-projective indecomposable modules correspond to intervals on the numberline. More precisely for every interval $[a,b]$ we have two indecomposable modules with composition factors $L(a),\ldots, L(b)$. The indecomposable module with socle $L(a), L(a+2),\ldots$ and cosocle $L(a+1),L(a+3),\ldots$ is denoted $R[a,\ldots,b]$. Its twisted dual is $B[a,\ldots,b] = R[a,\ldots,b]^*$ with cosocle $L(a), L(a+2),\ldots$ and socle $L(a+1),L(a+3),\ldots$. The Kac- and anti Kac-modules are then given by \[ V(a) = R[a,a+1], \ V(a)^* = B[a,a+1].\] If $R[a,\ldots,b]$ has even length, it has a filtration by the Kac modules $V(a+1),\ldots, V(b)$ and $B[a,\ldots,b]$ has a filtration by the anti Kac-modules $V(a)^*,\ldots, V(b-1)^*$. 

\begin{remark} This notation differs from the one used in \cite{Heidersdorf-semisimple}. There we use the notation $I^+[a,b]$ for the unique indecomposable module with $L(b) \in top(I^+[a,b])$ and $I^-[a,b] = I^+[a,b]^*$ for its twisted dual.
\end{remark}

\subsection{Morphisms} \label{sec:gl-m-1-morphisms}

\begin{lem} We have $[L(i),L(j)]=0$ unless $i\geq j$ and $i\equiv j$ modulo 2,
where $[L(i),L(j)]\cong k$.
\end{lem}

\begin{proof} We apply the functor $Hom_{Ho\mathcal C}(-,L(u))$
to  the exact sequence  $$ 0 \to L(-1) \to V \to L(0) \to 0 \ $$
defined by the Kac object $V$.
Since $V \in \mathcal C_+$ is cofibrant, we obtain 
$$  [V, L(u)] = k  $$
for $u=0$ and $0$ for $u\neq 0$ since
$Hom_{\mathcal C}(V,L(u))=k$ for $u=0$ and zero otherwise.
Furthermore for $u=0$ this morphism can not be factorized over a projective
object, since $K$ is clean and $L(u)$ is simple. Now put $u=0$ (for simplicity).
The long exact homotopy sequence
then implies $$[L(-1)[i],L(0)] \cong [L(0)[i-1],L(0)]$$ for all $i\leq -1$
and all $i\geq 2$. Furthermore for $i=0,1$ we have an exact sequence
$$ 0 \to [L(0)[-1],L(0)] \to [L(-1),L(0)] \to k  \to [L(0),L(0)] \to [L(-1)[1],L(0)] \to 0 $$   
We already know
$$  L(n+1) = L(n)[-1],$$
since $L(n) \cong V(n+1)^* /L(n+1)$ for the anti Kac module $V(n+1)^*$, which becomes zero in
$Ho\mathcal T$.  
Hence
$$[L(-1-i),L(0)] \cong [L(1-i),L(0)]$$ for all $i\leq -1$
and all $i\geq 2$. Furthermore for $i=0,1$ we have an exact sequence
$$ 0 \to [L(1),L(0)] \to [L(-1),L(0)] \to k  \to [L(0),L(0)] \to [L(-2),L(0)] \to 0 $$   
Since $[L(-2),L(0)] = [L(-1),L(0)]=0$ by theorem \ref{thm:main}, this implies
$$[L(i),L(0)]=0$$ for all odd $i$ and all even $i\leq -2$, and $$[L(i),L(0)]\cong k$$ for all even $i\geq 0$.
\end{proof}
%

Hence the triangulated category $\mathcal H$
decomposes into blocks ${\mathcal H}^\Lambda$, and each block decomposes into two subblocks
$$   {\mathcal H}^\Lambda  \ = \ {\mathcal H}_{ev}^\Lambda \ \oplus \ {\mathcal H}_{odd}^\Lambda \ $$
such that ${\mathcal H}_{odd}^\Lambda = {\mathcal H}_{ev}^\Lambda[1]$.
The images of the simple objects in the block ${\mathcal T}^\Lambda$ are identified with the integers.
Those in ${\mathcal H}_{ev}^\Lambda$ are identified with the even integers, and
the morphisms in ${\mathcal H}_{ev}^\Lambda$ arise $Hom(2j,2i)= k\cdot f_{ij}$ for a nonzero morphism
$f_{ij}$ if $j\geq i$ and $Hom(2j,2i)=0$ otherwise, such that
$$   f_{ij} \circ f_{jk} =   f_{ik} \quad , \quad i\leq j \leq k \ .$$

\begin{lem} $L(u)$ and $L(v)$ for $u\neq v$ atypical are isomorphic in
$Ho{\mathcal C}$ if and only if $u=v$.
\end{lem}
 
\begin{proof}  Assume $v\neq u$ and assume $L(u) \simeq L(v)$ in $Ho \mathcal C$. Then we may assume that the weight of $u$ is smaller than the weight of $v$ without restriction of generality. Any such isomorphism is represented by a homotopy class of a morphism $f$ in $Hom(QL(u),L(v))$ (\cite[Theorem 1.2.10ii]{Hovey})
$$  f: QL(u) \to L(v) \ .$$
Since all weights in $\Omega=QL(u)$ are $\leq u$ and hence $<v$, this implies
$f=0$.
Now $f$ becomes an isomorphism in $Ho{\mathcal C}$ if and only if
$f$ is a weak equivalence (\cite[Theorem 1.2.10iv]{Hovey}). If $f=0$ is a weak equivalence, we can 
factor $f=¸\psi\circ \varphi$ into a split monomorphism $\varphi:QL(u)\to Z$ with projective kernel and a
surjective morphism $\psi:Z\to L(v)$ with kernel in $\mathcal C_-$. But then $\psi =0$, since
$f=0$. This implies $L(v)=0$. Contradiction.
\end{proof}

\begin{lem} \begin{enumerate}
\item If the length of $B = B[a,\ldots,b]$ is even, $B$ becomes isomorphic to zero in $Ho\mathcal{C}$.
\item If the length of $B = B[a,\ldots,b]$ is odd, $B \simeq L(b)$ in $Ho\mathcal{C}$.
\item If the length of $R = R[a,\ldots,b]$ is even, $B$ is indecomposable in  $Ho\mathcal{C}$.
\item If the length of $R = R[a,\ldots,b]$ is odd, $R \simeq L(a)$ in  $Ho\mathcal{C}$.
\end{enumerate}
\end{lem}

\begin{proof} If the length of $B$ is even, it is in $\mathcal{T}_-$. If the length is odd, the quotient morphism $B[a,\ldots,b] \to L(b)$ has kernel in $\mathcal{T}_-$. If the length of $R$ is even, it is in $\mathcal{T}_+$. Since $R$ is cofibrant $[R,R] = Hom_{\overline{\mathcal{C}}}(R,R)$. Since the latter is one-dimensional, $R$ is indecomposable. If the length of $R$ is odd, then the morphism $L(a) \to R[a,\ldots,b]$ has cokernel in $\mathcal{T}_-$ and therefore $L(a) \simeq R$ in $Ho\mathcal{C}$.
\end{proof}

Any object in ${\mathcal C}$ is a direct sum of indecomposable modules.
Those in $\mathcal C_-$ become isomorphic to zero in $ Ho{\mathcal C}$. 
Those in $\mathcal C_+$ stay indecomposable unless they are projective. All the remaining
ones become isomorphic in $Ho{\mathcal C}$ to 
the image of some simple module $L(u)$.


\subsection{Cofibrant replacements} \label{gl-1-1-cofib}

In this section we explicitly determine the minimal models of the simple objects.

\medskip\noindent
{\it Cofibrant replacements}. Projective simple objects $X$ in $\mathcal C$ are cofibrant.
Atypical simple objects are not cofibrant. Objects in $\mathcal C_+$ are cofibrant.
For $X\in \mathcal C_-$ a projective resolution $q:P\to X\to 0$ defines a cofibrant replacement 
$QX \cong P$ of $X$.
We now construct an explicit  cofibrant
replacement $q:QX \to X$ for atypical simple modules $X=L(u)$ as a sequential inductive limit
 $$ \Omega = co\lim_i \Omega_i \ $$
of subobjects $$\Omega_i= R[u-1-2i,..,u] $$ with the obvious inclusion morphisms 
$\Omega_i \hookrightarrow \Omega_{i+1}$ (see \cite{Germonie} for further details).
$\Omega_{i+1}/\Omega_i \cong R[u-1-2i,u-2i]$ is in $F(\mathcal D)\subset \mathcal C_+$.  This shows $\Omega_i\in \mathcal C_+$, since $\mathcal C_+$ is closed under extensions.
$\Omega$ is a union
of the $\Omega_i$. Since $\mathcal C_+$ is closed under monomorphic sequential colimits
we obtain

\begin{lem} $\Omega$ is cofibrant.
\end{lem}

$$  \xymatrix@C=1em{ \ar@{.}[dr]&  &  L(u-4)  \ar@{-}[dr] & & L(u-2) \ar@{-}[dr]&    & L(u)  \cr
&  L(u-5) \ar@{-}[ur] & &L(u-3) \ar@{-}[ur] &  &  L(u-1) \ar@{-}[ur] &  \cr
}  \ .$$

There exists an exact sequence
$$ 0 \to  R \to \Omega \to L(u) \to 0 \ ,$$
where $R$ is isomorphic to
the cohomology of the complex
\[\bigoplus_{i=1}^\infty L(u-2i-1) \to  \bigoplus_{i=1}^\infty R[u-2i-1,u-2i]\oplus R[u-2i,u-2i+1]^* \to  \bigoplus_{i=1}^\infty L(u-2i) \ .\]
Notice that the simple module $L(u-2i)$ is a quotient of $R[u-2i-1,u-2i]$ and $R[u-2i,u-2i+1]^*$.
Similarly the simple module $L(u-2i-1)$ is a submodule of $R[u-2i-1,u-2i]$ and $R[u-2i,u-2i-1]^*$.

\begin{lem} \label{thm:splitting} The restriction $UR[u-1,u] = UL(u-1)\oplus UL(u) $ splits in $\mathcal D$.
\end{lem}

\begin{proof} The projective $P=P[u-2,u-1,u-1,u]$ contains $R[u-1,u]$
via the standard embedding. $P$ has a filtration by two anti Kac modules $V$ and $V'$,
which under the restriction functor $U$ become indecomposable projectives in $\mathcal D$.
Therefore the anti Kac filtration splits in $\mathcal D$, hence $$ UR[u-1,u] \ = \  (UR[u-1,u] \cap UV) \oplus (UR[u-1,u] \cap UV') \ .$$
On the other hand the anti Kac filtration on $P$ cuts out the standard filtration
on $R[u-1,u] \subset P$ with the graded pieces $L(u-1)$ and $L(u)$.
\end{proof}

\begin{lem} $R$ is in $\mathcal C_-$. Hence $\Omega$ is a cofibrant replacement $QX$ of 
the simple module $X=L(u)$.
\end{lem}

\begin{proof} The exact sequences $$ 0 \to L(u-1) \to R[u-1,u] \to L(u) \to 0 $$
split in $\mathcal D$ after applying the restriction functor $U$ by lemma \ref{thm:splitting}. This implies
$$UR \ \cong\ \bigoplus_{i=1}^\infty UR[u-2i,u-2i+1]^*) \ .$$ All  $UR[u-2,u-1]^* $ are injective, 
hence $UR$ is injective in $\mathcal D$. 
\end{proof}

\begin{example} \label{ex:minimal} Consider the indecomposable module $X=R[a,a+1,a+2]$ and the cofibrant
object $Q'X = \Omega(L(a)) \oplus P[a,a+1,a+1,a+2]$.  There is a morphism $q':Q'X \to X$
with kernel $K \cong kern(q: \Omega(L(a+2)\to L(a+2)) \in \mathcal C_-$, hence $Q'X$ is a cofibrant replacement of $X$. 
This shows that both arrows, the natural inclusion and $q'$, are in $\mathcal W$, and therefore the composite morphism 
$$    \Omega(L(a)) \hookrightarrow Q'X \to  X  $$
is in $\mathcal W$. This again implies
$$  \Omega(L(a)) \cong X $$
in $Ho\mathcal C$. Since $q'(P[a,a+1,a+1,a+2])\neq 0$, there does not exist a minimal model for $X$
by the last lemma.
\end{example}

\begin{example} In general $End_{\mathcal C}(X) \to [X,X]$ is not surjective.
Put $X=L(0)\oplus L(-2)$, then $End_{\mathcal C}(X) = k^2$, but 
$[X,X]$ also contains a nilpotent radical generated by the morphism $f_{-2,0}$.
\end{example}


\section{The case $GL(m|1)$: Semisimple quotients}

\subsection{Semisimplicity of $Ho \mathcal T / \mathcal N$} 

Let us write $Ho{\mathcal T}^{ss} = Ho \mathcal T / \mathcal N$.
The indecomposable objects in $\mathcal T_+$ become isomorphic to zero in $ Ho{\mathcal T}$. Those in $\mathcal T_-$ become zero in $Ho{\mathcal T}^{ss}$. All the remaining
ones become isomorphic in $Ho{\mathcal T}$, hence in $Ho{\mathcal T}^{ss}$, to 
the image of some simple module $L(u)$ by the previous section. Hence $Ho{\mathcal T}^{ss}$ is a 
semisimple category.

\begin{lem} $Ho{\mathcal T}^{ss}$ is a semisimple $k$-linear rigid closed monoidal tensor category with $End(\one)=k$. Its simple objects are parameterized by the atypical weights. It is of the form $Rep(G',\epsilon)$ for some supergroup $G$.
\end{lem}


\subsection{The quotient $Rep(GL(m|1))/\mathcal N$ and consequences} \label{sec:gl-m-1-semisimple}

We want to describe $Rep(G',\epsilon)$ explicitely. Recall from \cite[Example 0.4 (ii)]{Deligne-tensorielles} that if $G$ is an affine supergroup scheme and $\mu_2$ acts on $G$ by the parity automorphism, $Rep( \mu_2 \ltimes G, \epsilon = (-1,e))$ is the category of super representations of $G$.

\medskip\noindent
We recall now from \cite{Heidersdorf-mixed-tensors} results about the tensor product decomposition of simple $GL(m|1)$-modules. Any irreducible module $L(\lambda)$ can be written uniquely in the form $L(\tilde{\lambda}) \otimes Ber^{s_{\lambda}}$ where $L(\tilde{\lambda})$ is a direct summand in a space of mixed tensors $V^{\otimes r} \otimes (V^{\vee})^{\otimes s}$ for some $r,s$, and $s_{\lambda}$ is an explicit shift factor that can be read off from the cup diagram of $\lambda$. The irreducible mixed tensors generate a tensor category isomorphic to $Rep(GL(m-1))$ in the quotient category $Rep(GL(m|1))/\mathcal N$ and $L(\tilde{\lambda})$ corresponds to an irreducible $GL(m-1)$-representation $L(wt(\tilde{\lambda}))$. Therefore the Tannaka group generated by the irreducible $GL(m|1)$-modules in $Rep(GL(m|1))/\mathcal N$ is $GL(m|1) \times GL(1)$, so that the Tannaka category is equivalent to the super representations of $GL(m-1) \times GL(1)$, i.e. \begin{align*} (Rep(GL(m|1))/\mathcal N)^{irr} & \simeq Rep(\Z/2\Z \ltimes (GL(m-1) \times GL(1) ), \epsilon=(-1,e)),\\ \ L(\tilde{\lambda}) \otimes Ber^{s_{\lambda}} & \mapsto L(wt(\tilde{\lambda})) \times det^{s_{\lambda}}.\end{align*}

\begin{prop} \label{thm:gl-m-1-group} The quotient $Ho{\mathcal T}^{ss} \cong Rep(G',\epsilon)$ is equivalent to \[ Ho{\mathcal T}^{ss} \simeq Rep(\Z/2\Z \ltimes (GL(m-1) \times GL(1) ), \epsilon=(-1,e)).\]
\end{prop} 

\begin{proof} The irreducible representations in $Ho{\mathcal T}^{ss}$ and $(Rep(GL(m|1))/\mathcal N)^{irr}$ are both parametrized by the atypical weights. They obey the same tensor product decomposition and the categories are semisimple. Therefore one can write down an isomorphism between the Grothendieck semirings of these two categories which lifts to an isomorphism of groups. For further details we refer to the identical proof in \cite[Theorem 5.12]{Heidersdorf-semisimple}.
\end{proof}


\subsection{Isogenies: Semisimplicity} \label{sec:isogeny-semisimple}

Recall from the last section that the triangulated category $\mathcal H$
decomposes into blocks ${\mathcal H}^\Lambda$ where the images of the simple objects in the block ${\mathcal T}^\Lambda$ are identified with the integers.

\begin{lem} \label{lem:isogeny}
All $f_{ik}$ from section \ref{sec:gl-m-1-morphisms} are isogenies, i.e.
contained in $\Sigma$. 
\end{lem}

\begin{proof} For the indecomposable modules $B[2i,...,2j]$
the quotient morphism
$s: B[2i,...,2j] \to L(2j)$ is a weak equivalence in $\mathcal W$. On the other hand, the quotient morphism
$f: B[2i,...,2j] \to L(2i)$ in $\mathcal T$ has kernel in ${\mathcal T}_+$, hence induces an isogeny in $\mathcal H$.
The composition $f\circ s^{-1}: L(2j) \to L(2i)$ is well defined in $\mathcal H$ and is the morphism
$f_{ij}$ mentioned above up to a constant.
\end{proof} 

All objects in $\mathcal{T}_{\pm}$ become isomorphic to zero in $\mathcal{H}[\Sigma^{-1}]$. By lemma \ref{lem:isogeny} the image of a block in $\mathcal{H}[\Sigma^{-1}]$ has up to isomorphism two indecomposable elements (note $Hom_{{\mathcal H}[\Sigma^{-1}]}(X,X) = k \cdot id \cong k$
for simple atypical objects $X$), namely one representative in $\mathcal{H}_{ev}^\Lambda[\Sigma^{-1}]$ and one in ${\mathcal H}_{odd}^\Lambda[\Sigma^{-1}]$.

\begin{lem} $\mathcal{H}[\Sigma^{-1}]$ is a semisimple abelian category.
\end{lem}

\begin{proof} The pair \[(\mathcal{H}_{ev}[\Sigma^{-1}], \mathcal{H}_{odd}[\Sigma^{-1}]) \] for \[ \mathcal{H}_{ev}[\Sigma^{-1}] = \bigoplus_{\Lambda} \mathcal{H}_{ev}^{\Lambda}[\Sigma^{-1}], \ \ \mathcal{H}_{odd}[\Sigma^{-1}] = \bigoplus_{\Lambda} \mathcal{H}_{odd}^{\Lambda}[\Sigma^{-1}] \] is a torsion pair on $\mathcal{H}[\Sigma^{-1}]$ in the sense of \cite{Abe-Nakaoka} (more precisely $\mathcal{H}_{odd}[\Sigma^{-1}]$ is a cluster tilting subcategory \cite{Keller-Reiten}). The heart of a torsion pair is an abelian category. Here the heart equals the quotient $\mathcal{H}[\Sigma^{-1}]/\mathcal{H}_{odd}[\Sigma^{-1}] \simeq \mathcal{H}_{ev}[\Sigma^{-1}]$. The latter is therefore abelian and hence also semisimple.
\end{proof}

As for the quotient by the negligible morphisms this implies

\begin{cor} ${\mathcal H}[\Sigma^{-1}] \ \cong \ Rep(\mu,\tilde G) $ for
of a reductive algebraic super groupscheme 
$\tilde G$ over $k$.
\end{cor}


\subsection{Isogenies: The reductive group $\tilde{G}$}

The representation \[ \Pi = Ber^{-1} \otimes \Lambda^{m-1}(V) = L(0,0,\ldots,0,-1|1) \] is the socle of the Kac module $V(\one)$. Therefore it sits in the exact sequence \[ \xymatrix{ 0 \ar[r] & \one \ar[r] & V(\one)^* \ar[r] & \Pi \ar[r] & 0}.\] Since $V(\one)^*$ is zero in $Ho \mathcal T$ this implies the isomorphism \[ \Pi \simeq \one[1] \] in $Ho \mathcal T$. The Kac module $V(Ber)$ has constituents $Ber$ and $\Lambda^{m-1}(V)$. Since $V(Ber)$ and $V(Ber)^*$ are both trivial in ${\mathcal H}[\Sigma^{-1}]$, we obtain the isomorphism $Ber \simeq \Lambda^{m-1}(V)[1]$ in ${\mathcal H}[\Sigma^{-1}]$. Together with $Ber \simeq \Lambda^{m-1}(V)[1]$ in $Ho \mathcal T$ this implies $\Pi^2 \simeq \one$ in ${\mathcal H}[\Sigma^{-1}]$ for $\Pi \simeq \one[1]$.

\medskip\noindent
Since $\mathcal H \to {\mathcal H}[\Sigma^{-1}]$ is a full tensor functor into a semisimple tensor category, it factorizes by \cite{Heidersdorf-semisimple} \[ \xymatrix{ \mathcal H \ar[rr] \ar[dr] & & {\mathcal H}[\Sigma^{-1}]. \\ & \mathcal H / \mathcal N  \ar@{.>}[ur]^{\exists \phi} } \] The irreducible representations in ${\mathcal H}[\Sigma^{-1}]$ are now parametrized by the irreducible representations of $GL(m-1) \times \Z/2\Z$: Indeed the atypical blocks are in bijection with the irreducible representations of $GL(m-1)$ (every block contains exactly one irreducible mixed tensor \cite[Lemma 8.1]{Heidersdorf-mixed-tensors}) and every block gives two irreducible objects in ${\mathcal H}[\Sigma^{-1}]$, represented by the mixed tensor $L(\tilde{\lambda})$ in the given block and $L(\tilde{\lambda}) \otimes \Pi$. 

\begin{cor} ${\mathcal H}[\Sigma^{-1}]$ is tensor equivalent to the super representations of $GL(m-1) \times \Z/2\Z$.
\end{cor}







\addtocontents{toc}{\protect\setcounter{tocdepth}{-1}}

\end{document}